\documentclass[10pt,a4paper]{article}
\usepackage[utf8]{inputenc}
\usepackage{amsmath}
\usepackage{amsfonts}
\usepackage{amssymb}
\usepackage{amsthm}
\usepackage{enumerate}
\usepackage[margin=1.3in]{geometry}

\newtheorem{theorem}{Theorem}[section]
\newtheorem{proposition}[theorem]{Proposition}
\newtheorem{definition}[theorem]{Definition}
\newtheorem{lemma}[theorem]{Lemma}
\newtheorem{remark}[theorem]{Remark}
\newtheorem{corollary}[theorem]{Corollary}

\numberwithin{equation}{section} %\requirepackage{amsmath}

\DeclareMathOperator*{\esssup}{ess\,sup}
\DeclareMathOperator*{\essinf}{ess\,inf}

\DeclareMathOperator*{\R}{\mathbb{R}}

\def \om {\omega}

\def \0TR{[0,T]\times\mathbb{R}^k}
\def \otr {[0,T)\times\mathbb{R}^k}

\def \tx{t,x}
\def \S2{\mathcal{S}^2}
\def \H2d{\mathcal{H}^{2,d}}
\def \A2{\mathcal{A}^2}
\def \lG{l\in \Gamma}
\def \ms{\vspace*{3 ex}}
\def \iplus1{ i,i+1}
\def \ig{i\in\Gamma}
\def \p {\mathbb{P}}
\def \mx{\mbox}
\def \ind{{\bf 1}}
\def \rp {\mathbb{R}^p}
\def \G {\Gamma}
\def \sd {\mathcal{S}^2}
\def \hd {\mathcal{H}^{2,d}_{loc}}
\def \ad {\mathcal{A}_{loc}}
 \def \ox {0,x}
 \def \s {\sigma}
 \def \t {\tau}
 \def \G {\Gamma}
 \def \lb {\label}
\allowdisplaybreaks[4]
\def \nd {\noindent}
\def \tuv{\Theta(u,v)}
\def \ca {\mathcal{A}}
\def \cb {\mathcal{B}}
\def \cai {\mathcal{A}^{(1)}}
\def \cbi {\mathcal{B}^{(1)}}
\def \mba  {\mathbb{A}}
\def \mbb {\mathbb{B}}
\def \mbaun  {\mathbb{A}^{(1)}}

\def\txsp{(t,x)\in [0,T]\times \mathbb{R}^k}
\def \fr{\forall}
\def \nn{\nonumber}
\def \rwi{\rightarrow \infty}
\def \E{\mathbb{E}}
\def \qq{\qquad}

\def \hdd {\mathcal{H}^{2,d}} 
\def \rw {\rightarrow}
\def \frig{\fr \ig}
\def \frst{\fr s\in [t,T]}
\def \xtx {X^{\tx}}
\def \tT {[0,T]}
\def \stt {s\in [t,T]}
\def \g {\gamma}

\def \cH {\mathcal{H}}
\def \cS {\mathcal{S}}

\def \yphi {Y^{\phi,i}}
\def \zphi {Z^{\phi,i}}
\def \kphip {K^{\phi,i,+}}
\def \kphim {K^{\phi,i,-}}
\def \dis {\displaystyle}
\def \ypsi {Y^{\psi,i}}
\def \zpsi {Z^{\psi,i}}

\def \cBone {\mathcal{B}^{(1)}}
\def \cAone {\mathcal{A}^{(1)}}
\def \jphi {J^{\phi}(\Theta(u,v))_s}
\def \jpsi {J^{\psi}(\Theta(u,v))_s}
\def \jphir {J^{\phi}(\Theta(u,v))_r}
\def \jpsir {J^{\psi}(\Theta(u,v))_r}
\def \tuv {\theta(u,v)}
\def \cF {\mathcal{F}}
\def \kg {k\in \G}
\def \d {\delta}
\def \lb{\label}
\def \hz { (Z^i)_{\ig}\in \mathcal{H}^{2,d}}
\def \z {(Z^i)_{\ig}}
\begin{document}
\begin{titlepage}
\title{Zero-sum Switching Game, Systems of Reflected Backward SDEs and Parabolic PDEs with bilateral interconnected obstacles}
\author{Said {\sc Hamadène}\footnote{e-mail:  {hamadene@univ-lemans.fr}}, Tingshu {\sc Mu}\footnote{e-mail:  {tingshu.mu.etu@univ-lemans.fr}}\\
Le Mans University\\ Laboratoire  Manceau de Mathématiques \\Avenue Olivier Messaen, 72085 Le Mans Cedex 9, France \\
}

\end{titlepage}
\maketitle

\begin{abstract}
In this paper we study a zero-sum switching game and its verification theorems expressed in terms of either a system of Reflected Backward Stochastic Differential Equations (RBSDEs in short) with bilateral interconnected obstacles or a system of parabolic partial differential equations (PDEs in short) with bilateral interconnected obstacles as well. We show that each one of the systems has a unique solution. Then we show that the game has a value. 
\end{abstract}

\nd \textbf{Keywords}: Zero-sum switching game; System of PDEs;  HJB equations; Bilateral obstacles; Viscosity solution; Reflected Backward SDEs; Perron's method.
\ms

\nd {\bf AMS Classification subjects}: 49L25; 91A15; 35F21.
%*************************************************************************
%*                          Introduction                            **%
%*************************************************************************
\section{Introduction}
This paper is related to zero-sum switching games, systems of reflected backward differential equations (RBSDEs) with bilateral interconnected obstacles and systems of variational inequalities of min-max type with interconnected obstacles, namely the Hamilton-Jacobi-Bellman (HJB for short) system associated with the game. \\

\nd First let us describe the zero-sum switching game which we will consider in this paper.  Let $\G$ be the set $\{1,...,p\}$. Assume we have a system which has $p$ working modes indexed by $\G$. This system can be switched from one working mode to another one, e.g. due to economic, financial, ecological puposes, etc, by two players or decision makers $C_1$ and $C_2$. The main feature of the switching actions is that when the system is in mode $\ig$, and one of the players decides to switch it, then it is switched to mode $i+1$ (hereafter $i+1$ is $1$ if $i=p$). It means that the decision makers do not have their proper modes to which they can switch the system when they decide to switch (see e.g. \cite{hmm} for more details on this model). Therefore a switching strategy for the players are sequences of stopping times $u=(\s_n)_{n\ge 0}$ for $C_1$ and 
$v=(\t_n)_{n\ge 0}$ for $C_2$ such that $\s_n\le \s_{n+1}$ and $\t_n\le \t_{n+1}$ for any $n\ge 0$. On the other hand, the switching actions are not free and generate expenditures for the players. Loosely speaking at time $t\le T$, they amount to $A^u_t$ (resp. $B^v_t$) given by:
$$
A^u_t=\sum_{\s_n\le t}\underline{g}_{\zeta_n,\zeta_n+1}(\s_n) \,\,(\mbox{resp. }B^v_t=\sum_{\t_n\le t}\bar{g}_{\theta_n,\theta_n+1}(\t_n)).
$$
The process $\underline{g}_{i,i+1}(s)$ (resp. $\bar {g}_{i,i+1}(s)$) is the switching cost payed by $C_1$ (resp. $C_2$) is she makes the decision to switch the system from mode $i$ to mode $i+1$ at time $s$ while $\zeta_n$ (resp. $\theta_n$) is the mode in which the system is at time $\sigma_n$ (resp. $\t_n$). Next when the system is run under the control $u$ (resp. $v$) for $C_1$ (resp. $C_2$), there is a payoff $J(u,v)$ which is a profit (resp. cost) for $C_1$ (resp. $C_2$) given by: 
$$\begin{array}{l}
J(u,v)=\E[\int_0^Tf^{\delta_s}(s)ds-A^u_T+B^v_T+\zeta^{\d_T}].
\end{array}
$$
where $\delta:=(\d_s)_{s\le T}$ is the process valued in $\G$ which indicates the working modes of the system along with time. If at time $s$ the system is in mode $i_0$, then $\d_s=i_0$. It is bind to the controls $u$ and $v$ implemented by both players. On the other hand, for $\ig$, the process $f^i$ is the utility of the system in mode $i$ and finally $\zeta^{\d_T}$ is the terminal payoff or bequest.

The problem we are interested in is to know whether or not the game has a value, i.e., roughly speaking, if the following equality holds:
$$
\inf_v\sup_u J(u,v)=\sup_u\inf_v J(u,v)\,\,
$$
 In case of equality we say that the game has a value. Finally we say that the game has a saddle-point $(u^*,v^*)$ if, for any $u$ and $v$, controls of $C_1$ and $C_2$ respectively, we have:
$$
J(u,v^*)\le J(u^*,v^*)\le J(u^*,v).
$$
Note that in such a case, the game has a value.
\ms

\nd From the  probabilistic point of view, this zero-sum switching game problem turns into looking for a solution of its associated system of reflected BSDEs with interconnected bilateral obstacles (see e.g. \cite{hmm} for the case of proper modes of players). A solution for such a system are adapted processes $(Y^i,Z^i,K^{i,\pm})_{i\in \G}$ such that for any $\ig$, and $s\le T$, 
\begin{equation}\label{yigameintro}
\left\lbrace
\begin{array}{l}
Y^i \mx{ and } K^{i,\pm} \mx{ continuous};  
K^{i,\pm} \mx{ increasing}; (Z^i(\omega)_t)_{t\le T}\mx { is } dt-\mx{ square integrable}; \\ 
Y_s^i=\zeta^i+\int_s^Tf^i(r)dr-\int_s^TZ^i_rdB_r+K_T^{i,+}-K_s^{i,+}-(K_T^{i,-}-K_s^{i,-});\\
L^i(\vec Y)_s\leq Y_s^i\leq U^i(\vec Y)_s;\\
\int_0^T (Y^i_s-L^i(\vec Y)_s)dK_s^{i,+}=0\;\mbox{and} \; \int_0^T
(Y^i_s-U^i(\vec Y)_s)dK_s^{i,-}=0 
\end{array}
\right.
\end{equation}
where: a) $B:=(B_t)_{t\le T}$ is a Brownian motion; b) $\vec Y:=(Y^i)_{\ig}$; c) $L^i(\vec Y)_s=Y^{i+1}_s-\underline{g}_{i,i+1}(s)$ and $U^i(\vec Y)_s=Y^{i+1}_s+\bar {g}_{i,i+1}(s)$. 

Actually the solution of the previous system provides the value of the zero-sum switching game which is equal to $Y^i_0$ if the starting mode of the system is $i$. Roughly speaking, system \eqref{yigameintro} is the verification theorem for the zero-sum switching game problem. 

In the Markovian framework, i.e., when randomness stems from a diffusion process $X^{\tx}$ ($\txsp$) which satifies:
\begin{equation}\label{EDS_3intro}
\begin{array}{l}
dX^{\tx}_s=b(s,X^{\tx}_s)ds+\sigma(s,X^{\tx}_s)dB_s,\,\,s\in [t,,T] \mx{ and }X^{\tx}_s=x \mx{ for }s\leq t
\end{array}
\end{equation}and the data of the game are deterministic functions of $(s,X^{\tx}_s)$, the Hamilon-Jacobi-Bellman system associated with this switching game is
the following system of partial differential equations (PDEs in short) with a bilateral interconnected obstacles:
$\forall i \in \Gamma$, $\fr (\tx)\in\0TR,$
\begin{equation}\label{sysedpintro}
\left\lbrace
\begin{array}{l}
\min\{v^i(t,x)-L^i(\vec{v})(t,x);\max\left[v^i(t,x)-U^i(\vec{v})(t,x);\right.
\left.-\partial_tv^i(t,x)-\mathcal{L}^X(v^i)(t,x)-f^i(t,x)\right]\}=0;\\\\
v^i(T,x)=h^i(x).
\end{array}
\right.
\end{equation}
where: a) 
$\vec{v}=(v^i)_{\ig}$; b) $L^i(\vec{v})(t,x):=v^{i+1}(t,x)-\underline{g}_{\iplus1}(t,x), U^i(\vec{v})(t,x):=v^{i+1}(t,x)+\overline{g}_{\iplus1}(t,x)$; c) $\mathcal{L}^X$, the infinitesimal generator  of $X$, is given by: $\forall \phi \in {\cal C}^{1,2}(\0TR)$,
\begin{align*}
\begin{split}
\mathcal{L}^X\phi(t,x):=\dfrac{1}{2}\mx{Tr}[\sigma\sigma^\top(t,x)D^2_{xx}\phi(t,x)]+b(t,x)^\top D_x\phi(t,x).
\end{split}
\end{align*}Usually it is shown that the value functions of the game is a unique solution of \eqref{sysedpintro}.
\medskip

This work is originated by an article by N.Yamada \cite{Y83} where the author deals with the system of PDEs \eqref{sysedpintro} in the case when the switching costs are constant and for bounded domains $\bar\Omega$. By penalization method, the author proved existence and uniqueness of the solution of \eqref{sysedpintro} in a weak sense (actually in a Sobolev space). Then he gives an interpretation of the solution of this system as a value function of the zero-sum switching game described previously. A saddle-point of the game is also given. However neither this interpretation nor the existence of the saddle-point are clear because the question of admissiblity of the controls which are supposed to realize the saddle-point property is not addressed. In zero-sum switching games this issue of admissibility of those controls, defined implicitely through 
$(Y^i)_{\ig}$, is crucial (see e.g. \cite{hmm}). Note also that there is another paper by 
N.Yamada \cite{YN87} where the solution of system \eqref{sysedpintro} is considered in viscosity sense. Once more by penalization, he shows existence and uniqueness of the solution in bounded domains $\bar\Omega$.

Therefore the main objectif of this work is to show that: 

\nd i) the system of reflected BSDEs with interconnected obstacles \eqref{yigameintro} 
has a unique solution in  the Markovian framework. 

\nd ii) the zero-sum switching game described above has a value in different settings. 

\nd iii) The system of PDEs \eqref{sysedpintro} has a unique solution. 
\ms

Actually in this paper we show that system of PDEs \eqref{sysedpintro} has a unique continuous with polynomial growth solution $(v^i)_{\ig}$ in viscosity sense on $\0TR$. Mainly this solution is constructed by using Perron's method in combination with systems of reflected BSDEs with one lower interconnected obstacle and the Feynman-Kac representation of their solutions in the Markovian framework. Then we show that the following system of RBSDEs with interconnected bilateral obstacles has a unique solution: For any $\ig$ and $\stt$, 
\begin{equation}\label{yigameintro2}
\left\lbrace
\begin{array}{l}
Y^i \mx{ and } K^{i,\pm} \mx{ are continuous};  
K^{i,\pm} \mx{ are increasing }(K^{i,\pm}_t=0); (Z^i(\omega)_t)_{t\le T}\mx { is } dt-\mx{ square integrable}; \\ 
Y_s^i=h^i(X^{\tx}_T)+\int_s^Tf^i(r,X^{\tx}_r)dr-\int_s^TZ^i_rdB_r+K_T^{i,+}-K_s^{i,+}-(K_T^{i,-}-K_s^{i,-});\\
L^i(\vec Y)_s\leq Y_s^i\leq U^i(\vec Y)_s;\\
\int_t^T (Y^i_s-L^i(\vec Y)_s)dK_s^{i,+}=0\;\mbox{and} \; \int_t^T
(Y^i_s-U^i(\vec Y)_s)dK_s^{i,-}=0 
\end{array}
\right.
\end{equation}
where $X^{\tx}$ is the Markov process solution of \eqref{EDS_3intro}, $L^i(\vec Y)_s=Y^{i+1}_s-\underline{g}_{i,i+1}(s,X^{\tx}_s)$ and $U^i(\vec Y)_s=Y^{i+1}_s+\bar {g}_{i,i+1}(s,X^{\tx}_s)$. 
\ms

\nd Finally we consider the zero-sum switching game and we show that when the processes $Z^i$, $\ig$, of \eqref{yigameintro2} are:
\ms

\nd a) $dt\otimes d\p$-square integrable then $Y^i_0$ is the value of the game under square integrable controls, i.e., $\E[(A^u_T)^2+(B^v_T)^2]<\infty$.
\ms

\nd b)  only $\omega$ by $\om$, $dt$-square integrable then $Y^i_0$ is the value of the game under integrable controls, i.e., $\E[A^u_T+B^v_T]<\infty$.
\ms

\nd The paper is organized as follows:
\ms

\nd In Section 2, we introduce the zero-sum switching game and especially the notion of coupling which is already used in several papers including \cite{hmm,stet}. In Section 3, we show that the solution of \eqref{yigameintro2} is the value of the zero-sum switching game over square integrable controls when $Z^i$, $\ig$, are $dt\otimes d\p$-square integrable. Without additional assumptions on the data of the problem, this property is rather tough to check in practice because it depends on the room between the barriers $L^i(\vec Y)$  and $U^i(\vec Y)$ which depend on the solution $\vec Y$. For example, it is not clear how to assume an hypothesis like Mokobodski's one (see e.g. \cite{CK96,shjpl}) since the barriers depend on the solution and this latter is not explicit. However by localiztion, we can show that in some cases, e.g. when the switching costs are constant, $Y^i_0$ is actually the value function over square integrable controls even when we do not know that $Z^i$, $\ig$, 
are $dt\otimes d\p$-square integrable. In the case when for any $\ig$ and $\p$-a.s. $(Z^i_s(\om))_{s\le T}$ is $dt$-square integrable only, which is the minimum condition to define the stochastic integral, $Y_0^i$ is the value function of the zero-sum switching game over integrable controls. To show this property we proceed by localization. 
Section 4 is devoted to existence and uniqueness of the solution of system of PDEs \eqref{sysedpintro} in a more general form. The result is given in Theorem \ref{thm43}, but the main steps of its proof are postponed to Appendix. This proof 
is based on Perron's method and the construction of this solution (more or less the same as in \cite{DHMZ17})  proceeds as follows: a) we first introduce the processes $(Y^{i,m},Z^{i,m},K^{\pm,i,m})_{\ig}$, $m\ge 1$, solution of the system of reflected BSDEs with interconnected lower barriers associated with $\{
{f}^{i}(r,X_r^{\tx},\vec y, z^i)-m(y^i-y^{i+1}-\bar g_{i,i+1}(r,X_r^{\tx}))^+,h^i(X_T^{\tx}),\underline g_{i,i+1}(r,X_r^{\tx})\}_{\ig}$ (see \eqref{ym}). It is a 
decreasing penalization scheme. As the framework is Markovian then there exist deterministic functions continuous and of polynomial growth $(v^{i,m})_{\ig}$ such that the following Feynman-Kac representation holds: For any $\ig$, $m\ge 1$ and $\stt$,
$$
Y^{i,m}_s=v^{i,m}(s,X_s^{\tx}). 
$$
As for any $\ig$, $m\ge 1$, $Y^{i,m}\ge Y^{i,m+1}$ then we have also $v^{i,m}\ge v^{i,m+1}$. Now if we define $v^i=\lim_{m}v^{i,m}$, then $(v^i)_{\ig}$ is a subsolution of \eqref{sysedpintro} and for any fixed $m_0$, $(v^{i,m_0})_{\ig}$ is a supersolution of \eqref{sysedpintro}. Next it is enough to use Perron's method to show that \eqref{sysedpintro} has a unique solution since comparison principle holds. Finally, by uniqueness this solution does not depend on $m_0$ and is $(v^{i})_{\ig}$. Additionally for any $\ig$, $v^i$ is of polynomial growth and continuous. In Section 5, we show existence and uniqueness of the solution of system of RBSDEs \eqref{yigameintro}  and give some extensions. This proof is based on results on zero-sum Dynkin games and standard two barriers reflected BSDEs. The component $Y^i$, $\ig$, is just the limit of the processes $(Y^{i,m})_m$. We make use of the fact that, by Dini's Theorem, $(v^{i,m})_m$ converges to $v^i$ uniformly on compact sets  since $v^i$ is continuous and then the sequence $(Y^{i,m})_m$ converges uniformly in $L^2(d\p)$ to $Y^i$, $\ig$. As mentionned previously, this latter property stems from the PDE part. Note also that the following representation holds:
$$
\frst, Y^{i}_s=v^i(s,X_s^{\tx}).
$$
Here we should point out that since the switching of the system is made from $i$ to $i+1$ and the players do not have their proper sets of switching modes, then the method used e.g. in \cite{hmm} cannot be applied in our framework. As a consequence of this fact, the question of a solution of \eqref{yigameintro} outside the Markovian framework still open. At the end of the paper there is the Appendix. \qed
%*************************************************************************
%*                          2                                           **%
%*************************************************************************
\section{Preliminaries. Setting of the stochastic switching game}
Let $T$ be a fixed positive constant. Let $(\Omega,\mathcal{F},\mathbb{P})$
denote a complete probability space, $B=(B_t)_{t\in[0,T]}$ a \textit{d-}dimensional
Brownian motion whose natural filtration is $(\mathcal{F}^0_t:=\sigma\lbrace{B_s,s\leq t
\rbrace})_{0\leq t\leq T}$ and we denote by $\mathbb{F}=(\mathcal{F}_{t})_{0\leq t\leq T}$
the completed filtration of $(\mathcal{F}^0_{t})_{0\leq t\leq T}$ with the $\mathbb{P}$-null
sets of $\mathcal{F}$. Then it satisfies the usual conditions. On the other hand, let $\mathcal{P}$ be the $\sigma$-algebra on $[0,T]\times \Omega$ of the $\mathbb{F}$-progressively
measurable sets.\\

\noindent Next, we denote by:
\begin{itemize}
\item[-]
$\mathcal{S}^2$: the set of $\mathcal{P}$-measurable continuous processes $\phi=(\phi_t)_{t \in [0,T]}$ such that $\mathbb{E}(\sup_{t \in [0,T]} |\phi_t|^2)  < \infty$;
\item[-]
$\mathcal{A}^2$ : the subset of $\mathcal{S}^2$ with all non-decreasing processes $K=(K_t)_{t\leq T}$ with $K_0=0$;

\item[-] $\mathcal{A}_{loc}$: the set of $\mathcal{P}$-measurable continuous
non-decreasing processes $K=(K_t)_{t\leq T}$ with $K_0=0$ such that
$\p-a.s.\,\, K_T(\omega)<\infty$;

\item[-] $\mathcal{H}^{2,d}_{loc}(d\geq 1):$
the set of $\mathcal{P}$-measurable $\mathbb{R}^d$-valued processes
$\psi=(\psi_t)_{t \in [0,T]}$ such that $\p-a.s.$, $\int^T_0
|\psi_t|^2dt< \infty$.

\item[-]
$\mathcal{H}^{2,d}$: the subset of 
$\mathcal{H}^{2,d}_{loc}(d\geq 1)$ of processes $\psi=(\psi_t)_{t \in [0,T]}$ such
that $\mathbb{E}(\int^T_0 |\psi_t|^2dt)  < \infty$.
\item[-] For $s\le T$, $\mathcal{T}_{s}$ is the set of stopping times $\nu$ such that $\p$-a.s., $s\le \nu \le T$.

\item[-] $\mathcal{S}^2([t,T])$ is the set $\mathcal{S}^2$ reduced to the time interval $[t,T]$. The same meaning is valid for the other spaces introduced above.\qed 
\end{itemize}

\noindent Now for any $(t,x)\in [0,T]\times \mathbb{R}^k$, let us
consider the process $(X_s^{\tx})_{s\in[t,T]}$ solution of the following standard SDEs:
\begin{equation}\label{EDS_3}\left\{
\begin{array}{l}
dX_s^{t,x}=b(s,X_s^{t,x})ds+\sigma(s,X_s^{t,x})dB_s,\,\,s\in [t,T];\\\\
X_s^{t,x}=x, \,\, s\leq t
\end{array}\right.
\end{equation}
where, throughout this paper, $b$ and $\sigma$ satisfy the following conditions:
\begin{itemize}
\item[\textbf{[H0]}]
The functions $b$ and $\sigma$ are Lipschitz continuous w.r.t. $x$ uniformly in $t$, i.e. for any $(t,x,x') \in [0,T]\times \mathbb{R}^{k+k}$, there exists a non-negative constant $C$ such that
\begin{equation}\label{EDS_1}
 \vert \sigma(t,x)-\sigma(t,x')\vert +\vert b(t,x)-b(t,x')\vert \leq C\vert x-x' \vert.
\end{equation}
Moreover we assume that they are jointly continuous in $(t,x)$. The continuity of $b$ and $\sigma$ imply their linear
growth w.r.t. $x$, i.e. there exists a constant $C$ such that for any 
$(t,x) \in [0,T]\times \mathbb{R}^k,$
\begin{equation}\label{EDS_2}
\vert b(t,x) \vert +\vert \sigma(t,x)\vert \leq C(1+\vert x\vert).\qed
\end{equation}
\end{itemize}
Therefore under assumption [H0], the SDE \eqref{EDS_3} has a unique solution $X^{t,x}$ which satisfies the following
estimates: $\forall \gamma\geq 1,$
\begin{equation}\label{State_3}
\mathbb{E}[\sup_{s\leq T} \vert X_s^{t,x}\vert^\gamma]\leq C(1+\vert x\vert^\gamma).\qed
\end{equation}

Next a  function $\Phi: (t,x)\in [0,T]\times \mathbb{R}^k \mapsto
\Phi(t,x)\in \mathbb{R} $ is called of polynomial growth if there
exist two non-negative real constants $C$ and $\gamma$ such that
\begin{equation*}\forall (t,x)\in [0,T]\times \mathbb{R}^k,
\vert \Phi(t,x)\vert \leq C(1+\vert x\vert^\gamma).
\end{equation*}
Hereafter this class of functions is denoted by $\Pi_g$. 
%Finally let $C^{1,2}([0,T]\times\mathbb{R}^k)$ (or $C^{1,2}$ for short) denote the set of  real-valued functions defined on $[0,T]\times \mathbb{R}^k$ which are respectively once and twice differentiable w.r.t. $t$ and $x$, with continuous derivatives.\\
\subsection{Description of the zero-sum stochastic switching game}
Let $\Gamma:=\lbrace1,2,...,p\rbrace$ and for $i\in \Gamma$,
let us set $\Gamma^{-i}:=\Gamma-\{i\}$. For
$\vec{y}:=(y^i)_{i \in \Gamma} \in \mathbb{R}^p$ and
$\hat{y}\in \R$, we denote by $[\vec{y}_{-i},\hat{y}]$
or $[(y^k)_{k\in\Gamma^{-i}},\hat y]$, the element of $\rp$ obtained  in replacing the $i$-th component of $\vec{y}$ with $\hat{y}$. \\

\noindent We now introduce the following deterministic functions:
for any $i \in \Gamma$,
\begin{itemize}
\item[-]
$f^i$: $(t,x,\vec{y},z) \in [0,T] \times \mathbb{R}^{k+p+d} \mapsto f^i(t,x,\vec{y},z) \in \mathbb{R}$
\item[-]
$\underline{g}_{i,i+1}$: $(t,x) \in [0,T]\times\mathbb{R}^k \mapsto \underline{g}_{i,i+1}(t,x) \in \mathbb{R}$
\item[-]
$\overline{g}_{i,i+1}$: $(t,x) \in [0,T]\times\mathbb{R}^k  \mapsto \overline{g}_{i,i+1}(t,x) \in \mathbb{R}$
\item[-]
$h^i: x\in \mathbb{R}^k \mapsto h^i(x)\in \mathbb{R}$
\end{itemize}

\noindent Next let us consider a system with $p$ working modes
indexed by the set $\Gamma$. On the other hand, there are two agents
or controllers $C_1$ and $C_2$, whose interests or profits are antagonistic and who act on this system, along with
time, by switching its working mode from the current one, say $i_0$, to the next one $i_0+1$ if $i_0\le p-1$ and $1$ if $i_0=p$, whatever which agent decides to switch first. Therefore a switching control for
$C_1$ (resp. $C_2$) is $u:=(\sigma_n)_{n \geq
0}$ (resp. $v:=(\tau_n)_{n \geq 0}$) an
increasing sequence of stopping times which correspond to the
successive times where $C_1$ (resp. $C_2$) decides to switch the
system. The control $u$ (resp. $v$) is called $admissible$ if 
\begin{equation}\lb{adm1}\p[\sigma_n<T, \forall n\ge 0]=0 \,\,  (\mx{resp.}\,\, \p[\tau_n<T, \forall
n\ge 0]=0).\end{equation}
The set of admissible controls of $C_1$ (resp. $C_2$) is denoted by $\bf{A}$ (resp. $\bf{B}$). 
%We are now going to define pairs of admissible strategies of the players. 

Now let $u:=(\sigma_n)_{n \geq 0}$ (resp.
$v:=(\tau_n)_{n \geq 0}$) be an admissible control of $C_1$
(resp. $C_2$). Let $(r_n)_{n\geq 0}$ and $(s_n)_{n\ge 0}$ be the
sequences defined by: $r_0=s_0=0$, $r_1=s_1=1$ and for $n\ge 2$,
$$
  r_n=r_{n-1}+\ind_{\{\sigma_{r_{n-1}}\le \tau_{s_{n-1}}\}}  \mx{ and
  }
  s_n=s_{n-1}+\ind_{\{\tau_{s_{n-1}}< \sigma_{r_{n-1}}\}}.
$$
For $n\ge 0$, let us set $\rho_n=\sigma_{r_n}\wedge
\tau_{s_n}$. It is a stopping time and it stands for the time when
the $n$-th switching of the system, by one of the players, occurs. On the other hand, the piecewise process $(\theta(u,v)_s)_{s\le T}$ which indicates in which mode the system is at time $s$ is given by: 
$\forall s\le T$,
$$
  \theta(u,v)_s=\theta_0 \ind_{[\rho_0,\rho_1]}(s)+\sum_{n\ge 1}\theta_n\ind_{(\rho_n,\rho_{n+1}]}(s)
                                      $$
where:

i) $(\rho_n,\rho_{n+1}]=\varnothing$ on $\{\rho_n=\rho_{n+1}\}$ ;

ii) $\theta_0=i$ if at $t=0$, the system is in mode $i$ ;

iii) For $n\ge 1$, $\theta_n=\theta_{n-1}+1$ if $\theta_{n-1}\le p-1$ and $\theta_n=1$ if $\theta_{n-1}=p$.

\nd The sequence $\Theta(u,v):=(\rho_n,\theta_n)_{n\ge 0}$, called the $coupling$ of $(u,v)$, indicates the successive times and modes of switching of the system operated by the players. 
\ms

When the players implement the pair of admissible controls $(u,v)$, this incurs switching costs which amount to $A^u_T$ and $B^v_T$, for $C_1$ and $C_2$ respectively, and given by:
$$ \forall s<T, A^u_s=\sum_{n\ge 1}\underline g_{\theta_{n-1}\theta_n}(\rho_n,X^{0,x}_{\rho_n})\ind_{\{\rho_n=\sigma_{r_n}\le s\}} \mx{ and } A^u_T=\lim_{s\rightarrow T}A^u_s;$$ $$
\forall s<T, B^v_s=\sum_{n\ge 1}\overline g_{\theta_{n-1}\theta_n}(\rho_n,X^{0,x}_{\rho_n})\ind_{\{\rho_n=\t_{s_n}\le s\}} \mx{ and } B^v_T=\lim_{s\rightarrow T}B^v_s.$$
The admissible control $u$ (resp. $v$) of $C_1$ (resp. $C_2$) is called {\it square integrable }if $$
\mathbb{E}[(A^u_T)^2]<\infty\,\, (\mx{resp. } \mathbb{E}[(B^v_T)^2]<\infty).
$$
The set of square integrable admissible controls of $C_1$ (resp. $C_2$) is denoted by $\ca$ (resp.  $\cb$). 
\ms

\nd The admissible control $u$ (resp. $v$) of $C_1$ (resp. $C_2$) is called {\it integrable }if $$
\mathbb{E}[A^u_T]<\infty\,\, (\mx{resp. } \mathbb{E}[B^v_T]<\infty).
$$
The set of integrable admissible controls of $C_1$ (resp. $C_2$) is denoted by $\ca^{(1)}$ (resp.  $\cb^{(1)}$). 
\ms

\nd The coupling $\tuv$, of a pair $(u,v)$ of admissible controls, is called {\it square integrable} (resp. {\it integrable}) if 
$$C^{\theta(u,v)}_\infty:=\lim_{n\rightarrow \infty}C^{u,v}_N \in L^2(d\p) \,\,(\mx{resp.} \in L^1(d\p) )$$
where for any $N\ge 1$, 
$$
C^{\theta(u,v)}_N:=\sum_{n=1,N}\underline g_{\theta_{n-1}\theta_n}(\rho_n,X^{0,x}_{\rho_n})\ind_{\{\rho_n=\sigma_{r_n}<T\}}- \sum_{n=1,N}\overline g_{\theta_{n-1}\theta_n}(\rho_n,X^{0,x}_{\rho_n})\ind_{\{\rho_n=\t_{s_n}<T\}} .
$$
Note that $C^{\theta(u,v)}_\infty$, defined as the pointwise limit of $C^{\theta(u,v)}_N$, exists since the controls $u$ and $v$ are admissible. On the other hand, the quantity $C^{\theta(u,v)}_N$ is nothing but the switching costs associated with the $N$ first switching actions of both players. 
\ms

Next when $C_1$ (resp. $C_2$) implements $u\in \bf{A}$ (resp. $v\in \bf{B}$), there is a payoff which a is reward for $C_1$ and a cost for $C_2$ which is given by (we suppose that $\theta_0=i$):
\begin{equation}\label{value}
J_i(\tuv)=\mathbb{E}\left[ h^{\theta(u,v)_T}(X_T^{\ox})+\int_0^T
f^{\theta(u,v)_r}(r,X_r^{\ox}) dr -C_\infty^{\theta(u,v)}\right].
\end{equation}
It means that between $C_1$ and $C_2$ there is a game of zero-sum type. The main objective of this section is to deal with the issue of existence of a value for this zero-sum switching game, i.e., whether or not it holds
\begin{equation}\lb{valsqint}
\inf_{v\in  \cb}\sup_{u\in  \ca }J_i(\tuv)=\sup_{u\in  \ca }\inf_{v\in  \cb }J_i(\tuv)
\end{equation}
or 
\begin{equation}\lb{valint}
\inf_{v\in  \cbi}\sup_{u\in  \cai}J_i(\tuv)=\sup_{u\in  \cai }\inf_{v\in  \cbi}J_i(\tuv).
\end{equation}
\begin{remark} In our framework when the players decide to switch at
the same time, we give priority to the maximizer $C_1$. This appears
through the definition of $r_n$ for $n\ge 2$. On the other hand, for the well-posedness of  $J_i(\tuv)$, it is enough that the coupling $\tuv$ is integrable. \end{remark}

To proceed we are going to define the notion of admissible 
square integrable and integrable strategies.
\begin{definition}[Non-anticipative switching strategies]\label{dss}
	Let $s \in [0,T]$ and $\nu$ a stopping time such that $\p$-a.s. $\nu\ge s$. Two controls $u^{1}=(\s_n^1)_{n\ge 0}$ and $u^{2}=(\s_n^2)_{n\ge 0}$ in $\bf A$ are said to be {\it equivalent}, denoting this by $u^{1} \equiv u^{2}$, {\it on $[s,\nu]$} if we have $\p$-a.s.,
	\[
	\mathbf{1}_{[\sigma^{1}_{0},\sigma^{1}_{1}]}(r) + \sum\limits_{n \ge 1}\mathbf{1}_{(\sigma^{1}_{n},\sigma^{1}_{n+1}]}(r) = \mathbf{1}_{[\sigma^{2}_{0},\sigma^{2}_{1}]}(r) + \sum\limits_{n \ge 1}\mathbf{1}_{(\sigma^{2}_{n},\sigma^{2}_{n+1}]}(r), \quad s \le r \le \nu.
	\]
	A {\it non-anticipative strategy} for $C_1$ is a mapping $\overline{\alpha} \colon \bf{B} \to \bf{A}$ such that for any $s \in [0,T]$, $\nu \in \mathcal{T}_{s}$, and $v^{1},v^{2} \in \bf{B}$ such that $v^{1} \equiv v^{2}$ on $[s,\nu]$, we have $\overline{\alpha}(v^{1}) \equiv \overline{\alpha}(v^{2})$ on $[s,\nu]$.
\ms

\nd The non-anticipative strategy $\overline \alpha$ for $C_1$ is called $square-integrable$	(resp. $integrable$) if for any $v \in \mathcal{B}$ we have $\overline{\alpha}(v) \in \mathcal{A}$ (resp. 
for any $v \in \mathcal{B}^{(1)}$ we have $\overline{\alpha}(v) \in \cai$).
\ms

\nd In a similar manner we define non-anticipative strategies,  square integrable and merely integrable strategies for $C_2$ denote by $\overline \beta$.
\ms

\nd We denote by $\mba$ and $\mbb$ (resp. $\mba^{(1)}$ and $\mbb^{(1)}$) the set of non-anticipative 
square integrable (resp. integrable) strategies for $C_1$ and $C_2$ respectively. \qed
\end{definition}
\section{Existence of a value of the zero-sum switching game. Link with systems of reflected BSDEs}

We are now going to deal with the issue of existence of a value for the zero-sum switching game described previously. For that let us introduce the following assumptions on the functions $f^i$, $h^i$, $\underline{g}_{i,i+1}$ and $\overline{g}_{i,i+1}$. Some assumptions will be only  applied in the next sections.
\ms

\noindent \textbf{Assumptions (H):} \ms

\noindent \textbf{[H1]} For any
$ i \in \Gamma$, $f^i$ does not
depend on $(\vec{y},z)$, is continuous in $(t,x)$ and \\belongs to class $\Pi_g$ ;
\ms

\noindent  \textbf{[H2]} For any $i \in \Gamma$, the function $h^i$,
which stands for the terminal payoff, is continuous w.r.t. $x$,
belongs to class $\Pi_g$ and satisfies the following consistency
condition: $\forall i \in \Gamma, \forall x\in \mathbb{R}^k,$
\begin{equation}
h^{i+1}(x)-\underline{g}_{i,i+1}(T,x)\leq h^i(x)\leq h^{i+1}(x)+\overline{g}_{\iplus1}(T,x).
\end{equation}

\noindent \textbf{[H3]} a) For all $i \in \Gamma$ and $(t,x) \in
[0,T]\times\mathbb{R}^k$, the functions $\underline{g}_{\iplus1}$ and $\overline{g}_{\iplus1}$ are continuous, non-negative, belong to $\Pi_g$  and verify: $$\underline{g}_{\iplus1}(\tx)+
    \overline{g}_{\iplus1}(\tx)>0.$$
%    b) For any $i,j,k \in \Gamma$ such that $k \not=i,j, $\\
%    \begin{equation}\label{2.1}
%    \begin{array}{c}
%    \mathbb{P}-a.s. \; \quad \underline{g}_{ij}(t,x) \leq \underline{g}_{ik}(t,x)+\underline{g}_{kj}(t,x) \;
%    \left(resp. \;\overline{g}_{ij}(t,x) \leq \overline{g}_{ik}(t,x)+\overline{g}_{kj}(t,x) \right);
%    \end{array}
%    \end{equation}
b) They satisfy the non-free loop property, i.e., for any $j\in \Gamma$ and $(t,x)\in \0TR$,
\begin{equation}\label{nflp}
\varphi_{j,j+1}(t,x)+...+
\varphi_{p-1,p}(t,x)+
\varphi_{p,1}(t,x)+...+
\varphi_{j-1,j}(t,x)\neq 0
\end{equation}
where $\varphi_{\ell,\ell+1}(t,x)$ is either
$-\underline{g}_{\ell,\ell+1}(t,x)$ or
$\overline{g}_{\ell,\ell+1}(t,x)$. Let us notice that \eqref{nflp} also implies:
\begin{equation}\label{nflp2}
\overline g_{j,j+1}(t,x)+...+
\overline g_{p-1,p}(t,x)+
\overline g_{p,1}(t,x)+...+
\overline g_{j-1,j}(t,x)> 0
\end{equation} and
 \begin{equation}\label{nflp3}
\underline g_{j,j+1}(t,x)+...+
\underline g_{p-1,p}(t,x)+
\underline g_{p,1}(t,x)+...+
\underline g_{j-1,j}(t,x)> 0.
\end{equation}

\noindent \textbf{[H4]} For any $i=1,...,m$, the processes 
$(\bar{g}_{i,i+1}(s,X^{0,x}_s))_{s\le T}$ and $(\underline{g}_{i,i+1}(s,X^{0,x}_s))_{s\le T}$ are non decreasing.\qed
\ms

\noindent\textbf{[H5]} For any $\ig,$
\begin{enumerate}
\item[a)]$f^i$ is Lipschitz continuous in $(\vec{y},z)$ uniformly in $(\tx),$ i.e. for any $\vec{y}_1, \vec{y}_2\in\R^p$, $z_1,z_2\in \R^d$, $\txsp$,
\[|f^i(\tx,\vec{y}_1,z_1)-f^i(\tx,\vec{y}_2,z_2)|\leq C(|\vec{y}_1-\vec{y}_2|+|z_1-z_2|);\]
\item[b)] $\forall j\in\Gamma^{-i}$, the mapping $\bar y\mapsto f^i(t,x,[(y^k)_{k\in\Gamma^{-j}},\bar y],z)$ is non-decreasing when the other components $t,x,(y^k)_{k\in\Gamma^{-j}}$ and $z$ are fixed.
\item[c)]$f^i$ is continuous in $(t,x)$ uniformly in $(\vec y, z)$ and $f^i(t,x,0,0)$ belongs to $\Pi_g$. 
\end{enumerate}

In order to deal with the zero-sum switching game we rely on
solutions of systems of reflected BSDEs with oblique reflection or inter-connected bilateral obstacles of type \eqref{yigame} below. The following result
whose proof is given in Section 5 will allow us to show that the zero-sum switching game has a value.
\begin{theorem}\label{thyigame}Assume that assumptions [H1], [H2] and [H3] are fulfilled.
Then there exist processes $(Y^i,Z^i,K^{i,\pm})_{i\in \G}$ such
that: For any $\ig$ and $(\tx) \in \0TR$, $\forall \stt$,
\begin{equation}\label{yigame}
\left\lbrace
\begin{array}{l}
Y^i \in \sd ([t,T]); K^{i,\pm} \in \ad \,\,(K^{i,\pm}_t=0) \mbox{ and }Z^i \in \hd; \\ \\
Y_s^i=h^i(X_T^{\tx})+\int_s^Tf^i(r,X_r^{\tx})dr-\int_s^TZ_rdB_r+K_T^{i,+}-K_s^{i,+}-(K_T^{i,-}-K_s^{i,-});\\\\
L^i(\vec Y)_s\leq Y_s^i\leq U^i(\vec Y)_s;\\\\
\int_t^T (Y^i_s-L^i(\vec Y)_s)dK_s^{i,+}=0\;\mbox{and} \; \int_t^T
(Y^i_s-U^i(\vec Y)_s)dK_s^{i,-}=0 ;
\end{array}
\right.
\end{equation}
where for any $\stt$, $L^i(\vec
Y)_s:=Y^{i+1}_s-\underline{g}_{\iplus1}(s,X_s^{\tx})$ and $U^i(\vec
Y)_s:=Y^{i+1}_s+\bar{g}_{\iplus1}(s,X_s^{\tx})$.
\end{theorem}
\noindent Note that obviously the solution
$(Y^i,Z^i,K^{i,\pm})_{i\in \G}$ of \eqref{yigame} depends also on
$(\tx)$ which we omit as there is no possible confusion. 

To proceed let
$(Y^i,Z^i,K^{i,\pm})_{i\in \G}$ be the solution of \eqref{yigame}
when $t=0$.  We then have (see e.g. \cite{shjpl}, for more details):
\begin{proposition}\label{Problem:Implicit-Dynkin-Game-Upper-Value}
For all $\ig$ and $s\le T$,  
	\vskip1em
\nd $(a)$ \enskip
	\begin{equation}\label{eq:Implicit-Dynkin-Game-Value}
	Y^{i}_{0} = \essinf_{\tau \in \mathcal{T}_{0}}\esssup_{\sigma \in \mathcal{T}_{0}}\mathcal{J}^{i}_{0}(\sigma,\tau) = \esssup_{\sigma \in \mathcal{T}_{0}}\essinf_{\tau \in \mathcal{T}_{0}}\mathcal{J}^{i}_{0}(\sigma,\tau),
	\end{equation}
	where,\begin{equation}\begin{array}{l}
		\mathcal{J}^{i}_{s}(\sigma,\tau)=\mathbb{E}\biggl[\int_{s}^{\sigma \wedge \tau}f^{i}(r,X^{0,x}_r)dr+ \mathbf{1}_{\{\tau < \sigma\}}{U^{i}_{\tau}}({Y}) + \mathbf{1}_{\{\sigma \le \tau,\; \sigma < T\}}{L^{i}_{\sigma}}({Y})
	 +h^i(X_T^{0,x})\mathbf{1}_{\{\sigma = \tau = T\}}\bigm \vert \mathcal{F}_{s} \bigr].
\end{array}		\end{equation}
$(b)$ We have $Y^{i}_{s} = \mathcal{J}^{i}_{s}(\sigma^{i}_{s},\tau^{i}_{s})$ where $\sigma^{i}_{s} \in \mathcal{T}_{s}$ and $\tau^{i}_{s} \in \mathcal{T}_{s}$ are stopping times defined by,
	\begin{equation}\label{eq:Implicit-Obstacle-Time-Selector}
	\begin{cases}
	\sigma^{i}_{s}  = \inf\{s \le t \le T \colon Y^{i}_{t} = L^{i}_{t}(\vec {Y})\} \wedge T, \\
	\tau^{i}_{s} = \inf\{s \le t \le T \colon Y^{i}_{t} = U^{i}_{t}(\vec {Y})\} \wedge T,
	\end{cases}
	\end{equation}
	and we use the convention that $\inf\emptyset = +\infty$. Moreover, $\bigl(\sigma^{i}_{s},\tau^{i}_{s}\bigr)$ is a saddle-point for the zero-sum Dynkin game, i.e.,
	\begin{equation}
	\mathcal{J}^{i}_{s}(\sigma,\tau^{i}_{s}) \le \mathcal{J}^{i}_{s}(\sigma^{i}_{s},\tau^{i}_{s}) \le \mathcal{J}^{i}_{s}(\sigma^{i}_{s},\tau) \quad \forall \sigma,\,\,\t\in \mathcal{T}_{s}. \qed
	\end{equation}
\end{proposition}
\begin{remark}
For any $s<T$ and $\ig$, $\p[\s^i_s=\t^i_s<T]=0$ due to assumption [H3]-a) on $\underline g_{i,i+1}$ and $\overline{g}_{i,i+1}$.
\end{remark}
\subsection{Value of the zero-sum switching game on square integrable admissible controls}
We are now going to focus on the link between $Y^i$, $\ig$, with the value function of the zero-sum switching game over square integrable controls, namely the relation \eqref{valsqint}. For that we are going to make another supplementary assumption on the solution $(Y^i,Z^i,K^{i,\pm})_{i\in \G}$ of system \eqref{yigame} which is related to integrability of $Z^i$, $\ig$. Later on we will show that we have also the relation \eqref{valsqint} without this latter assumption, but at the price of some additional regularity properties of the switching costs 
$\underline{g}_{\iplus1}$ and $\bar{g}_{\iplus1}$ (see [H4]).
\ms

To proceed, consider the following sequence $(\rho_n,\theta_n)_{n\ge 0}$ defined as following: $\rho_0=0$, $\theta_0=i$ and for $n\ge 1$,
$$
\rho_n=\sigma_{\rho_{n-1}}^{\theta_{n-1}}\wedge \t_{\rho_{n-1}}^{\theta_{n-1}} \mx{ and }\theta_n=\left \{\begin{array}{l}
1+\theta_{n-1} \mx{ if }\theta_{n-1} \le p-1,\\
1 \mx{ if }\theta_{n-1}=p;
\end{array}
\right.
$$
where $\sigma_{\rho_{n-1}}^{\theta_{n-1}}$ and $\t_{\rho_{n-1}}^{\theta_{n-1}}$ are defined using \eqref{eq:Implicit-Obstacle-Time-Selector}. Next let $u^{(1)}:=(u_s^{(1)})_{s\le T}$ (resp. $u^{(2)}:=(u_s^{(2)})_{s\le T}$) be the piecewise process defined by: $u_s^{(1)}=0 \mx{ for }s<\rho_1 \mx{ and for } 
n\ge 1, s\in [\rho_n,\rho_{n+1}),$
$$
\,\,u_s^{(1)}=\left\{\begin{array}{l}
1+u_{\rho_n-}^{(1)} \mx {if }Y^{\theta_{n-1}}_{\rho_n}=
Y^{\theta_n}_{\rho_n}-\underline g_{\theta_{n-1},\theta_{n}}(\rho_n,X^{0,x}_{\rho_n}),\\
u_{\rho_n-}^{(1)} \mx{ if }Y^{\theta_{n-1}}_{\rho_n}>
Y^{\theta_n}_{\rho_n}-\underline g_{\theta_{n-1},\theta_{n}}(\rho_n,X^{0,x}_{\rho_n})
\end{array}
\right.
$$where $u_{\rho_n-}^{(1)}$ is the left limit of $u^{(1)}$ at ${\rho_n}$ (resp.
$u_s^{(2)}=0 \mx{ for }s<\rho_1 \mx{ and for } 
n\ge 1, s\in [\rho_n,\rho_{n+1}),$
$$
\,\,u_s^{(2)}=\left\{\begin{array}{l}
1+u_{\rho_n-}^{(2)} \mx {if }Y^{\theta_{n-1}}_{\rho_n}=
Y^{\theta_n}_{\rho_n}+\overline g_{\theta_{n-1},\theta_{n}}(\rho_n,X^{0,x}_{\rho_n}),\\
u_{\rho_n-}^{(2)} \mx{ if }Y^{\theta_{n-1}}_{\rho_n}<
Y^{\theta_n}_{\rho_n}+\overline g_{\theta_{n-1},\theta_{n}}(\rho_n,X^{0,x}_{\rho_n})
\end{array}
\right.
$$where $u_{\rho_n-}^{(2)}$ is the left limit of $u^{(2)}$ at ${\rho_n}$). Next let $u^*$ and $v^*$ be the following sequences of stopping times: $\s^*_0=\t^*_0=0$ and for $n\ge 1$,
$$
\s^*_n=\inf\{s\ge \s^*_{n-1}, u^{(1)}_s>u^{(1)}_{s-}\}\wedge T\mx{ and }
\t^*_n=\inf\{s\ge \t^*_{n-1}, u^{(2)}_s>u^{(2)}_{s-}\}\wedge T.$$
Note that $\theta(u^*,v^*)=(\rho_n,\theta_n)_{n\ge 0}$. We then have:
\begin{proposition}\lb{prop34}Assume that [H1], [H2], [H3] and $\hz$.  Then the  following properties of  $u^*=(\s^*_n)_{n\ge 0}$ and $v^*=(\t^*_n)_{n\ge 0}$  hold true:

i)  $u^*$ and $v^*$ are admissible ;

ii) the coupling $\theta(u^*,v^*)$ is square integrable ; 

iii) $$Y^i_0=J_i(\theta(u^*,v^*)).$$
\end{proposition}
\begin{proof}
i) Let us show that $u^*$ is admissible. Assume that $\p[\s^*_n<T, \forall n\ge 0]>0$. As the $\s^*_n$'s are defined through the $\rho_n's$, then there exists a loop $\{j,j+1,...,p-1,p,1,...,j-1,j\}$ such that 
$$\begin{array}{l}
\p[\omega, \exists \mx{ a subsequence }(n_\ell)_{\ell\ge 0} \mx{ such that }Y^j_{\rho_{n_\ell}}=Y^{j+1}_{\rho_{n_\ell}}+\varphi_{j,j+1}(\rho_{n_\ell},X^{0,x}_{\rho_{n_\ell}}), \dots,\\ \qquad\qquad Y^{j-1}_{\rho_{n_\ell+p-1}}=Y^{j}_{\rho_{n_\ell+p-1}}+\varphi_{j-1,j}(\rho_{n_\ell+p-1},X^{0,x}_{\rho_{n_\ell+p-1}}), \forall \ell \ge 0]>0\end{array}
$$ where $\varphi_{i,i+1}$ is the same as in \eqref{nflp} and equal to either
$-\underline{g}_{i,i+1}$ or $\overline{g}_{i,i+1}$ depending on whether $C_1$ or $C_2$ makes the decision to switch from the current state $j_0$ to the next one. Next let us set $\gamma=\lim_ {\ell\rightarrow \infty}\rho_{n_\ell}$. Take the limit w.r.t $\ell$ in the previous equalities to deduce that:
$$
\p[\varphi_{j,j+1}(\gamma,X^{0,x}_\gamma)+...+
\varphi_{p-1,p}(\gamma,X^{0,x}_\gamma)+
\varphi_{p,1}(\gamma,X^{0,x}_\gamma)+...+
\varphi_{j-1,j}(\gamma,X^{0,x}_\gamma)=0]>0
$$which is contradictory with the non free loop property \eqref{nflp}. By the same reasoning we obtain the admissibility of $v^*$.
\ms

\noindent ii) Let us recall the definition of the square integrability for $\theta(u^*,v^*)$. As $u^*$ and $v^*$ are proved admissible in i), then the coupling $\theta(u^*,v^*)$ exists. Next we will prove that $\lim_{N\rwi}C_N^{u^*,v^*}\in\mathbb{L}^2(d\mathbb{P})$.

For this recall that $i$ is fixed, $\rho_0=0$ and $\theta_0=i$. Next let us consider the equation satisfied by $Y^i$ on $[0,\rho_1]$. We then have: 
\begin{align}\label{cn1}
Y_0^i &= h^i(X_T^{0,x})1_{(\rho_1=T)}+Y_{\rho_1}^i1_{(\rho_1<T)}+\displaystyle\int_0^{\rho_1} f^i\left(r,X_r^{0,x}\right)dr-\int_0^{\rho_1}Z_r^idB_r+\int_0^{\rho_1} dK_r^{i,+}-\int_0^{\rho_1} dK_r^{i,-}\nonumber\\\
&=h^i(X_T^{0,x})1_{(\rho_1=T)}+\left(Y_{\sigma_0^i}^{i+1}-\underline{g}_{i,i+1}(\sigma_0^i,X_{\sigma_0^i}^{0,x})\right)1_{(\sigma_0^i\leq \tau_0^i)}1_{(\sigma_0^i<T)}+\left(Y_{\tau_0^i}^{i+1}+\bar{g}_{i,i+1}(\tau_0^i,X_{\tau_0^i}^{0,x})\right)1_{(\tau_0^i<\sigma_0^i)}\nonumber\\
&\qq \qq\qq+\displaystyle \int_0^{\rho_1}f^i\left(r,X_r^{0,x}\right)dr-\int_0^{\rho_1}Z_r^idB_r\nonumber\\\
&=h^{\theta_0}(X_T^{0,x})1_{(\rho_1=T)}+Y_{\rho_1}^{\theta_1}1_{(\rho_1<T)}-\left[ \underline{g}_{\theta_0\theta_1,}(\rho_1,X_{\rho_1}^{0,x})1_{(\rho_1=\sigma_0^{\theta_0})}-\bar{g}_{\theta_0\theta_1}(\rho_1,X_{\rho_1}^{0,x})1_{(\rho_1=\tau_0^{\theta_0})}  \right]1_{(\rho_1<T)}\nonumber\\
&\qq \qq\qq+\displaystyle \int_0^{\rho_1}f^{\theta_0}\left(r,X_r^{0,x}\right)dr-\int_0^{\rho_1}Z_r^{\theta_0}dB_r.
\end{align}
Next we deal with $Y_{\rho_1}^{\theta_1}$ by considering the doubly RBSDEs \eqref{yigame} in the interval $[\rho_1,\rho_2]$, i.e.
\begin{align}\label{cn2}
Y_{\rho_1}^{\theta_1}&=Y_{\rho_1}^{i+1}\nonumber\nonumber\\\
&=h^{\theta_1}(X_T^{0,x})1_{(\rho_2=T)}+Y_{\rho_2}^{\theta_2}1_{(\rho_2<T)}-\left[ \underline{g}_{\theta_1\theta_2}(\rho_2,X_{\rho_2}^{0,x})1_{(\rho_2=\sigma_{\rho_1}^{\theta_1})}-\bar{g}_{\theta_1\theta_2}(\rho_2,X_{\rho_2}^{0,x})1_{(\rho_2=\tau_{\rho_1}^{\theta_1})}  \right]1_{(\rho_2<T)}\nonumber\\
&+\displaystyle \int_{\rho_1}^{\rho_2}f^{\theta_1}\left(r,X_r^{0,x}\right)dr-\int_{\rho_1}^{\rho_2}Z_r^{\theta_1}dB_r.
\end{align}
By replacing $Y_{\rho_1}^{\theta_1}$ in \eqref{cn1} with \eqref{cn2}, then \eqref{cn1} yields
\begin{align}\label{cn3}
Y_0^i &= \displaystyle\sum_{n=1}^2h^{\theta_{n-1}}(X_T^{0,x})1_{(\rho_n=T)}1_{(\rho_{n-1}<T)}+Y_{\rho_2}^{\theta_2}1_{(\rho_2<T)}+ \int_{0}^{\rho_2}f^{\theta(u^*,v^*)_r}\left(r,X_r^{0,x}\right)dr-\int_0^{\rho_2}Z_r^{\theta(u^*,v^*)_r}dB_r\nonumber\\
&-\sum_{n=1}^2\left[ \underline{g}_{\theta_{n-1}\theta_n}(\rho_n,X_{\rho_n}^{0,x})1_{(\rho_n=\sigma_{\rho_{n-1}}^{\theta_{n-1}},\rho_n<T)}-\bar{g}_{\theta_{n-1}\theta_{n}}(\rho_n,X_{\rho_n}^{0,x})1_{(\rho_n=\tau_{\rho_{n-1}}^{\theta_{n-1}},\rho_n<T)}  \right].
\end{align}
Following \eqref{cn3} we replace iteratively $Y_{\rho_n}^{\theta_n}$ for $n=1,2,...,N$ we deduce that
\begin{align}\label{cn4}
Y_0^i&=\displaystyle\sum_{n=1}^Nh^{\theta_{n-1}}(X_T^{0,x})1_{(\rho_n=T)}1_{(\rho_{n-1}<T)}+Y_{\rho_N}^{\theta_N}1_{(\rho_N<T)}-C_N^{\theta(u^*,v^*)}+
\int_0^{\rho_N}f^{\theta(u^*,v^*)_r}(r,X_r^{0,x})dr\nonumber\\
&\qq\qq\qq-\int_0^{\rho_N}Z_r^{\theta(u^*,v^*)_r}dB_r.
\end{align} 
From \eqref{cn4} we obtain: $\fr N\ge 1$, 
\begin{align}
%\label{cn5}
|C_N^{\theta(u^*,v^*)}|&\le \displaystyle\sum_{n=1}^N|h^{\theta_{n-1}}(X_T^{0,x})|1_{(\rho_n=T)}1_{(\rho_{n-1}<T)}+|Y_{\rho_N}^{\theta_N}1_{(\rho_N<T)}|+|Y_0^i|+|\int_0^{\rho_N}f^{\theta(u^*,v^*)_r}(r,X_r^{0,x})dr|\nonumber\\
&\qquad\qquad  +|\int_0^{\rho_N}Z_r^{\theta(u^*,v^*)_r}dB_r|\nn\\
&\leq \max_{\i\in\G}\left|h^i(X_T^{0,x})\right|+2\max_{\i\in\G}\sup_{s\in[0,T]}|Y_s^i|+\displaystyle\int_0^T|f^{\theta(u^*,v^*)_r}(r,X_r^{0,x})|dr+\sup_{s\in[0,T]}|\int_0^{s}Z_r^{\theta(u^*,v^*)_r}dB_r|.\nn
\end{align}
Finally by taking the supremum over $N$ we obtain:
\begin{align}\label{cn6}
\sup_{N\geq 1}\left|C_N^{\theta(u^*,v^*)}\right|&\leq \max_{\i\in\G}\left|h^i(X_T^{0,x})\right|+2\max_{\i\in\G}\sup_{s\in[0,T]}|Y_s^i|\nonumber\\&+\displaystyle\int_0^T|f^{\theta(u^*,v^*)_r}(r,X_r^{0,x})|dr+\sup_{s\in[0,T]}|\underbrace{\int_0^{s}Z_r^{\theta(u^*,v^*)_r}dB_r}_{M_s^{\theta(u^*,v^*)_s}}|.
\end{align}As $(Z^i)_{i\in\G}$ are $dt\otimes d\p$-square integrable, then  
$$
\E[\sup_{s\le T}|M_s^{\theta(u^*,v^*)_s}|^2]\le C\E[\sum_{i=1,m}\int_0^T|Z^i_s|^2ds]<\infty.
$$
It implies that the right-hand side of \eqref{cn6} belongs to 
$\mathbb{L}^2(d\mathbb{P})$ and then 
$\lim_{N\rwi}C_N^{\theta(u^*,v^*)}$ is square integrable, therefore $\theta(u^*,v^*)$ is 
square integrable.
\ms

\noindent Finally for iii), by directly taking the expectation on both sides of \eqref{cn4} we obtain
\begin{align}\label{cn7}
Y_0^i&=\mathbb{E}\left[\sum_{n=1}^Nh^{\theta_{n-1}}(X_T^{0,x})1_{(\rho_n=T)}1_{(\rho_{n-1}<T)}+Y_{\rho_N}^{\theta_N}1_{(\rho_N<T)}-C_N^{\theta(u^*,v^*)}+
\int_0^{\rho_N}f^{\theta(u^*,v^*)_r}(r,X_r^{0,x})dr\right ]
\end{align} 
Now it is enough to take the limit w.r.t. $N$ in \eqref{cn7} and to use the Lebesgue dominated convergence theorem since $\lim_{N\to\infty}\rho_N=T$ and considering \eqref{cn6}, to deduce that \begin{align*}
Y_0^i&=\mathbb{E}\left[h^{\theta(u^*,v^*)_T}(X_T^{0,x})+
\int_0^Tf^{\theta(u^*,v^*)_r}(r,X_r^{0,x})dr-C_\infty^{\theta(u^*,v^*)}\right]=J_i(\theta(u^*,v^*))
\end{align*}
since $\lim_{N\to\infty}C_N^{\theta(u^*,v^*)}=C_\infty^{\theta(u^*,v^*)}.$
\end{proof}
\ms

Let $i$ be the starting mode of the system which is fixed. Let $\s=(\s_n)_{n\ge 0}$ be an admissible control of $C_1$ (which then belongs to $\bf{A}$) and $v^*(\s)=:(\bar{\tau}_n)_{n\geq 0}$ be the optimal response strategy of $C_2$ which we define below. Indeed let $(\rho_n,\theta_n)_{n\ge 0}$ be the sequence defined as follows:  $\rho_0=0$, $\theta_0=i$ and for $n\ge 1$
\[ \rho_0=0,\; \theta_0=i,\; \mbox{ and for } n\geq 1,\] 
\begin{align}\label{v*u}
\rho_n=\sigma_{\check r_n}\wedge\tilde{\tau}_n,\; \theta_n=
\left\lbrace\begin{array}{l}
1+\theta_{n-1} \quad    \mbox{ if } \;\theta_{n-1}\leq p-1\\
1        \qquad\qquad\mbox{ if } \;\theta_{n-1}=p
\end{array}
\right.
\end{align} 
where $$\tilde{\tau}_n:=\tau_{\rho_{n-1}}^{\theta_{n-1}}:=\inf\left\lbrace s\geq \rho_{n-1}, Y_s^{\theta_{n-1}}=Y_s^{\theta_n}+\bar{g}_{\theta_{n-1}\theta_n}(s) \right\rbrace\wedge T \,\,(\mx{according to } \eqref{eq:Implicit-Obstacle-Time-Selector})$$ and $\check r_n$ is defined by $\check r_0=0, \check r_1=1,$ for $n\geq 2,$
\[ \check r_n=\check r_{n-1}+1_{\{\sigma_{\check r_{n-1}}\leq\tilde{\tau}_{n-1} \} }.\]
Next let $\check v$ be the piecewise process defined by:
$\check v_s=0 \mx{ for }s<\rho_1 \mx{ and for } 
n\ge 1, s\in [\rho_n,\rho_{n+1})$,
$$
 \check v_s=\left\{\begin{array}{l}
1+\check v_{\rho_n-}\mx {if }\rho_n=\tilde{\tau}_n<\sigma_{\check r_n}\\\\
\check v_{\rho_n-}\mx{ if }\rho_n=\sigma_{\check r_n}\le \tilde{\tau}_n
\end{array}
\right.
$$where $\check v_{\rho_n-}=\lim_{s\nearrow \rho_n}\check v_s$. 
Now the stopping times $\bar \t_n$, $n\ge 0$, are defined as follows:
\begin{equation}\label{v*u3}
\bar \t_0=0 \mx{ and for $n\ge 1$, }\bar\t_n=\inf\{s\ge \bar \tau_{n-1},\,\,
\check v_{s}>\check v_{s-}\}\wedge T 
\end{equation}
where $\check v_{s-}=\lim_{r\nearrow s}\check v_r$.

Next we are going to define the notion of optimal responce $u^*(v)=(\bar \s_n)_{n\ge 0}$ of $C_1$ to an admissible control $v=(\t_n)_{n\ge 0}$ of the second player $C_2$. Indeed let $(\rho_n,\theta_n)_{n\ge 0}$ be the sequence defined as follows:  $\rho_0=0$, $\theta_0=i$ and for $n\ge 1$
\[ \rho_0=0,\; \theta_0=i,\; \mbox{ and for } n\geq 1,\] 
\begin{align}\label{u*v}
\rho_n=\tilde \sigma_{n}\wedge{\tau}_{\check s_n},\; \theta_n=
\left\lbrace\begin{array}{l}
1+\theta_{n-1} \quad    \mbox{ if } \;\theta_{n-1}\leq p-1\\
1        \qquad\qquad\mbox{ if } \;\theta_{n-1}=p
\end{array}
\right.
\end{align} 
where $$\tilde{\s}_n:=\s_{\rho_{n-1}}^{\theta_{n-1}}:=\inf\left\lbrace s\geq \rho_{n-1}, Y_s^{\theta_{n-1}}=Y_s^{\theta_n}-\underline{g}_{\theta_{n-1}\theta_n}(s) \right\rbrace\wedge T\,\,(\mx{according to } \eqref{eq:Implicit-Obstacle-Time-Selector})$$ and $\check s_n$ is defined by $\check s_0=0, \check s_1=1,$ for $n\geq 2,$
\[ \check s_n=\check s_{n-1}+1_{\{\tilde\sigma_{n-1}>{\tau}_{ \check s_{n-1}} \} }.\]
Next let $\check u$ be the piecewise process defined by:
$\check u_s=0 \mx{ for }s<\rho_1 \mx{ and for } 
n\ge 1, s\in [\rho_n,\rho_{n+1})$,
$$
 \check u_s=\left\{\begin{array}{l}
1+\check u_{\rho_n-}\mx {if }\rho_n=\tilde \sigma_{n}\le {\tau}_{\check s_n}
\\\\
\check u_{\rho_n-}\mx{ if }\rho_n={\tau}_{\check s_n}<\tilde \sigma_{n}
\end{array}
\right.
$$where $\check u_{\rho_n-}=\lim_{s\nearrow \rho_n}\check u_s$. 
Now the stopping times $\bar \s_n$, $n\ge 0$, are defined as follows:
\begin{equation}\lb{v*u4}
\bar \s_0=0 \mx{ and for $n\ge 1$, }\bar\s_n=\inf\{s\ge \bar \s_{n-1},\,\,
\check u_{s}>\check s_{s-}\}\wedge T 
\end{equation}
where $\check u_{s-}=\lim_{r\nearrow s}\check u_r$. We then have:
\ms

\begin{proposition}\label{u*v} Assume [H1], [H2], [H3] and  $\hz$. Then for any $u\in\mathcal{A}$ and $v\in\mathcal{B}$, we have:\\
i)  $u^*(v)\in\mathcal{A}, v^*(u)\in\mathcal{B}$;\\
ii)
\begin{align}\label{u*v0}\; J_i\left(\theta(u,v^*(u))\right)\leq Y_0^i\leq J_i\left( \theta(u^*(v),v) \right).
\end{align}
\end{proposition}

\begin{proof}
\noindent i) In order to show $u^*(v)\in\mathcal{A}$, when $v=(\t_n)_{n\ge 0}\in \cb$, we need to prove that $u^*(v)=(\bar \s_n)_{n\ge 0}$ is admissible and $\mathbb{E}\left[(A_T^{u^*(v)})^2\right]<\infty $.\\

Indeed if $u^*(v)=(\bar \s_n)_{n\ge 0}$ is not admissible then there would exist a loop $\{j,j+1,...,p-1,p,1,...,j-1,j\}$ which is visited infinitley many times, i.e.,  
$$\begin{array}{l}
\p[\omega, \exists \mx{ a subsequence }(n_\ell)_{\ell\ge 0} \mx{ such that }Y^j_{\bar \s_{n_\ell}}=Y^{j+1}_{\bar \s_{n_\ell}}-
\underline{g}_{j,j+1}(\bar \s_{n_\ell},X^{0,x}_{\bar \s_{n_\ell}}), \dots,\\ \qquad\qquad Y^{j-1}_{\bar \s_{n_\ell+p-1}}=Y^{j}_{\bar \s_{n_\ell+p-1}}-\underline{g}_{j-1,j}(\bar \s_{n_\ell+p-1},X^{0,x}_{\bar \s_{n_\ell+p-1}}), \forall \ell \ge 0]>0.\end{array}
$$ Next let us set $\eta=\lim_ {\ell\rightarrow \infty}\bar \s_{n_\ell}$. 
Take the limit in the previous equalities yield:
$$
\p[\underline g_{j,j+1}(\eta,X^{0,x}_\eta)+...+
\underline g_{p-1,p}(\eta,X^{0,x}_\eta)+
\underline g_{p,1}(\eta,X^{0,x}_\eta)+...+
\underline g_{j-1,j}(\eta,X^{0,x}_\eta)=0]>0.
$$
But this is contradictory with the non free loop property \eqref{nflp3}.  

Next let us show that $\mathbb{E}\left[(A_T^{u^*(v)})^2\right]<\infty $. Proceeding similarly as in the proof of Proposition \ref{prop34}, in the interval $[0,\rho_1]$ we have
\begin{align}\label{u*v1}
Y_0^i=h^i(X_T^{0,x})1_{(\rho_1=T)}+Y_{\rho_1}1_{(\rho_1<T)}+\int_0^{\rho_1}f^i(r,X_r^{0,x})dr-\int_0^{\rho_1}Z_r^idB_r+\int_0^{\rho_1}dK_r^{i,+}-\int_0^{\rho_1}dK_r^{i,-}.
\end{align}
Note that the minimizer $C_2$'s control $v=(\t_n)_{n\ge 0}$ is not necessarily optimal, then $\int_0^{\rho_1}dK_r^{i,-}\geq 0$ and we know that for any $s\in [0,T]$, $Y_s^i\leq Y_s^{i+1}+\bar{g}_{i,i+1}(s,X_s^{0,x})$. On the other hand, since $\rho_1=\bar \s_1\wedge \t_{\check s_1}$ then $\int_0^{\rho_1}dK_r^{i,+}= 0$. It follows that:
\begin{align}
Y_0^i	&\leq h^i(X_T^{0,x})1_{(\rho_1=T)}+Y_{\rho_1}^i1_{(\rho_1<T)}+\int_0^{\rho_1}f^i(r,X_r^{0,x})dr-\int_0^{\rho_1}Z_r^idB_r\nonumber \\
&\le h^i(X_T^{0,x})1_{(\rho_1=T)}+1_{(\rho_1<T)}(Y^i_{\bar \s_1}1_{\{\rho_1=\bar \s_1\}}+
Y^i_{\t_{\check s_1}}1_{\{\rho_1=\t_{\check s_1}\}})+\int_0^{\rho_1}f^i(r,X_r^{0,x})dr-\int_0^{\rho_1}Z_r^idB_r\nonumber\\
&\leq h^{\theta_0}(X_T^{0,x})1_{(\rho_1=T)}+Y_{\rho_1}^{\theta_1}1_{(\rho_1<T)}-\left[ \underline{g}_{\theta_0\theta_1}(\rho_1,X_{\rho_1}^{0,x})1_{(\rho_1=\bar{\sigma}_1<T)}-\bar{g}_{\theta_0\theta_1}(\rho_1,X_{\rho_1}^{0,x})1_{(\rho_1= \tau_{\check s_1}<T)} \right]\nonumber\\
&\qq\qq +\int_0^{\rho_1}f^{\theta_0}(r,X_r^{0,x})dr-\int_0^{\rho_1}Z_r^{\theta_0}dB_r.\label{u*v2}
\end{align}
Proceeding then iteratively for $n=1,2,...,N$ to obtain
\begin{align}\label{eq:321}
Y_0^i &\leq \displaystyle\sum_{n=1}^N h^{\theta_{n-1}}(X_T^{0,x})1_{(\rho_{n-1}<T,\rho_n=T)}+Y_{\rho_N}^{\theta_N}1_{(\rho_N<T)}+\int_0^{\rho_N}f^{\theta(u^*(v),v)_r}(r,X_r^{0,x})dr-\int_0^{\rho_N}Z_r^{\theta(u^*(v),v)_r}dB_r\nonumber\\
&-\sum_{n=1}^N\left[ \underline{g}_{\theta_{n-1}\theta_n}(\rho_n,X_{\rho_n}^{0,x})1_{(\rho_n=\bar{\sigma}_n<T)}-\bar{g}_{\theta_{n-1}\theta_n}(\rho_n,X_{\rho_n}^{0,x})1_{(\rho_n=\tau_{\check s_n}<T)} \right].\end{align}
Then we have
\begin{align}
A_{\rho_N}^{u^*(v)}&\leq \displaystyle\sum_{n=1}^N h^{\theta_{n-1}}(X_T^{0,x})1_{(\rho_{n-1}<T,\rho_n=T)}+Y_{\rho_N}^{\theta_N}1_{(\rho_N<T)}+\int_0^{\rho_N}f^{\theta(u^*(v),v)_r}(r,X_r^{0,x})dr\nonumber\\
&-\int_0^{\rho_N}Z_r^{\theta(u^*(v),v)_r}dB_r-Y_0^i+B_{\rho_N}^v.\label{u*v4}
\end{align}
Next as $v\in \cb$ and since $\hz$, taking the squares of each hand-side of the previous inequality to deduce that:
$$
\E[(A_{\rho_N}^{u^*(v)})^2]\le C
$$
for some real constant $C$. Finally to conclude it is enough to use Fatou's Lemma since $\rho_N \rightarrow T$ as $N\rightarrow \infty$.

In the same way we show that $v^*(u)$ belongs to $\cb$ when $u$ belongs to $\ca$.
\ms

\nd iii) Let $v\in \cb$. Going back to \eqref{eq:321}, take expectation to obtain: 
$$
Y^i_0=\mathbb{E}[Y_0^i]\le \mathbb {E}[\sum_{n=1}^N h^{\theta_{n-1}}(X_T^{0,x})1_{(\rho_{n-1}<T,\rho_n=T)}+Y_{\rho_n}^{\theta_n}1_{(\rho_n<T)}+\int_0^{\rho_N}f^{\theta(u^*(v),v)_r}(r,X_r^{0,x})dr-C_N^{\theta(u^*(v),v)}].
$$
As $v\in \cb$ and $u^*(v)\in \ca$, then for any $N\ge 1$, $|C_N^{\theta(u^*(v),v)}|\le A^{u^*(v)}_T+B^v_T\in L^2(d\p)$. Take now the limit w.r.t $N$ in the right-hand side of the previous inequality and using dominated convergence theorem to deduce that:
$$
Y^i_0\le \E[h^{\theta_T(u^*(v),v)}(X_T^{0,x})+\int_0^{T}f^{\theta(u^*(v),v)_r}(r,X_r^{0,x})dr-C_\infty^{\theta(u^*(v),v)}]=J_i(\theta(u^*(v),v)), \,\,\fr v\in \cb.
$$
The other inequality is shown in a similar fashion.
\end{proof}
As a by-product we obtain the following result:
\begin{theorem}\label{thm36} Assume [H1], [H2],  [H3] and $\hz$. Then for any $i=1,...,m$, 
$$
Y^i_0=\sup_{u\in \ca}\inf_{v\in \cb}J_i(\theta(u,v))=\inf_{v\in \cb}\sup_{u\in \ca}J_i(\theta(u,v)).
$$
\end{theorem}
\begin{proof}By \eqref{u*v0}, we know that for any $u\in \ca$ and $v\in \cb$,
$$
J_i\left(\theta(u,v^*(u))\right)\leq Y_0^i\leq J_i\left( \theta(u^*(v),v) \right).
$$Therefore
$$
\sup_{u\in \ca}J_i\left(\theta(u,v^*(u))\right)\leq Y_0^i\leq \inf_{v\in \cb}J_i\left( \theta(u^*(v),v) \right).
$$
As when $u\in \ca$ (resp. $v\in \cb$), $v^*(u)\in \cb$ (resp. $u^*(v)\in \ca$) then 
$$
\inf_{v\in \cb}\sup_{u\in \ca}J_i\left(\theta(u,v))\right)\le \sup_{u\in \ca}J_i\left(\theta(u,v^*(u))\right)\leq Y_0^i\leq \inf_{v\in \cb}J_i\left( \theta(u^*(v),v) \right)\leq \sup_{u\in\ca}\inf_{v\in \cb}J_i\left( \theta(u,v) )\right)
$$which implies the desired result since the right-hand side is smaller than the left-hand one. \end{proof}
\begin{remark}\lb{jeuetendu}Note that we have also the following equalities: For any $\ig$, 
\begin{align*}
Y_0^i&=\sup_{u\in \ca}J_i\left(\theta(u,v^*(u))\right)=\inf_{v\in \cb}J_i\left( \theta(u^*(v),v) \right)\\&=\inf_{v\in \mbb }\sup_{u\in \ca}J_i\left(\theta(u,v(u))\right)=\sup_{u\in\mba }\inf_{v\in \cb}J_i\left( \theta(u(v),v) \right).
\end{align*}
Actually let us show the fourth equality. Let $\tilde u(.)\in \mba$. Then
$$
\inf_{v\in \cb}J_i\left( \theta(\tilde u(v),v) \right)\le \inf_{v\in \cb}\sup_{u\in \ca}J_i\left( \theta(u,v) \right)=Y^i_0=\inf_{v\in \cb}J_i\left( \theta(u^*(v),v) \right)
$$which implies the fourth equality since $u^*(.)\in \mba$. The third one is proved similarly.\qed
\end{remark}

As mentioned before, the bottleneck for proving the existence of a value for the zero-sum switching game over square integrable controls is the square integrability of $(Z^i)_{i\in\G}$. The point now is whether or not it is possible to characterize $Y^i$ as the value of the zero-sum switching game without assuming the square integrability of $(Z^i)_{\ig}$. At least at the cost of adding some supplementary conditions on the data of the game. The answer is affirmative if we require assumption [H4]  on the switching costs. Note that this assumption [H4]  is satisfied if $\bar{g}_{i,i+1}$ and 
$\underline{g}_{i,i+1}$, $i=1,\dots,p$, do not depend on $x$ and are non decreasing w.r.t $t$ (e.g. they are constant). 
\ms

We then have: 

\begin{theorem}\label{game-value-without-z}
Assume [H1], [H2],[H3] and [H4]. Then for any $i\in\G$, $$
Y^i_0=\sup_{u\in \ca}\inf_{v\in \cb}J_i(\theta(u,v))=\inf_{v\in \cb}\sup_{u\in \ca}J_i(\theta(u,v)).
$$
\end{theorem}
\begin{proof}
First recall the processes $(Y^i,Z^i,K^{i,\pm})_{i\in \G}$ that satisfy: For any $\ig$ and $s\leq T$,
\begin{equation}\label{yigameox}
\left\lbrace
\begin{array}{l}
Y^i \in \sd ; K^{i,\pm} \in \ad \mbox{ and }Z^i \in \hd; \\ 
Y_s^i=h^i(X_T^{\ox})+\displaystyle\int_s^Tf^i(r,X_r^{\ox})dr-\int_s^TZ_rdB_r+K_T^{i,+}-K_s^{i,+}-(K_T^{i,-}-K_s^{i,-});\\
L^i(\vec Y)_s\leq Y_s^i\leq U^i(\vec Y)_s;\\
\int_0^T (Y^i_s-L^i(\vec Y)_s)dK_s^{i,+}=0\;\mbox{and} \; \int_0^T
(Y^i_s-U^i(\vec Y)_s)dK_s^{i,-}=0
\end{array}
\right.
\end{equation}where for $s\le T$, $L^i(\vec
Y)_s:=Y^{i+1}_s-\underline{g}_{\iplus1}(s,X_s^{\ox})$ and $U^i(\vec
Y)_s:=Y^{i+1}_s+\bar{g}_{\iplus1}(s,X_s^{\ox})$.

Next for any $k\geq 0$, let us define the following stopping time:
\begin{align}\lb{gammak}
\gamma_k:=\inf\lbrace s\geq0, \int_0^s \{\sum_{i=1,m}|Z_r^i|^2\}dr\geq k \rbrace\wedge T.
\end{align}
First note that the sequence $(\gamma_k)_{k\ge 1}$ is increasing, of stationnary type and converges to $T$. Next we have $\int_0^{\gamma_k}|Z_r^i|^2dr\le k$,  which means that the processes $(Z_s^i\ind_{\{s\le \gamma_k\}})_{s\le T}$ belong to ${\cal H}^{2,d}$. Let us now define $(\bar Y^i,\bar Z^i,\bar K^{i,\pm})_{i\in \G}$ as follows: For all $i\in \G$ and $s\le T$, 
\begin{align}\lb{paraloc}
\bar{Y}_s^i:={Y}_{s\wedge \gamma_k}^i, 
\bar{Z}_s^i=Z_s^i\ind_{\{s\le \gamma_k\}},
\bar{K}_s^{i,+}:={K}^{i,+}_{s\wedge \gamma_k} \mx{ and }
\bar{K}_s^{i,-}:={K}^{i,-}_{s\wedge \gamma_k}.
\end{align}
Thus the family $\left( \bar{Y}^i,\bar{Z}^i,\bar{K}^{i,+},\bar{K}^{i,-}\right)_{i\in\G}$ is the solution of the following system: $\forall i\in\G,$
\begin{equation}\label{croig1}
\left\lbrace
\begin{array}{l}
\mx{i)} \,\,\bar Y^i\in\mathcal{S}^2, \bar Z^i\in\mathcal{H}^{2,d}, \bar K^{i,\pm}\in\ad ;\\
\mx{ii)} \,\, \bar{Y}_s^i={Y}_{\gamma_k}^i+\displaystyle\int_s^T 1_{(r\leq \gamma_k)}f^i(r,X_r^{\ox})dr-\int_s^T \bar{Z}_r^idB_r+\bar{K}_{T}^{i,+}-\bar{K}_s^{i,+}-(\bar{K}_{T}^{i,-}-\bar{K}_s^{i,-}), \fr s\le T;\\
\mx{iii)} \,\,\bar{Y}_s^{i+1}-{\underline{g}}_{\iplus1}(s,X_s^{\ox})\leq \bar{Y}_s^i \leq \bar{Y}_s^{i+1}+{\bar{g}}_{\iplus1}(s,X_s^{\ox}), \fr s\le T;\\
\mx{iv)} \,\,\int_0^T \left( \bar{Y}_s^i-{L}^i(\vec{\bar{Y}})_s\right)d\bar K_s^{i,+}=0\mx{ and }
\int_0^T \left( \bar{Y}_s^i-{U}^i(\vec{\bar{Y}})_s\right)d\bar{K}_s^{i,-}=0
\end{array}
\right.
\end{equation}
where ${U}^i(\vec{\bar{Y}})$ and ${L}^i(\vec{\bar{Y}})$ are defined as in \eqref{yigameox}. Let us amphazise that here we need the assumption [H4] to show the inequalities in point iii) which actually hold true. Indeed for $s\le \gamma_k$, the inequalities hold true by the definition of the processes $\left( \bar{Y}^i,\bar{Z}^i,\bar{K}^{i,+},\bar{K}^{i,-}\right)_{i\in\G}$  and \eqref{yigameox}. If $s>\gamma_k$, by [H4] we have, 
\begin{align*}
&\bar{Y}_s^{i+1}-{\underline{g}}_{\iplus1}(s,X_s^{\ox})=
{Y}_{\gamma_k}^{i+1}-{\underline{g}}_{\iplus1}(s,X_s^{\ox})
\le {Y}_{\gamma_k}^{i+1}-{\underline{g}}_{\iplus1}({\gamma_k},X_{\gamma_k}^{\ox})\\&
\qq\qq\qq \leq {Y}_{\gamma_k}^{i}=\bar{Y}_s^i\leq {Y}_{\gamma_k}^{i+1}+{\bar{g}}_{\iplus1}(s,X_s^{\ox})=\bar{Y}_{s}^{i+1}+{\bar{g}}_{\iplus1}(s,X_s^{\ox}).
\end{align*}
On the other hand, by definition of $\bar K^{\pm,i}$ and $\bar Y^{i}$, $\ig$, we have 
$$\begin{array}{l}
\int_0^T \left( \bar{Y}_s^i-{L}^i(\vec{\bar{Y}})_s\right)d\bar K_s^{i,+}=
\int_0^{\gamma_k}({Y}_s^i-{L}^i(\vec{{Y}})_s)dK_s^{i,+}=0.
\end{array}
$$

Similarly we have also $ 
\int_0^T ( \bar{Y}_s^i-{U}^i(\vec{\bar{Y}})_s)d\bar{K}_s^{i,-}=0$. Therefore the processes $(\bar Y^i,\bar Z^i,\bar K^{i,\pm})_{i\in \G}$ verify \eqref{croig1}.

Now using the result of Theorem \ref{thm36}, we obtain: For any $\ig$, 
$$
Y^i_0=\bar Y^i_0=\sup_{u\in \ca}\inf_{v\in \cb}J_i^k(\theta(u,v))=\inf_{v\in \cb}\sup_{u\in \ca}J_i^k(\theta(u,v)).
$$
with
$$
J_i^k(\theta(u,v))=\mathbb{E}\left[ Y_{\gamma_k}^{\theta(u,v)_T}+\int_0^T1_{(r\leq \gamma_k)}
f^{\theta(u,v)_r}(r,X_r^{\ox}) dr -C_\infty^{\theta(u,v)}\right]
$$
where $\tuv$ is the coupling of the pair $(u,v)$ of controls and 
$C^{\theta(u,v)}_\infty:=\lim_{n\rightarrow \infty}C^{u,v}_N.$
Next let us set:
$$
\breve Y^i_0=\sup_{u\in \ca}\inf_{v\in \cb}J_i(\theta(u,v))\mx{ and }\ \tilde Y^i_0=\inf_{v\in \cb}\sup_{u\in \ca}J_i(\theta(u,v)).
$$Therefore 
\begin{align*}
|\breve Y^i_0-Y^i_0|=&|\sup_{u\in \ca}\inf_{v\in \cb}J_i(\theta(u,v))-\sup_{u\in \ca}\inf_{v\in \cb}J_i^k(\theta(u,v))|\\&\le \sup_{(u,v)\in \ca \times \cb}\E[|Y_{\gamma_k}^{\theta(u,v)_T}-h^{\theta(u,v)_T}(X_T^{\ox})|\\&\qq\qq\qq+\int_0^T|1_{(r\leq \gamma_k)}
f^{\theta(u,v)_r}(r,X_r^{\ox}) dr-f^{\theta(u,v)_r}(r,X_r^{\ox})|dr]
\\&\le \E[\sum_{i=1,m}|Y_{\gamma_k}^{i}-h^{i}(X_T^{\ox})|+
\int_{\gamma_k}^T\sum_{i=1,m}|
f^{i}(r,X_r^{\ox})|dr].
\end{align*}
But the right-hand side converges to $0$ as $k\rightarrow \infty$. Therefore 
$$
\breve Y^i_0=Y^i_0=\sup_{u\in \ca}\inf_{v\in \cb}J_i(\theta(u,v)).$$
In the same way we obtain also that 
$$
\tilde Y^i_0=Y^i_0=\inf_{v\in \cb}\sup_{u\in \ca}J_i(\theta(u,v)).$$
It follows that 
$$
Y^i_0=\sup_{u\in \ca}\inf_{v\in \cb}J_i(\theta(u,v))=\inf_{v\in \cb}\sup_{u\in \ca}J_i(\theta(u,v)).$$
Thus the zero-sum switching game has a value on square integrable controls which is equal to $Y^i_0$.\end{proof}
\subsection{Value of the zerosum switching game on integrable admissible controls}
In this part, we are not going to assume the square integrability of $\z$ neither [H4] and show that the relation \eqref{valint} holds true and this common value is equal to $Y^i_0$, where $(Y^i,Z^i,K^{i,\pm})_{i\in \G}$ is the solution of system \eqref{yigame}. Actually we have the following result:
\begin{theorem}\label{value sans [H4]}Assume [H1], [H2] and [H3]. Then for any $\ig$, 
\begin{equation}\lb{valintthm42}
Y^i_0=\inf_{v\in  \cbi}\sup_{u\in  \cai}J_i(\tuv)=\sup_{u\in  \cai }\inf_{v\in  \cbi}J_i(\tuv).
\end{equation}
\end{theorem}
\begin{proof}Let $u=(\s_n)_{n\ge 0}$ and $v=(\t_n)_{n\ge 0}$ be two admissible controls which belong to $\cai$ and $\cbi$ respectively. 
Next recall the optimal responses $u^*(v)=(\bar \s_n)_{n\ge 0}$ and $v^*(u)=(\bar \t_n)_{n\ge 0}$ defined in \eqref{v*u4} and \eqref{v*u3} respectively. First note that, as shown in Proposition \ref{u*v}, the controls $u^*(v)$ and $v^*(u)$ are admissible. Let us now show $u^*(v)$ belongs to $\cai$. A similar procedure will show that $v^*(u)$ belongs to $\cbi$. 

Indeed for $k\ge 1$, recall the stopping time $\gamma_k$ defined in \eqref{gammak} and the sequences $(\rho_n)_{n\ge 0}$ and $(\theta_n)_{n\ge 0}$ defined in \eqref{u*v}. Next for $k\ge 1$, let us define: $\fr$ $n\ge 0$,
$$\rho_n^k=\rho_n 1_{\{\rho_n< \gamma_k\}}+T1_{\{\rho_n\ge \gamma_k\}} \mx{ and }\theta_n^k=\theta_n1_{\{\rho_n< \gamma_k\}}+\theta_{n_k}1_{\{\rho_n\ge \gamma_k\}}
$$where $n_k=\inf\{n\ge 0, \rho_n\geq \gamma_k\}-1$. Note that $\rho_n^k$ is a stopping time and $\{\rho_n^k<T\}=\{\rho_n<\gamma_k\}$. The sequences 
$(\rho_n^k)_{n\ge 0}$ and $(\theta_n^k)_{n\ge 0}$ constitute the fact that we freeze the actions of the controllers when $\gamma_k$ is reached. Next going back to the system of equations \eqref{yigame} satisfied by the family $\left({Y}^i,{Z}^i,{K}^{i,+},{K}^{i,-}\right)_{\ig}$ and as in \eqref{u*v2} we have: 
\begin{align}\label{u*v1}
Y_0^i&=h^i(X_T^{0,x})1_{(\rho^k_1=T)}+Y^i_{\rho^k_1}1_{(\rho^k_1<T)}+\int_0^{\rho^k_1}f^i(r,X_r^{0,x})dr-\int_0^{\rho^k_1}Z_r^idB_r+\underbrace{\int_0^{\rho^k_1}dK_r^{i,+}}_{=0}-\int_0^{\rho^k_1}dK_r^{i,-}\nn\\
&\le h^i(X_T^{0,x})1_{(\rho^k_1=T)}+Y^i_{\rho^k_1}1_{(\rho^k_1<T)}+\int_0^{\rho^k_1}f^i(r,X_r^{0,x})dr-\int_0^{\rho^k_1}Z_r^idB_r.
\end{align}
But $\{\rho^k_1<T\}=\{\rho_1<\gamma_k\}$. Therefore $$
Y^i_{\rho^k_1}1_{(\rho^k_1<T)}=
Y^i_{\rho_1}1_{(\rho_1<\gamma_k)}=(Y^i_{\bar \s_1}1_{\{\rho_1=\bar \s_1\}}+
Y^i_{\t_{\check s_1}}1_{\{\rho_1=\t_{\check s_1}\}})1_{(\rho_1<\gamma_k)}$$
and then 
\begin{align}\label{u*v1}
Y_0^i
&\le h^i(X_T^{0,x})1_{(\rho^k_1=T)}+(Y^i_{\bar \s_1}1_{\{\rho_1=\bar \s_1\}}+
Y^i_{\t_{\check s_1}}1_{\{\rho_1=\t_{\check s_1}\}})1_{(\rho_1<\gamma_k)}+\int_0^{\rho^k_1}f^i(r,X_r^{0,x})dr-\int_0^{\rho^k_1}Z_r^idB_r
\end{align}
But for any $s\in [0,T]$, $Y_s^i\leq Y_s^{i+1}+\bar{g}_{i,i+1}(s,X_s^{0,x})$ and \\
$$Y^i_{\bar \s_1}1_{\{\rho_1=\bar \s_1\}}1_{(\rho_1<\gamma_k)}=
(Y^{i+1}_{\bar \s_1}-\underline g_{i,i+1}(\bar\s_1,X_{\bar\s_1}^{0,x}))1_{\{\rho_1=\bar \s_1\}}1_{(\rho_1<\gamma_k)}.
$$Plug now this in \eqref{u*v1} to obtain:
\begin{align}
Y_0^i	&\leq h^{i}(X_T^{0,x})1_{(\rho^k_1=T)}+
(Y^{i+1}_{\bar \s_1}-\underline g_{i,i+1}(\bar\s_1,X_{\bar\s_1}^{0,x}))1_{\{\rho_1=\bar \s_1\}}
1_{(\rho_1<\gamma_k)}\nn\\&+
(Y_{\t_{\check s_1}}^{i+1}+\bar{g}_{i,i+1}({\t_{\check s_1}},X_{\t_{\check s_1}}^{0,x}))
1_{\{\rho_1=\t_{\check s_1}\}}1_{(\rho_1<\gamma_k)}
+\int_0^{\rho^k_1}f^{\theta_0}(r,X_r^{0,x})dr-\int_0^{\rho^k_1}Z_r^{\theta_0}dB_r.\label{u*v2}
\end{align}
As  
$$
(Y^{i+1}_{\bar \s_1}1_{\{\rho_1=\bar \s_1\}}
+
Y_{\t_{\check s_1}}^{i+1}
1_{\{\rho_1=\t_{\check s_1}\}})1_{(\rho_1<\gamma_k)}=Y^{\theta_1^k}_{\rho_1^k}1_{(\rho^k_1<\gamma_k)}
$$
and \begin{align}
(-\underline g_{i,i+1}(\bar\s_1,X_{\bar\s_1}^{0,x})1_{\{\rho_1=\bar \s_1\}}
&+
\bar{g}_{i,i+1}({\t_{\check s_1}},X_{\t_{\check s_1}}^{0,x})
1_{\{\rho_1=\t_{\check s_1}\}})1_{(\rho_1<\gamma_k)}\nn\\&=
(-\underline g_{\theta_0,\theta_1^k}(\bar\s_1,X_{\bar\s_1}^{0,x})1_{\{\rho^k_1=\bar \s_1\}}
+
\bar{g}_{\theta_0,\theta_1^k}({\t_{\check s_1}},X_{\t_{\check s_1}}^{0,x})
1_{\{\rho^k_1=\t_{\check s_1}\}})1_{(\rho^k_1<\gamma_k)}\nn
\end{align}
then from \eqref{u*v2}, we obtain:
\begin{align}
Y_0^i	&\leq h^{\theta_0}(X_T^{0,x})1_{(\rho^k_1=T)}+
Y^{\theta_1^k}_{\rho_1^k}1_{(\rho^k_1<\gamma_k)}+
(-\underline g_{\theta_0,\theta_1^k}(\bar\s_1,X_{\bar\s_1}^{0,x})1_{\{\rho^k_1=\bar \s_1\}}
+
\bar{g}_{\theta_0,\theta_1^k}({\t_{\check s_1}},X_{\t_{\check s_1}}^{0,x})
1_{\{\rho^k_1=\t_{\check s_1}\}})1_{(\rho^k_1<\gamma_k)}\nn\\&
+\int_0^{\rho^k_1}f^{\theta_0}(r,X_r^{0,x})dr-\int_0^{\rho^k_1}Z_r^{\theta_0}dB_r.\label{u*v2}
\end{align}
But we can do the same with $Y^{\theta_1^k}_{\rho_1^k}1_{(\rho^k_1<\gamma_k)}$ to obtain:
\begin{align}
Y^{\theta_1^k}_{\rho_1^k}1_{(\rho^k_1<\gamma_k)}	&\leq h^{\theta_2^k}(X_T^{0,x})1_{(\rho^k_1<\gamma_k,\rho^k_2=T)}+
Y^{\theta_2^k}_{\rho_2^k}1_{(\rho^k_2<\gamma_k)}+\nn\\&
\qq\qq (-\underline g_{\theta_1^k,\theta_2^k}(\bar\s_2,X_{\bar\s_2}^{0,x})1_{\{\rho^k_2=\bar \s_2\}}
+
\bar{g}_{\theta_1^k,\theta_2^k}({\t_{\check s_2}},X_{\t_{\check s_2}}^{0,x})
1_{\{\rho^k_2=\t_{\check s_2}\}})1_{(\rho^k_2<\gamma_k)}\nn\\&
\qq\qq +\int_{\rho^k_1}^{\rho^k_2}f^{\theta_2^k}(r,X_r^{0,x})dr-\int_{\rho^k_1}^{\rho^k_2}Z_r^{\theta_2^k}dB_r.\label{u*v5}
\end{align}
Plug now \eqref{u*v5} in \eqref{u*v2} and repeat this procedure $N$ times to obtain:
\begin{align}\label{eq:321}
Y_0^i &\leq \displaystyle\sum_{n=1}^N h^{\theta^k_{n-1}}(X_T^{0,x})1_{(\rho^k_{n-1}<T,\rho^k_n=T)}+Y_{\rho^k_N}^{\theta^k_N}1_{(\rho^k_N<\gamma_k)}+\int_0^{\rho^k_N}f^{\theta(u^*(v),v)_r}(r,X_r^{0,x})dr\nonumber\\
&-\int_0^{\rho^k_N}Z_r^{\theta(u^*(v),v)_r}dB_r\underbrace{-\sum_{n=1}^N\left[ \underline{g}_{\theta^k_{n-1}\theta^k_n}(\rho^k_n,X_{\rho^k_n}^{0,x})1_{(\rho^k_n=\bar{\sigma}_n<\gamma_k)}-\bar{g}_{\theta^k_{n-1}\theta^k_n}(\rho^k_n,X_{\rho^k_n}^{0,x})1_{(\rho^k_n=\tau_{\check s_n}<\gamma_k)} \right]}_{A_{\rho^k_N}^{u^*(v)}-\tilde B^v_{\rho^k_N}}\end{align}
where 
$0\le \tilde B^v_{\rho^k_N}\le B^v_{\rho^k_N}$, since $C_1$ has priority when the two players decide to switch at the same time. Then take expectation in both hand-sides to obtain: \begin{align}\lb{eq336}
\E[A_{\rho^k_N}^{u^*(v)}]\le& -Y^i_0+\E[\displaystyle\sum_{n=1}^N h^{\theta^k_{n-1}}(X_T^{0,x})1_{(\rho^k_{n-1}<T,\rho^k_n=T)}+Y_{\rho^k_N}^{\theta^k_N}1_{(\rho^k_N<\gamma_k)}+\int_0^{\rho^k_N}f^{\theta(u^*(v),v)_r}(r,X_r^{0,x})dr+B^v_{\rho^k_N}].\end{align}
As $v\in \cbi$, then $\E[B^v_{\rho^k_N}]\le \E[B^v_{T}]$ and then the right hand side of \eqref{eq336} is bounded. Therefore there exists a constant $C$ such that 
$$
\E[A_{\rho^k_N}^{u^*(v)}]\le C+ \E[B^v_{T}].
$$
Finally by using twice Fatou's Lemma (w.r.t $k$ then $N$) we deduce that $
\E[A_T^{u^*(v)}]<\infty$ which is the claim.
\ms

\nd iii) Let $v\in \cbi$. Going back to \eqref{eq:321}, take expectation to obtain: 
\begin{align}\label{eq:321x}
Y_0^i &\leq \E\{\displaystyle\sum_{n=1}^N h^{\theta^k_{n-1}}(X_T^{0,x})1_{(\rho^k_{n-1}<T,\rho^k_n=T)}+Y_{\rho^k_N}^{\theta^k_N}1_{(\rho^k_N<\gamma_k)}+\int_0^{\rho^k_N}f^{\theta(u^*(v),v)_r}(r,X_r^{0,x})dr\nonumber\\
&\qq\qq-\sum_{n=1}^N\left[ \underline{g}_{\theta^k_{n-1}\theta^k_n}(\rho^k_n,X_{\rho^k_n}^{0,x})1_{(\rho^k_n=\bar{\sigma}_n<\gamma_k)}-\bar{g}_{\theta^k_{n-1}\theta^k_n}(\rho^k_n,X_{\rho^k_n}^{0,x})1_{(\rho^k_n=\tau_{\check s_n}<\gamma_k)} \right]\}.\end{align}
By taking the limit w.r.t $k$ then $N$ we obtain that 
$$
Y^i_0\leq J_i(u^*(v),v), \forall v\in \cbi.$$
In the same way as previously, for any $u\in \cai$, $v^*(u)$ belongs to $\cbi$ and 
$$Y^i_0\geq J_i(u,v^*(u)).$$
It follows that for any $u\in \cai$ and $v\in \cbi$,
$$
J_i(u,v^*(u))\le Y^i_0\leq J_i(u^*(v),v).
$$ 
Therefore
$$
\sup_{u\in \cai}J_i(u,v^*(u))\le Y^i_0\leq \inf_{v\in \cbi}J_i(u^*(v),v).
$$
As $u^*(v)$ (resp. $v^*(u)$) belongs
to $\cai$ (resp. $\cbi$) when $v\in \cbi$ (resp. $u\in \cai$), then  
$$
\underbrace{\inf_{v\in \cbi}\sup_{u\in \cai}J_i(u,v)}_{V^+}\le \sup_{u\in \cai}J_i(u,v^*(u))\le Y^i_0\leq \inf_{v\in \cbi}J_i(u^*(v),v)\le \underbrace{\sup_{u\in \cai}\inf_{v\in \cbi}J_i(u,v)}_{V^-}
$$
and the claim is proved since $V^+\ge V^-$. 
\end{proof}
\begin{remark}\label{rmq_value}a) As in Remark \ref{jeuetendu} we have also the following equalities: For any $\ig$, 
\begin{align*}
Y_0^i&=\inf_{v\in \mbb^{(1)} }\sup_{u\in \ca^{(1)}}J_i\left(\theta(u,v(u))\right)=\sup_{u\in\mbaun }\inf_{v\in \cb^{(1)}}J_i\left( \theta(u(v),v) \right).
\end{align*}
b) Let $(Y^{i,\tx},Z^{i,\tx},K^{i,\pm,\tx})_{\ig}$ be the $\mathcal{P}$-measurable processes solution of the system \eqref{yigame}. Then, as previously one can show that for any $\ig$ and $\stt$,  
$$
Y^{i,\tx}_s=\essinf_{v\in  \cbi_s}\esssup_{u\in  \cai_s}J^{t,x}(\tuv)_s=\esssup_{u\in  \cai_s}\essinf_{v\in  \cbi_s}J^{t,x}(\tuv)_s
$$
where 
$$\begin{array}{c}
J^{t,x}(\tuv)_s:=\E\{ h^{\theta(u,v)_T}(X_T^{\tx})+\int_s^T
f^{\theta(u,v)_r}(r,X_r^{\tx}) dr -C_\infty^{\theta(u,v)}|\mathcal{F}_s\}\end{array}
$$and $\cai_s$ (resp. $\cbi_s$) is the set of admissible integrable controls which start from $i$ at $s$. \qed
\end{remark}
\section{System of PDEs of min-max type with interconnected obstacles}
We are going now to deal with the problem of existence and uniqueness of a solution in viscosity sense of
the following system of PDEs of min-max type with interconnected obstacles:
\begin{equation}\label{sysedpmilie1}
\left\lbrace
\begin{array}{l}
\min\{v^i(t,x)-L^i(\vec{v})(t,x);\max\left[v^i(t,x)-U^i(\vec{v})(t,x);\right.\\
\qquad\left.-\partial_tv^i(t,x)-\mathcal{L}^X(v^i)(t,x)-f^i(t,x,(v^l(t,x))_{l\in\Gamma},\sigma(t,x)^\top D_xv^i(t,x)\right]\}=0;\\
v^i(T,x)=h^i(x)
\end{array}
\right.
\end{equation} where for any $\ig$, 
$L^i(\vec{v})(t,x):=v^{i+1}(t,x)-\underline{g}_{\iplus1}(t,x)$ and $U^i(\vec{v})(t,x):=v^{i+1}(t,x)+\overline{g}_{\iplus1}(t,x)$. Note that $f^i$ is more general w.r.t. the HJB system of \eqref{sysedpintro} since it depends also on $\vec y$ and $z^i$.
\ms

\nd The result is given in Theorem \ref{thm43}  but its proof, based on Perron's method, is postponed to Appendix. Nonetheless in this section we will introduce some notions which we need also in Section 5 when we deal with system of RBSDEs \eqref{yigameintro} or more generally \eqref{exist}.
\ms

\nd For any locally bounded deterministic function $u:[0,T]\times\mathbb{R}^k\rightarrow \mathbb{R}$, we denote by $u_*$ (resp. $u^*$) the lower semi-continuous (lsc) (resp. upper semi-continuous (usc)) envelope  of $u$ as follows: $\forall(\tx)\in\0TR$, 
\[u_*(t,x)=\liminf_{(t',x')\mapsto (t,x),t'<T} u(t',x') \mx{ and }u^*(t,x)=\limsup_{(t',x')\mapsto(t,x),t'<T}u(t',x').\]
Next for an lsc (resp. usc) function $u$ we denote by $\bar{J}^-u(t,x)$ (resp. $\bar{J}^+u(t,x)$), the parabolic limiting subjet (resp. superjet) of $u$ at $(\tx)$ (see e.g. \cite{CIL92} for the definition and more details).
\begin{definition}: Viscosity solution to \eqref{sysedpmilie1} \label{visco}

\nd Let $\vec v:=(v^i)_{\ig}$ be a p-tuple of $\R$-valued, locally bounded functions defined on $\0TR$.\\

\nd A) We say that $\vec v$ is a viscosity supersolution (resp. subsolution) of \eqref{sysedpmilie1} if for any $\ig$:

(i) $v^i_*(T,x)\geq h^i(x)$ (resp. $v^{i*}(T,x)\leq h^i(x)$), for any $x\in \R^k$ ; 

(ii) For any $(\tx)\in [0,T)\times \R^k$ and for any 
$(p,q,M)\in\bar{J}^-{v^i_*}(\tx)$ (resp. $\bar{J}^+{v^{i*}}(\tx)$), we have:
\begin{equation}
\begin{array}{l}
\min\lbrace v^i_*(t,x)-L^i(\vec{v_*})(t,x), \\
\qquad \max\lbrace -p-b(\tx).q-\dfrac{1}{2}Tr[(\sigma\sigma^\top)(\tx)M]-f^i(t,x,\vec{v_*}(\tx),\sigma^\top(t,x) q); \\
\qquad v^i_*(\tx)-U^i(\vec{v_*})(\tx)\rbrace\rbrace\,\,\ge 0
\end{array}
\end{equation}
where $\vec{v_*}=(v^i_*)_{\ig}$ (resp. 
\begin{equation}
\begin{array}{l}
\min\lbrace v^{i*}(t,x)-L^i(\vec{v^*})(t,x), \\
\qquad \max\lbrace -p-b(\tx).q-\dfrac{1}{2}Tr[(\sigma\sigma^\top)(\tx)M]-f^i(t,x,\vec{v^*}(\tx),\sigma^\top(t,x) q); \\
\qquad v^{i*}(\tx)-U^i(\vec{v^*})(\tx)\rbrace\rbrace\,\,\le 0
\end{array}
\end{equation}
where $\vec{v^*}=(v^{i*})_{\ig}$).  
\ms 

\nd B) A locally bounded function $\vec{v}=(v^i)_{\ig}$ is called a viscosity solution of \eqref{sysedpmilie1} if $(v_*^i)_{\ig}$ and $(v^{i*})_{\ig}$ are viscosity supersolution and viscosity subsolution  of \eqref{sysedpmilie1} respectively. \qed
\end{definition}

Next $(\tx)$ be fixed and let us consider the following sequence of BSDEs: $\forall m,n \in \mathbb{N}$, $\frig$,\\
\begin{equation}
\label{ymn}
\left\lbrace
\begin{array}{l}
Y^{i,m,n} \in \mathcal{S}^2, Z^{i,m,n} \in \hdd;\\
Y_s^{i,m,n}=h^i(X_T^{\tx})+\int_s^T f^{i,m,n}(r,X_r^{\tx},(Y_r^{l,m,n})_{l \in \Gamma},Z_r^{i,m,n})dr -\int_s^T Z_r^{i,m,n}dB_r, \; s\leq T;\\
Y_T^{i,m,n}=h^i(X_T^{\tx})
\end{array}
\right.
\end{equation}
where $$f^{i,m,n}(s,X_s^{\tx},\vec{y},z)=f^i(s,X_s^{\tx},\vec{y},z)+n\left\lbrace y^i-[y^{i+1}-\underline{g}_{\iplus1}(s,X_s^{\tx})]\right\rbrace^- - m\left\lbrace y^i-[y^{i+1}+\overline{g}_{\iplus1}(s,X_s^{\tx})]\right\rbrace^+.$$
As \eqref{ymn} is a classical BSDE without obstacle, thanks to the results by Pardoux-Peng \cite{EPQ97}, the solution exists and is unique. In addition there exist deterministic functions $(v^{i,m,n})_{\ig}$ (see Theorem 4.1. in \cite{EPQ97}) such that:
\begin{equation}\label{marvmn}\frst,\,\,
Y_s^{i,m,n}=v^{i,m,n}(s,X_s^{\tx}).
\end{equation}
On the other hand, we have the following properties which we collect in the following proposition. 
\begin{proposition}[see \cite{HM13},\cite{DHM15}]\label{pena}
\nd Assume that [H2], [H3] and [H5] are fulfilled. Then we have:\\
\ms

\nd a) $\mathbb{P}-\text{a.s.}, \forall s\leq T,  Y^{i,m+1,n}_s \leq Y^{i,m,n}_s \leq Y^{i,m,n+1}_s, \; \forall \ig, n,m \geq 0$, which also implies the same property for  $(v^{i,m,n})_{\ig}$, i.e. for any $(\tx)\in\0TR, \ig,$
\begin{equation}\label{inegvnm}v^{i,m+1,n}(\tx)\leq v^{i,m,n}(\tx)\leq v^{i,m,n+1}(\tx).\end{equation}
b) The sequence $((Y^{i,m,n})_{\ig})_{n \geq 0})$ (resp. $((Y^{i,m,n})_{\ig})_{m \geq 0}$) converges in $(\mathcal{S}^2)^p$ to $(\bar{Y}^{i,m})_{\ig}$ (resp. $(\underline{Y}^{i,n})_{\ig}$) which verifies the following system of reflected RBSDEs:
\begin{equation} \label{ym}
\left\lbrace
\begin{array}{l}
\overline{Y}^{i,m} \in \mathcal{S}^2, \overline{Z}^{i,m} \in \mathcal{H}^2, \overline{K}^{i,m,+} \in \mathcal{A}^2;\\
\overline{Y}_s^{i,m}=h^i(X_T^{\tx})+\int_s^T \overline{f}^{i,m}(r,X_r^{\tx},(\overline{Y}_r^{l,m})_{l \in \Gamma},\overline{Z}_r^{i,m})dr -\int_s^T \overline{Z}_r^{i,m}dB_r+\overline{K}^{i,m,+}_T-\overline{K}^{i,m,+}_s, s\leq T;\\
\overline{Y}^{i,m}_s \geq L^i((\bar{Y}_s^{l,m})_{\l\in\Gamma}), s\leq T;\\
\int_0^T (\overline{Y}^{i,m}_s-L^i((\bar{Y}_s^{l,m})_{\l\in\Gamma}))d\overline{K}^{i,m,+}_s=0
\end{array}
\right.
\end{equation}
where $$\overline{f}^{i,m}(s,X_s^{\tx},\vec{y},z^i)=f^i(s,X_s^{\tx},\vec{y},z^i)-m(y^i-[y^{i+1}+\overline{g}_{\iplus1}(s,X_s^{\tx})])^+.$$ (resp.
\begin{equation}\label{yn}
\left\lbrace
\begin{array}{l}
\underline{Y}^{i,n} \in \mathcal{S}^2, \underline{Z}^{i,n} \in \mathcal{H}^2, \underline{K}^{i,n,-} \in \mathcal{A}^2;\\
\underline{Y}_s^{i,n}=h^i(X_T^{\tx})+\int_s^T \underline{f}^{i,n}(r,X_r^{\tx},(\underline{Y}_r^{l,n})_{l \in \Gamma},\underline{Z}_r^{i,n})dr -\int_s^T \underline{Z}_r^{i,n}dB_r+\underline{K}^{i,n,-}_T-\underline{K}^{i,n,-}_s, s\leq T;\\
\underline{Y}^{i,n}_s \leq U^i((\underline{Y}_s^{l,m})_{l\in\Gamma}), s\leq T;\\
\int_0^T (\underline{Y}^{i,m}_s-U^i((\underline{Y}_s^{l,m})_{l\in\Gamma}))d\underline{K}^{i,n,-}_s=0
\end{array}
\right.
\end{equation}
where $$\underline{f}^{i,n}(s,X_s^{\tx},\vec{y},z^i)=f^i(s,X_s^{\tx},\vec{y},z^i)+n(y^i-[y^{i+1}-\underline{g}_{\iplus1}(s,X_s^{\tx})])^+.$$\\

c) There exist deterministic continuous functions $(\overline{v}^{i,m})_{\ig}$ (resp. $(\underline{v}^{i,n})_{\ig} $) such that for any $(t,x)\in\0TR, s\in[t,T],$
\begin{equation}\label{ymvm}
\overline{Y}_s^{i,m}=\overline{v}^{i,m}(s,X_s^{\tx})
\end{equation}
(resp.
\begin{equation}\label{ynvn}
\underline{Y}_s^{i,n}=\underline{v}^{i,n}(s,X_s^{\tx})).
\end{equation}
In addition for any $\ig$, the sequence $((\overline{v}^{i,m})_{m\geq 0})_{\ig}$ (resp. $((\underline{v}^{i,n})_{n\geq 0})_{\ig} )$ are decreasing w.r.t. $m$ (resp. increasing w.r.t. $n$).\\

d) $(\overline{v}^{i,m})_{\ig}$ (resp. $(\underline{v}^{i,n})$) belong to class $\Pi_g$ and is the unique viscosity solution of following system of variational inequalities with a reflected obstacle:
\begin{equation}\label{vm}
\left\lbrace
\begin{array}{l}
\min\lbrace \overline{v}^{i,m}(t,x)-L^i((\bar{v}^{l,m})_{l\in\Gamma})(\tx);\\
\; -\partial_x\overline{v}^{i,m}(\tx)-\mathcal{L}^X(\overline{v}^{i,m})(\tx)-f^{i,m}(t,x,(\overline{v}^{l,m}(\tx))_{l\in\Gamma},\sigma(\tx)^\top D_x\overline{v}^{i,m}(\tx))\rbrace=0;\\
\overline{v}^{i,m}(T,x)=h^i(x).
\end{array}
\right.
\end{equation}
(resp.
\begin{equation}\label{vn}
\left\lbrace
\begin{array}{l}
\max\lbrace \underline{v}^{i,n}(t,x)-U^i((\underline{v}^{l,m})_{l\in\Gamma})(\tx);\\
\; -\partial_x\underline{v}^{i,n}(\tx)-\mathcal{L}^X(\underline{v}^{i,n})(\tx)-f^{i,n}(t,x,(\underline{v}^{l,n}(\tx))_{l\in\Gamma},\sigma(\tx)^\top D_x\underline{v}^{i,n}(\tx))\rbrace=0;\\
\underline{v}^{i,n}(T,x)=h^i(x)).
\end{array}
\right.
\end{equation}
\end{proposition}
\begin{proof}
The proofs can be found in \cite{HM13} and \cite{DHM15} so we omit them.
\end{proof}

Next for any $\ig$ and $\txsp$, we denote by $$\overline{v}^i(\tx):=\lim_{m\to\infty}\bar{v}^{i,m}(\tx) \mx{ and }\underline{v}^i(t,x):=\lim_{n\to \infty}\underline{v}^{i,n}.$$Then from \eqref{inegvnm} we deduce that for any $\txsp$ 
$$\underline{v}^i(\tx)\leq\overline{v}^i(\tx).$$
Note that since for any $\ig$, 
$$
\underline{v}^{i,0}\le \underline{v}^i\le \bar v^i\le \bar v^{i,0}
$$
then $\underline{v}^i$ and $\bar v^i$ belong to $\Pi_g$. Additionnaly we have: 
\begin{theorem}\lb{thm43}Assume [H2],[H3] and [H5]. Then the $p$-tuple of functions $(\overline{v}^i)_{\ig}$ are continuous, of polynomial growth and unique viscosity solution, in the class $\Pi_g$, of the following systems: 
$\forall i \in \Gamma$ and $(\tx)\in\0TR,$
\begin{equation}\label{sysedpmilieu}
\left\lbrace
\begin{array}{l}
\min\{v^i(t,x)-L^i(\vec{v})(t,x);\max\left[v^i(t,x)-U^i(\vec{v})(t,x);\right.\\
\qquad\left.-\partial_tv^i(t,x)-\mathcal{L}^X(v^i)(t,x)-f^i(t,x,(v^l(t,x))_{l\in\Gamma},\sigma(t,x)^\top D_xv^i(t,x)\right]\}=0;\\
v^i(T,x)=h^i(x).
\end{array}
\right.
\end{equation}
\end{theorem}
\begin{proof}It is rather long and then postponed to Appendix.
\end{proof}
As a consequence we have the following result for the increasing scheme:
\begin{corollary}
The $p$-tuple of functions $(\underline{v}^i)_{\ig}$ is also continuous and the unique viscosity solution, in the class $\Pi_g$, of the following system of max-min type: 
$\forall i \in \Gamma$ and $(\tx)\in\0TR,$
\begin{equation}\label{sysedpmilieu2}
\left\lbrace
\begin{array}{l}
\max\{v^i(t,x)-U^i(\vec{v})(t,x);\min\left[v^i(t,x)-L^i(\vec{v})(t,x);\right.\\
\qquad\left.-\partial_tv^i(t,x)-\mathcal{L}^X(v^i)(t,x)-f^i(t,x,(v^l(t,x))_{l\in\Gamma},\sigma(t,x)^\top D_xv^i(t,x)\right]\}=0;\\
v^i(T,x)=h^i(x).
\end{array}
\right.
\end{equation}
\end{corollary}

To obtain the proof of this result it is enough to consider $(-\underline v^i)_{\ig}$ which becomes a decreasing scheme associated with $\{(-f^i(t,x,-\vec y,-z))_{\ig},(-h^i)_{\ig}, (\bar g_i)_{\ig},(\underline g_i)_{\ig}\}$, to use the previous theorem and finally a result by G.Barles (\cite{barles}, pp.18).\qed
\section{Systems of Reflected BSDEs with bilateral interconnected barriers}
First note that throughout this section we assume that $[H2], [H3]$ and $[H5]$ are fulfilled. Next recall the system of RBSDEs $(\bar{Y}^{i,m,\tx},\bar{Z}^{i,m,\tx},\bar{K}^{i,m,+,\tx})$ in Proposition \ref{pena}-b)-c) and the representation \eqref{ymvm}. As the sequence $((\bar{v}^{i,m})_{\geq 0})_{\ig}$ converges pointwise decreasingly to the continuous functions $(v^i)_{\ig}$. Then, by Dini's theorem, this convergence is uniform on compact sets of $\0TR$. Next, the uniform polynomial growths of $(v^i)_{\ig}$ and $((\bar{v}^{i,m})_{\geq 0})_{\ig}$ combined with estimate \eqref{State_3} of $\xtx$ imply that for any $\ig$,
\begin{equation} \lb{convymy}\mathbb{E}(\sup_{s\in[t,T]}|\bar{Y}_s^{i,m,\tx}-{Y}_s^{i,\tx}|^2) \to_{m\to\infty} 0
\end{equation}
where we set:
\begin{equation}
\label{markov}
\forall s\le T \mx{ and }\ig,\,\,Y_s^{i,\tx}=v^i(s\vee t,X_{s\vee t}^{\tx}).
\end{equation}
\begin{proposition}
\label{refl}
For any $(\tx)\in\0TR, s\in[t,T], \ig,$
\begin{equation}\label{refl1}
Y_s^i\leq U^i((Y_s^l)_{l\in\Gamma}):=Y_s^{i+1}+\overline{g}_{i,i+1}(s,X_s^{\tx}).
\end{equation}
\end{proposition}
\begin{proof}
According to \eqref{markov}, it is enough to show the following inequality: for any $\ig, (\tx)\in\0TR$,
\begin{equation}
\label{refl2}
v^i(\tx)\leq v^{i+1}(\tx)+\bar{g}_{\iplus1}(\tx).
\end{equation}
Indeed, we assume by contradiction that there exists some $(t_0,x_0)\in [0,T)\times \R^k$ and a strictly positive $\epsilon>0$ such that
\begin{equation}
\label{refl22}
 v^i(t_0,x_0)-v^{i+1}(t_0,x_0)-\bar{g}_{\iplus1}(t_0,x_0)\geq \epsilon>0.
\end{equation}
By the uniform convergence of $(\bar{v}^{i,m})_{\ig}$ to the functions  $(v^i)_{\ig}$ on compact subsets, we can find some $\rho>0$ and a ball defined by
\[\mathcal{B}((t_0,x_0),\rho):=\lbrace(\tx)\in\0TR, \;\text{s.t.}\; |t-t_0|\leq \rho \; \text{and}\; |x-x_0|\leq \rho \rbrace  \]
and some $m_0$ large enough such that for any $m\geq m_0,$
\begin{equation}\label{bveps} \bar{v}^{i,m}(\tx)-\bar{v}^{i+1,m}(\tx)-\bar{g}_{\iplus1}(\tx)\geq \dfrac{\epsilon}{8}>0, \,\,\fr (\tx) \in  \mathcal{B}((t_0,x_0),\rho).\end{equation}
Next let us introduce the following stopping time
\[ \tau_{t_0,x_0};=\inf\lbrace s\geq t_0,X_s^{\tx}\not\in\mathcal{B}((t_0,x_0),\rho) \rbrace \wedge(t_0+\rho). \]
Notice that for any $s\in [t_0,\tau_{t_0,x_0}],$
\begin{align*}
\bar{Y}_s^{i,m,t_0,x_0}&=\bar{v}^{i,m}(s,X_s^{t_0,x_0})\\
&>\bar{v}^{i+1,m}(s,X_s^{t_0,x_0})+\bar{g}_{\iplus1}(s,X_s^{t_0,x_0})\\
&>\bar{v}^{i+1,m}(s,X_s^{t_0,x_0})-\underline{g}_{\iplus1}(s,X_s^{t_0,x_0})\\
&=\bar{Y}^{i+1,m,t_0,x_0}_s-\underline{g}_{\iplus1}(s,X_s^{t_0,x_0})
\end{align*}
As a result for $s\in [t_0,\tau_{t_0,x_0}], d\bar{K}^{i,m,+,t_0,x_0}_s=0$ and then from \eqref{ym}   we deduce that: $\fr$ $s\in [t_0,\tau_{t_0,x_0}]$, 
\begin{align}
\overline{Y}_s^{i,m,t_0,x_0}=&\bar{Y}_{\tau_{t_0,x_0}}^{i,m,t_0,x_0}+\int_s^{\tau_{t_0,x_0}} \{{f}^{i,m}(r,X_r^{t_0,x_0},(\overline{Y}_r^{l,m,t_0,x_0})_{l \in \Gamma},\overline{Z}_r^{i,m,t_0,x_0})\\&-m(\overline{Y}_r^{i,m,t_0,x_0}-[\overline{Y}_r^{i+1,m,t_0,x_0}+\overline{g}_{\iplus1}(r,X_r^{t_0,x_0})])^+\}dr \nn -\int_s^{\tau_{t_0,x_0}} \overline{Z}_r^{i,m,t_0,x_0}dB_r.\nn
\end{align}
Next as in \cite{EKPPQ97}, since $\overline{g}_{\iplus1}$, $\bar{v}^{i,m}$ and $\bar{v}^{i+1,m}$ are of polynomial growth (uniformly for these latter) and by using \eqref{State_3} we deduce that:
\begin{align}\label{refl5}
\begin{split}
&m^2\mathbb{E}[\{\int_{t_0}^{\tau_{t_0,x_0}}(\bar{Y}_s^{i,m,t_0,x_0}-\bar{Y}_s^{i+1,m,t_0,x_0}-\bar{g}_{\iplus1}(s,X_s^{t_0,x_0}))^+ds\}^2]\\
&\qq\qq\qq \leq C\mathbb{E}[\sup_{s\in[t_0,\tau_{t_0,x_0}], \ig}|\bar{Y}_s^{i,m,t_0,x_0}|^2]+C\mathbb{E}[\{\int_{t_0}^{\tau_{t_0,x_0}}f^i(s,X_s^{t_0,x_0},0,0)ds\}^2].
\end{split}
\end{align}
for some cosntant $C$ which is independant of $m$. Therefore using \eqref{bveps} we have
\begin{align}\label{refl52}
\begin{split}
&m^2\frac{\epsilon^2}{64}\p[t_0<\tau_{t_0,x_0}]\leq C\mathbb{E}[\sup_{s\in[t_0,\tau_{t_0,x_0}], \ig}|\bar{Y}_s^{i,m,t_0,x_0}|^2]+C\mathbb{E}[\{\int_{t_0}^{\tau_{t_0,x_0}}f^i(s,X_s^{\tx},0,0)ds\}^2].
\end{split}
\end{align}which implies, in sending $m$ to $+\infty$,  
$\p[t_0<\tau_{t_0,x_0}]=0$, i.e. $\p[t_0=\tau_{t_0,x_0}]=1$. But this is contradictory since $\rho>0$ and then $(t_0,x_0)$ satisfying \eqref{refl22} does not exists. The proof of the claim is complete.  
\end{proof}
We now give the main result of this section.

\begin{theorem}\label{exist}Assume that the assumptions [H2],[H3] and [H5] are fulfilled anf for any $\ig$, $f^i$ does not depend on $z^i$. Then for any $(\tx)\in \0TR$, there exist adapted processes $K^{i,\pm,\tx}$ and $Z^{i,\tx}$ valued respectively in $\R^+$ and $\R^d$ such that, in combination with $Y^{i,\tx}$, verify: For any $\ig$,

i) $K^{i,\pm,\tx}$ are continuous, non decreasing and 
$K^{i,\pm,\tx}_t=0$ ; $\p$-a.s. $K^{i,\pm,\tx}_T<\infty$ and $\int_t^T|Z^{i,\tx}_s|^2ds <\infty$ ; 

ii) $\frst$, 
\begin{equation}\label{yi}
\left\lbrace
\begin{array}{l}
Y_s^{i,\tx}=h^i(X_T^{\tx})+\int_s^Tf^i(r,X_r^{\tx},(Y_r^{l,\tx})_{l\in\Gamma})dr-\int_s^TZ^{i,\tx}_rdB_r\\\\\qq\qq\qq+K_T^{i,+,\tx}-K_s^{i,+,\tx}-(K_T^{i,-,\tx}-K_s^{i,-,\tx});
\\\\
L_s^i((Y^{l,\tx})_{l\in\Gamma})\leq Y^{i,\tx}_s\leq U_s^i((Y^{l,\tx})_{l\in\Gamma});\\\\
\int_t^T (Y^{i,\tx}_s-L_s^i((Y^{l,\tx})_{l\in\Gamma}))dK_s^{i,+,\tx}=0\;\mbox{and} \; \int_t^T (Y^{i,\tx}_s-U_s^i((Y^{l,\tx})_{l\in\Gamma}))dK_s^{i,-,{i,\tx}}=0
\end{array}
\right.
\end{equation}
where for $s\in \tT$, $L_s^i((Y^{l,\tx})_{l\in\Gamma}):=Y^{i+1,\tx}_s-\underline{g}_{\iplus1}(s,X_s^{\tx})$ and $U_s^i((Y^{l,\tx})_{l\in\Gamma}):=Y^{i+1,\tx}_s+\bar{g}_{\iplus1}(s,X_s^{\tx})$.

Moreover if there exists another quadruple 
$(\bar{Y}^{i,\tx},\bar{Z}^{i,\tx},\bar{K}^{i,\pm,\tx})$ which satisfies (i)-(ii), then for any $\stt$ and $\ig$,
$\bar{Y}^{i,\tx}_s={Y}^{i,\tx}_s$, $\bar{K}^{i,\pm,\tx}_s={K}^{i,\pm,\tx}_s$ and finally $\bar{Z}^{i,\tx}={Z}^{i,\tx}_s, \,\,ds\otimes d\p \mbox{ on }[t,T]\times \Omega.$
\end{theorem}
\begin{proof} \textit{\underline{Existence}}\\

\nd For any $\ig$ and $m\ge 0$, the processes $\bar Y^{i,m,,\tx}$ have the following representation (see e.g. A4 in \cite{DHM15} for more details): For any $s\in[t, T]$,
\begin{align}\label{exist2}
\begin{split}
\bar{Y}_s^{i,m,\tx}&=\esssup_{\sigma\geq s}\essinf_{\tau\geq s} \mathbb{E}[h^i(X_T^{\tx})1_{(\sigma=\tau=T)}+\int_s^{\sigma\wedge\tau}f^i(r,X_r^{\tx},(\bar{Y}_r^{l,m,\tx})_{l\in\Gamma})dr\\
&\qq\qq\qq+L_\sigma^i((\bar{Y}^{l,m,\tx})_{l\in\Gamma})1_{(\sigma<\tau)}+\lbrace U_\tau^i((\bar{Y}^{l,m})_{l\in\Gamma})\vee \bar{Y}^{i,m,\tx}_\tau\rbrace 1_{(\tau\leq \sigma,\t<T)}|\mathcal{F}_s].
\end{split}
\end{align}
Now the convergence of $(\bar Y^{i,m,\tx})_m$ to $Y^{i,,\tx}$ in $\sd([t,T])$ (by \eqref{convymy}) and the inequalities \eqref{refl1} imply that, in taking the limits in both hand-sides of \eqref{exist2}: $\frst$,
\begin{align}\label{exist3}
\begin{split}
&Y^{i,\tx}_s=\esssup_{\sigma\geq s}\essinf_{\tau\geq s}\mathbb{E}[h^i(X_T^{\tx})1_{(\sigma=\tau=T)}+\int_s^{\sigma\wedge\tau}f^i(r,X_r^{\tx},(Y_r^{l,\tx})_{l\in\Gamma})dr\\
&\qq\qq\qq\qq+L_\sigma^i((Y^{l,\tx})_{l\in\Gamma})1_{(\sigma<\t)}+U_\tau^i((Y^{l,\tx})_{l\in\Gamma})1_{(\tau\leq \sigma, \t<T)}|\mathcal{F}_s].
\end{split}
\end{align}
Next the third inequality in \eqref{ym} and \eqref{refl1} imply that: For any $\stt$ and $\ig$,
\begin{align*}
U_s^i((Y^l)_{l\in\Gamma})\ge Y^i_s\ge L_s^i((Y^l)_{l\in\Gamma}).
\end{align*}On the other hand by Assumption [H3]-a), 
\begin{align*}
U_s^i((Y^{l,\tx})_{l\in\Gamma})-L_s^i((Y^{l,\tx})_{l\in\Gamma})=\bar{g}_{\iplus1}(s,X_s^{\tx})+\underline{g}_{\iplus1}(s,X_s^{\tx})>0
\end{align*}
which means that the obstacles $U^i((Y^{l,\tx})_{l\in\Gamma})$ and $L^i((Y^{l,\tx})_{l\in\Gamma})$, for any $\ig$, are completely separated. Therefore by Theorem 3.7 in \cite{HH05}, there exist progressively measurable processes $\underbar Y^{i,\tx}$, $K^{i,\pm,\tx}$ and $Z^{i,\tx}$ valued respectively in $\R$, $\R^+$ and $\R^d$ such that:

i) $\underbar Y^{i,\tx}\in \sd([t,T])$, $K^{i,\pm,\tx}$ are continuous non decreasing and $K^{i,\pm,\tx}_t=0$ ; $\p$-a.s. $\int_t^T|Z^{i,\tx}_s|^2ds <\infty$ ; 

ii) The processes $(\underbar Y^{i,\tx}, K^{i,\pm,\tx},Z^{i,\tx})$ verify: 
$\frst$, 
\begin{equation}\label{yi2}
\left\lbrace
\begin{array}{l}
\underbar Y_s^{i,\tx}=h^i(X_T^{\tx})+\int_s^Tf^i(r,X_r^{\tx},(Y_r^{l,\tx})_{l\in\Gamma})dr-\int_s^TZ^{i,\tx}_rdB_r\\\\\qq\qq\qq+K_T^{i,+,\tx}-K_s^{i,+,\tx}-(K_T^{i,-,\tx}-K_s^{i,-,\tx});
\\\\
L_s^i((Y^{l,\tx})_{l\in\Gamma})\leq \underbar Y_s^{i,\tx}\leq U_s^i((Y^{l,\tx})_{l\in\Gamma});\\\\
\int_t^T (\underbar Y^{i,\tx}_s-L_s^i((Y^{l,\tx})_{l\in\Gamma}))dK_s^{i,+,\tx}=0\;\mbox{and} \; \int_t^T (\underbar Y^{i,\tx}_s-U_s^i((Y^{l,\tx})_{l\in\Gamma}))dK_s^{i,-,\tx}=0.
\end{array}
\right.
\end{equation}Moreover $\underbar Y^{i,\tx}$ has the following representation: $\frst$,
\begin{align}\label{exist35}
\begin{split}
&\underbar Y^{i,\tx}_s=\esssup_{\sigma\geq s}\essinf_{\tau\geq s}\mathbb{E}[h^i(X_T^{\tx})1_{(\sigma=\tau=T)}+\int_s^{\sigma\wedge\tau}f^i(r,X_r^{\tx},(Y_r^l)_{l\in\Gamma})dr\\
&\qq\qq\qq\qq+L_\sigma^i((Y^{l,\tx})_{l\in\Gamma})1_{(\sigma<\t)}+U_\tau^i((Y^{l,\tx})_{l\in\Gamma})1_{(\tau\leq \sigma,\t<T)}|\mathcal{F}_s].
\end{split}
\end{align}
Thus for any $s\in [t,T]$, $\underbar Y^{i,\tx}=Y^{i,\tx}$ and by \eqref{yi2},  
$(Y^{i,tx}, K^{i,\pm,\tx},Z^{i,\tx})$ verify \eqref{yi}. Finally as $i$ is arbitrary then $(Y^{i,\tx}, K^{i,\pm,\tx},Z^{i,\tx})_{\ig}$ is a solution for the system of reflected BSDEs with double obstacles \eqref{yi}. The proof of existence is then stated. It remains to show uniqueness. 
\ms

\noindent\textit{\underline{Uniqueness}}: In this part we apply the fixed point argument over the value of the stochastic game representation (Theorem \ref{value sans [H4]}), and the proof is similar to \cite{HM19}. In the following proof, the defined processes $(Y^{\phi,i},Z^{\phi,i},K^{\phi,i,\pm})_{\ig}$ and $(Y^{\psi,i},Z^{\psi,i},K^{\psi,i,\pm})_{\ig}$ depend on $(\tx)$, but for simplicity of notations we omit it as there is no confusion.

So let us define the following operator:
\begin{align*}
&\Phi:\cH^{2,p}\to \cH^{2,p}\\
&\vec{\phi}:=(\phi^i)_{\ig}\mapsto\Phi(\vec{\phi}):=(Y^{\phi,i})_{\ig}
\end{align*}
where $(Y^{\phi,i},Z^{\phi,i},K^{\phi,i,\pm})_{\ig}$ is the unique solution of 
\begin{align}\label{uniqy0}
\left\lbrace
\begin{array}{l}
\yphi\in\cS^2([t,T]), \p-a.s. \int_t^T|\zphi_s|^2ds<\infty \mx{ and }\kphip_T+\kphim_T<\infty \,\,(\kphip_t+\kphim_t=0);\\
\yphi_s=h^i(X_T^{\tx})+\dis\int_s^Tf^i(r,X_r^{\tx},\vec{\phi}(r))dr-\int_s^T\zphi_rdB_r+\\\qq\qq\qq\qq\qq\qq\kphip_T-\kphip_s-(\kphim_T-\kphim_s),\,\stt;\\
L_s^i((Y^{\phi,l})_{\lG})\leq \yphi_s\leq U_s^i((Y^{\phi,l})_{\lG}),\stt;\\
\dis\int_t^T\left(\yphi_s-L_s^i((Y^{\phi,l})_{\lG})\right)d\kphip_s=0\;\text{and}\;\dis\int_t^T\left(\yphi_s-U_s^i((Y^{\phi,l})_{\lG})\right)d\kphim_s=0.
\end{array}
\right.
\end{align}
In the similar way we define another element of $\cH^{2,p}$ by $\vec{\psi}:=(\psi^i)_{\ig}$ and let $(\ypsi_s,\zpsi_s,K_s^{\psi,i,\pm})_{\stt}$ be a solution of \eqref{uniqy0} where its driver is replaced with $f^i(t,x,\vec\psi(t)), \forall \ig.$

Next we set the following norm, denoted by $\Vert.\Vert_{2,\beta}$ on $\cH^{2,p}$:
\[\Vert y\Vert_{2,\beta}:=(\E[ \dis\int_t^Te^{\beta s}|y_s|^2ds ])^{1/2}.\]
The following calculus is dedicated to prove that $\Phi$ is a contraction on $(\cH^{2,p},\Vert.\Vert_{2,\beta})$ where the appropriate value of $\beta$ is determined later. 

Let us recall Theorem \ref{value sans [H4]} and Remark \ref{rmq_value}, for any $(\tx)\in [0,T]\times\R^k$ and $t\leq s\leq T,$ the following representation holds true:
\begin{equation}\label{uniqy1}
\yphi_s=\essinf_{v\in\cBone_s}\esssup_{u\in\cAone_s}\jphi=\esssup_{u\in\cAone_s}\essinf_{v\in\cBone_s}\jphi
\end{equation}
where $\jphi=\E\left[\left.h^{\tuv_T}(X_T^{\tx})+\dis\int_s^Tf^{\tuv_r}(r,X_r^{\tx},\vec{\phi}(r))-C_\infty^{\tuv}\right\vert \cF_s\right]$ ($u$ and $v$ start from $i$ at time $s$). In the same way $\ypsi$ has also the stochastic game representation by replacing $\phi$ to $\psi$. \\

Now we study the difference of $|\yphi-\ypsi|$. Indeed, $\forall \ig, t\in [0,T], t\leq s\leq T,$
\begin{align}\label{uniqy2}
|\yphi_s-\ypsi_s|\leq \esssup_{u\in\cAone_s}\esssup_{v\in\cBone_s}|\jphi-\jpsi|.
\end{align}
Thanks to the martingale representation theorem, there exists an $(\cF_s)_{s\leq T}-$adapted process $\Delta^{\phi,\psi,\tuv}\in\cH^{2,d}$ such that
\begin{align*}
&\jphi-\jpsi =\E\left[\left.\dis\int_s^T f^{\tuv_r}(r,X_r^{\tx},\vec{\phi}(r))-f^{\tuv_r}(r,X_r^{\tx},\vec{\psi}(r)) dr\right|\cF_s\right]\\
&=\E\left[\left.\dis\int_0^T f^{\tuv_r}(r,X_r^{\tx},\vec{\phi}(r))-f^{\tuv_r}(r,X_r^{\tx},\vec{\psi}(r)) dr\right|\cF_s\right]\\
&\quad-\int_0^s f^{\tuv_r}(r,X_r^{\tx},\vec{\phi}(r))-f^{\tuv_r}(r,X_r^{\tx},\vec{\psi}(r)) dr\\
&=\E\left[\dis\int_0^T f^{\tuv_r}(r,X_r^{\tx},\vec{\phi}(r))-f^{\tuv_r}(r,X_r^{\tx},\vec{\psi}(r)) dr\right]+\int_0^s\Delta_r^{\phi,\psi,\tuv}dB_r\\
&\quad-\int_0^s f^{\tuv_r}(r,X_r^{\tx},\vec{\phi}(r))-f^{\tuv_r}(r,X_r^{\tx},\vec{\psi}(r))dr.
\end{align*} 
Therefore we obtain the following differential form for the difference of the two value functions:
\begin{align*}
d(\jphi-\jpsi)=\left[f^{\tuv_s}(s,X_s^{\tx},\vec{\phi}(s))-f^{\tuv_s}(s,X_s^{\tx},\vec{\psi}(s))\right]ds+\Delta_s^{\phi,\psi,\tuv}dB_s
\end{align*}
Next for any $s\in [t,T]$, we apply It\^o's formula on $e^{\beta s}\left(\jphi-\jpsi\right)^2$ yielding
\begin{align}\label{uniqy4}
&d\left[e^{\beta s}\left(\jphi-\jpsi\right)^2\right]=\beta e^{\beta s}\left[\jphi-\jpsi\right]^2\notag\\
&\quad+2e^{\beta s}\left(\jphi-\jpsi \right)\left[-\left( f^{\tuv_s}(s,X_s^{\tx},\vec{\phi}(s))-f^{\tuv_s}(s,X_s^{\tx},\vec{\psi}(s))\right) \Delta_s^{\phi,\psi,\tuv}dB_s \right]\notag\\
&\quad +e^{\beta s}\left(\Delta_s^{\phi,\psi,\tuv}\right)^2ds.
\end{align}
By integrating \eqref{uniqy4} over $[s,T]$ we obtain
\begin{align}\label{uniqy5}
&e^{\beta s}\left(\jphi-\jpsi\right)^2+\dis\int_s^Te^{\beta r}\left(\Delta_r^{\phi,\psi,\tuv}\right)^2dr\notag\\
&=-\beta\dis\int_s^Te^{\beta r}\left( \jphir-\jpsir\right)^2dr\notag\\
&\quad +2\int_s^Te^{\beta r}\left( \jphir-\jpsir\right)\left(f^{\tuv_r}(r,X_r^{\tx},\vec{\phi}(r))-f^{\tuv_r}(r,X_r^{\tx},\vec{\psi}(r))\right)dr\notag\\
&\quad -2\dis\int_s^T\left(\jphir-\jpsir\right)\Delta_r^{\phi,\psi,\tuv}dB_r.
\end{align}
By applying the inequality $2ab\leq\beta a^2 +\dfrac{b^2}{\beta}$, \eqref{uniqy5} yields
\begin{align*}
&e^{\beta s}\left(\jphi-\jpsi\right)^2+\dis\int_s^Te^{\beta r}\left(\Delta_r^{\phi,\psi,\tuv}\right)^2dr\\
&\leq \dfrac{1}{\beta}\dis\int_s^Te^{\beta s}\left( f^{\tuv_r}(r,X_r^{\tx},\vec{\phi}(r))-f^{\tuv_r}(r,X_r^{\tx},\vec{\psi}(r))\right)^2dr\\
&\quad -2\dis\int_s^T\left(\jphir-\jpsir\right)\Delta_r^{\phi,\psi,\tuv}dB_r.
\end{align*}
By the Lipschitz condition on the driver $f^{\tuv}$, and using the fact that \\$\int_s^Te^{\beta r}\left(\Delta_r^{\phi,\psi,\tuv}\right)^2dr\geq 0,$ we then obtain
\begin{align}\label{uniqy6}
&e^{\beta s}\left(\jphi-\jpsi\right)^2\notag\\
&\le\dfrac{C^2}{\beta}\dis\int_s^T\vert\vec{\phi}(r)-\vec{\psi}(r)\vert^2dr -2\dis\int_s^T\left(\jphir-\jpsir\right)\Delta_r^{\phi,\psi,\tuv}dB_r
\end{align}
where $C=\sum_{\ig}C_i$ with $C_i$ the Lipschitz constant w.r.t. $f^i, \forall \ig$. On the other hand since $(2\dis\int_s^u\left(\jphir-\jpsir\right)\Delta_r^{\phi,\psi,\tuv}dB_r)_{u\in[s,T]}$ is a martingale, then taking the conditional expectation w.r.t. $\cF_s$ on both sides of \eqref{uniqy6} we have
\begin{align}\label{uniqy7}
e^{\beta s}\left(\jphi-\jpsi\right)^2\leq \dfrac{C^2}{\beta}\E\left[ \dis\int_s^T\vert\vec{\phi}(r)-\vec{\psi}(r)\vert^2dr|\cF_s\right].
\end{align}
Let us recall \eqref{uniqy2}, then by taking the expectation on both sides of \eqref{uniqy7} we obtain: $\fr$ $\stt$,
\begin{align}\label{uniqy8}
\E\left[e^{\beta s}\left(\yphi_s-\ypsi_s\right)^2\right]\leq \dfrac{C^2}{\beta}\E\left[ \dis\int_t^T\vert\vec{\phi}(r)-\vec{\psi}(r)\vert^2dr\right].
\end{align}
The last step is integrating \eqref{uniqy8} over $s\in [t,T]$ and then summing over all $\ig$ to obtain: 
\begin{align}\label{uniqy9}
\E\left[\dis\int_t^T \sum_{\ig}e^{\beta s}\left(\yphi_s-\ypsi_s\right)^2ds\right]\leq\dfrac{C^2TP}{\beta}\E\left[ \dis\int_t^T\vert\vec{\phi}(r)-\vec{\psi}(r)\vert^2dr\right].
\end{align}
Obviously it is enough to take $\beta>C^2TP$ (for example we can let $\beta:=4C^2TP$) then the operator $\Phi$ is a contraction on $\cH^{2,p}$ to itself. As a consequence, there exists a fixed point which is nothing but the unique solution of \eqref{yi}.

\noindent Next we suppose that there exists another solution $(\hat{Y}^i,\hat{Z}^i,\hat{K}^{i,\pm})_{\ig}$ of \eqref{yi}, i.e.
\begin{equation}\label{yihat}
\left\lbrace
\begin{array}{l}
\hat{Y}_s^i=h^i(X_T^{\tx})+\int_s^Tf^i(r,X_r^{\tx},(\hat{Y}_r^l)_{l\in\Gamma})dr-\int_s^T\hat{Z}^{i,\tx}_rdB_r\\\\\qq\qq\qq+\hat{K}_T^{i,+,\tx}-\hat{K}_s^{i,+,\tx}-(\hat{K}_T^{i,-,\tx}-\hat{K}_s^{i,-,\tx}), \stt;
\\\\
L_s^i((\hat{Y}^l)_{l\in\Gamma})\leq Y_s^i\leq U_s^i((\hat{Y}^l)_{l\in\Gamma}),\stt;\\\\
\int_t^T (\hat{Y}^i_s-L_s^i((\hat{Y}^l)_{l\in\Gamma}))d\hat{K}_s^{i,+}=0\;\mbox{and} \; \int_t^T (\hat{Y}^i_s-U_s^i((\hat{Y}^l)_{l\in\Gamma}))d\hat{K}_s^{i,-}=0.
\end{array}
\right.
\end{equation}
Thanks to the fixed point result \eqref{uniqy9} we have immediately $Y^i=\hat{Y}^i, \forall \ig$. By applying the equality of $Y^i$ and $\hat{Y}^i$,  we also have $Z^i=\hat{Z}^i$ since from the representation of \eqref{yi} and \eqref{yihat}, their martingale parts should be equal, i.e. for any $\ig, s\in[t,T], \int_s^TZ_s^idB_s=\int_s^T\hat{Z}_s^idB_s.$ Moreover by \eqref{yi} and \eqref{yihat} we have 
$ \forall s\in[t,T], \ig, K^{i,+}_s-K_s^{i,-}=\hat{K}^{i,+}_s-\hat{K}^{i,-}_s$. It remains us now to prove the equality of the increasing processes correspondingly.

For any $s\in[t,T], \ig$ we have 
\begin{align}\label{uniqk1}
\int_t^s\left( Y_r^i-L^i_r((Y_l)_{l\in\Gamma})\right)(dK_r^{i,+}-dK_r^{i,-})=\int_t^s\left( Y_r^i-L^i_r((Y_l)_{l\in\Gamma})\right)(d\hat{K}_r^{i,+}-d\hat{K}_r^{i,-}).
\end{align}
On the other hand by the minimality conditions we have 
\begin{align}\label{uniqk2}
\forall s\in [t,T], \ig, &\int_t^s\left( Y_r^i-L^i_r((Y_l)_{l\in\Gamma})\right)(dK_r^{i,+}-dK_r^{i,+})=-\int_t^s\left( Y_r^i-L^i_r((Y_l)_{l\in\Gamma})\right)dK_r^{i,-}\notag\\
&=-\int_t^s\left( U_r^i((Y_l)_{l\in\Gamma})-L^i_r((Y_l)_{l\in\Gamma})\right)dK_r^{i,-}.
\end{align}
This last equality is due to the fact that $\forall r\in[t,s], dK_r^{i,-}\not=0$ only if $Y^i$ touches the upper obstacle. In the same way we have also the following condition for $\hat{K}^{i,-}: \forall \ig, s\in[t,T],$
\begin{align}\label{uniqk3}
&\dis\int_t^s\left( Y_r^i-L^i_r((Y_l)_{l\in\Gamma})\right)(d\hat{K}_r^{i,+}-d\hat{K}_r^{i,+})=-\int_t^s\left( Y_r^i-L^i_r((Y_l)_{l\in\Gamma})\right)d\hat{K}_r^{i,-}\notag\\
&=-\int_t^s\left( U_r^i((Y_l)_{l\in\Gamma})-L^i_r((Y_l)_{l\in\Gamma})\right)d\hat{K}_r^{i,-}.
\end{align}
Combining \eqref{uniqk1}-\eqref{uniqk3} and (H3)-a)(the two obstacles are totally separated), we finally obtain 
\[\forall \ig,\; \stt, \,K^{i,-}_s=\hat{K}^{i,-}_s \]
since $K^{i,-}_t=\hat{K}^{i,-}_t=0$. Finally the equality 
$K^{i,+}_s-K_s^{i,-}=\hat{K}^{i,+}_s-\hat{K}^{i,-}_s$, $\stt$, implies $K^{i,+}=\hat{K}^{i,+}$. The proof of uniqueness is now finished. 
\end{proof}

We now go back to systems \eqref{sysedpmilieu} and \eqref{sysedpmilieu2} and the question is whether or not they have the same solution. We have the following result:
\begin{proposition}Assume that the assumptions [H2],[H3] and [H5] are fulfilled anf for any $\ig$, $f^i$ does not depend on $z^i$. Then for any $\ig$, $\bar v^i=\underbar v^i$.
\end{proposition}
\begin{proof}:
Actually $(-\underline v^i)_{\ig}$ is the unique solution of the following system of PDEs with obstacles:
\begin{equation}\label{sysedpmilieu3}
\left\lbrace
\begin{array}{l}
\min\{v^i(t,x)-\check L^i(\vec{v})(t,x);\max\left[v^i(t,x)-\check U^i(\vec{v})(t,x);\right.\\
\qquad\left.-\partial_tv^i(t,x)-\mathcal{L}^X(v^i)(t,x)+f^i(t,x,(-v^l(t,x))_{l\in\Gamma},-\sigma(t,x)^\top D_xv^i(t,x)\right]\}=0;\\
v^i(T,x)=-h^i(x)
\end{array}
\right.
\end{equation}where $\check L^i(\vec{v})(t,x)=v^i(\tx)-\bar g_{i,i+1}(t,x)$ and 
$\check U^i(\vec{v})(t,x)=v^i(\tx)-\underline g_{i,i+1}(t,x)$.
Therefore $-\underline v^i$, has accordingly, the representation \eqref{exist3}, i.e. for any $(\tx)$ and $\ig$, setting $\underline Y^{i,\tx}_s=\underline v^i(s\vee t,X^{\tx}_{s\vee t})$ for $\stt$, we have:
\begin{align}\label{exist32}
\begin{split}
-\underline Y^{i,\tx}_s&=\esssup_{\sigma\geq s}\essinf_{\tau\geq s}\mathbb{E}[-h^i(X_T^{\tx})1_{(\sigma=\tau=T)}+\int_s^{\sigma\wedge\tau}-f^i(r,X_r^{\tx},(-\underline Y_r^{l,\tx})_{l\in\Gamma})dr\\
&\qq\qq\qq\qq+\check L_\sigma^i((-\underline Y^{l,\tx})_{l\in\Gamma})1_{(\sigma<\t)}+\check U_\tau^i((-\underline Y^{l,\tx})_{l\in\Gamma})1_{(\tau\leq \sigma, \t<T)}|\mathcal{F}_s]\\
&=\essinf_{\tau\geq s}\esssup_{\sigma\geq s}\mathbb{E}[-h^i(X_T^{\tx})1_{(\sigma=\tau=T)}+\int_s^{\sigma\wedge\tau}-f^i(r,X_r^{\tx},(-\underline Y_r^{l,\tx})_{l\in\Gamma})dr\\
&\qq\qq\qq\qq+\check L_\sigma^i((-\underline Y^{l,\tx})_{l\in\Gamma})1_{(\sigma<\t)}+\check U_\tau^i((-\underline Y^{l,\tx})_{l\in\Gamma})1_{(\tau\leq \sigma, \t<T)}|\mathcal{F}_s]
\end{split}
\end{align}
since the barriers are completely separated (see e.g. \cite{HH05}). Therefore 
\begin{align}\label{exist33}
\begin{split}
&\underline Y^{i,\tx}_s=\esssup_{\sigma\geq s}\essinf_{\tau\geq s}\mathbb{E}[h^i(X_T^{\tx})1_{(\sigma=\tau=T)}+\int_s^{\sigma\wedge\tau}f^i(r,X_r^{\tx},(\underline Y_r^{l,\tx})_{l\in\Gamma})dr\\
&\qq\qq\qq\qq+L_\sigma^i((\underline Y^{l,\tx})_{l\in\Gamma})1_{(\sigma<\t)}+U_\tau^i((\underline Y^{l,\tx})_{l\in\Gamma})1_{(\tau\leq \sigma, \t<T)}|\mathcal{F}_s].
\end{split}
\end{align}Which means that $((\underline Y^{i,\tx}_s)_{\stt})_{\ig}$ verifes \eqref{yi}. As the solution of this latter is unique then for any $\ig$, $\underline Y^{i,\tx}=Y^{i,\tx}$ which means that for $\ig$, $\bar v^i=\underline v^i$. 
\end{proof}
\section{Appendix: Proof of Theorem \ref{thm43}}
In this section, we prove that the system of \eqref{sysedpmilie1} has a unique continuous solution in viscosity sense in the class $\Pi_g$. Indeed, we firstly provide a comparison result of subsolution and supersolution of \eqref{sysedpmilie1} if they exist, then we show that $(\bar v^i)_{\ig}$ is a solution by Perron's method. We recall once for all that the results in this section are constructed under [H2],[H3] and [H5].

\subsection{A comparison result}
Before investigating \eqref{sysedpmilie1}, we provide some a priori results and a comparison principle for sub. and supersolutions of system \eqref{sysedpmilie1}. To begin with let us show the following: 
\begin{lemma}\lb{lemmeindices}
Let $\vec{u}:=(u^i)_{\ig} \,\mbox{$ (resp. $} \vec{\hat{u}}:=(\hat{u}^i)_{\ig})\,$ be an usc subsolution (resp. sci supersolution) of \eqref{sysedpmilie1}. For any $(\tx)\in\0TR,$ let $\hat{\Gamma}(\tx)$ be the following set:
\[\hat{\Gamma}(\tx):=\lbrace \ig, u^i(\tx)-\hat{u}^i(\tx)=\max_{l\in\Gamma}(u^l(\tx)-\hat{u}^l(\tx))    \rbrace.\]
Then there exists $i_0\in\hat{\Gamma}(\tx)$ such that
\[u^{i_0}(\tx)>u^{i_0+1}(\tx)-\underline{g}_{i_0,i_0+1}(\tx)\mx{ and
}\hat{u}^{i_0}(\tx)<\hat{u}^{i_0+1}(\tx)+\overline{g}_{i_0,i_0+1}(\tx).\]
\end{lemma}
\begin{proof}
Let $(\tx)\in\0TR$ be fixed. As $\Gamma$ is a finite set then $\hat{\Gamma}$ is not empty. To proceed, we assume, by contradiction that for any $i\in\hat{\Gamma}(\tx)$, either
\begin{equation}\label{cont1}
u^i(\tx)\leq u^{i+1}(\tx)-\underline{g}_{\iplus1}(\tx)
\end{equation}
or
\begin{equation}\label{cont2}
\hat{u}^i(\tx)\geq \hat{u}^{i+1}(\tx)+\overline{g}_{\iplus1}(\tx)
\end{equation}
holds.\\

Assume first that \eqref{cont1} holds true i.e. $u^i(\tx)\leq u^{i+1}(\tx)-\underline{g}_{\iplus1}(\tx)$.
As $\vec{\hat{u}}$ is a supersolution of \eqref{sysedpmilie1}, we deduce that
\begin{equation}\label{cont3}
\hat{u}^i(\tx)\geq \hat{u}^{i+1}(\tx)-\underline{g}_{\iplus1}(\tx)
\end{equation}
By taking into account of \eqref{cont1} we have
\[\hat{u}^{i+1}(\tx)-\hat{u}^i(\tx)\leq \underline{g}_{\iplus1}(\tx)\leq u^{i+1}(\tx)-u^i(\tx)        \]
which implies
\[ u^i(\tx)-\hat{u}^i(\tx)\leq u^{i+1}(\tx)-\hat{u}^{i+1}(\tx) .\]
However as $i\in\hat{\Gamma}(\tx)$, then the previous inequality is an equality and then 
\begin{equation}\label{2etoiles}\hat{u}^{i+1}(\tx)-\hat{u}^i(\tx)=u^{i+1}(\tx)-u^i(\tx)=\underline{g}_{\iplus1}(t,x).
\end{equation}
As a result we deduce that $(i+1)\in\hat{\Gamma}(\tx)$ and also the equality \eqref{2etoiles} holds. 
\medskip

Next if $u^i(\tx)\leq u^{i+1}(\tx)-\underline{g}_{\iplus1}(\tx)$ does not hold, then $u^i(\tx)> u^{i+1}(\tx)-\underline{g}_{\iplus1}(\tx)$.  On the other hand, assume that \eqref{cont2} holds true, i.e., $\hat{u}^i(\tx)\geq \hat{u}^{i+1}(\tx)+\overline{g}_{\iplus1}(\tx)$. Since $u^i$ is a subsolution of \eqref{sysedpmilie1}, we have
\[ u^i(\tx)\leq u^{i+1}(\tx)+\overline{g}_{\iplus1}(\tx) \]
which implies
\[ \hat{u}^{i+1}(\tx)-\hat{u}^i(\tx)\leq -\overline{g}_{\iplus1}(\tx) \leq u^{i+1}(\tx)-u^i(\tx) \]
and then
\[ u^i(\tx)-\hat{u}^i(\tx)\leq u^{i+1}(\tx)-\hat{u}^{i+1}(\tx).\]
However as $i\in\hat{\Gamma}(\tx)$, then the last inequality is an equality and  $(i+1)\in\hat{\Gamma}(\tx)$. Moreover \begin{equation} \label{3etoiles}u^{i+1}(\tx)-u^{i}(\tx)=-\overline{g}_{\iplus1}(\tx)=\hat{u}^{i+1}(\tx)-\hat{u}^{i}(\tx).\end{equation}
It means that  \eqref{cont1} or \eqref{cont2} imply that $(i+1)\in\hat{\Gamma}(\tx)$ and one of the equalities \eqref{2etoiles}, \eqref{3etoiles}. Repeat now this reasonning as many times as necessary (actually $p$ times) to find a loop such that $\sum_{\ig}\varphi_{\iplus1}(\tx)=0$ ($\varphi_{\iplus1}$ is defined  in \eqref{nflp}), which is contradictory to assumption [H3].
\end{proof}

\noindent Next we give the comparison result.
\begin{proposition}\label{comp}
Let $\vec{u}:=(u^i)_{\ig}$ be an usc subsolution (resp. $\vec{w}:=(w^i)_{\ig}$ be a lsc supersolution) of the system \eqref{sysedpmilie1} and for any $\ig$, both $u^i$ and $w^i$ belong to class $\Pi_g$, i.e. there exist two constants $\gamma$ and $C$ such that
\[ \forall \ig, (\tx)\in\0TR ,\; |u^i(\tx)|+|w^i(\tx)|\leq C(1+|x|^\gamma).\]
Then it holds true that:
\begin{equation} \label{compuiwi}u^i(\tx)\leq w^i(\tx), \; \forall \ig, (\tx)\in\0TR. \end{equation}
\end{proposition}
\begin{proof}
Let us show the result by contradiction, i.e. there exists $\epsilon_0>0$ and some $(t_0,x_0)\in [0,T)\times \R^k$ such that
\begin{equation}\label{comp1}
\max_{\ig}(u^i(t_0,x_0)-w^i(t_0,x_0))\geq \epsilon_0.
\end{equation}
Next without loss of generality we assume that 
there exists $R>0$ such that for $t\in [0,T]$, $|x|\ge R$ we have for any $\ig$, 
\begin{equation}\label{comp2}(u^i-w^i)(\tx)<0.
\end{equation}
Actually if \eqref{comp2} does not hold, it is enough to consider the following functions $w^{i,\theta,\mu}$ defined by
\[ w^{i,\theta,\mu}=w^i(\tx)+\theta e^{-\bar \lambda t}(1+|x|^{2\gamma+2}),\; (\tx)\in\0TR \]
which still a supersolution of \eqref{sysedpmilie1} for any $\theta >0$ and $\bar \lambda \ge \lambda_0$ ($\lambda_0$ is fixed). Then to show that 
$u^i-w^{i,\theta,\mu}\le 0$ for any $\ig$ and finally to take the limit as $\theta \rw 0$ to obtain \eqref{compuiwi}. But for any $\ig$, $u^i-w^{i,\theta,\mu}$ is negative uniformly in $t$ when $|x|$ is large enough since $u^i$ belongs to $\Pi_g$ with polynomial exponent $\g$.

To proceed, let \eqref{comp1}-\eqref{comp2} be fulfilled. Then 
\begin{align*}
\max_{(\tx)\in\0TR} \max_{\ig}\lbrace u^i(\tx)-w^i(\tx) \rbrace  &=\max_{(\tx)\in[0,T]\times B(0,R)}\max_{\ig}\lbrace u^i(\tx)-w^i(\tx)\rbrace\\
&:=\max_{\ig}(u^i-w^i)(t^*,x^*)\geq \epsilon_0>0
\end{align*}where $B(0,R)$ is the ball centered in the origin with radius $R$. Note that $t^*<T$ since $u^i(T,x)\leq h^i(x)\leq w^i(T,x)$.
\ms

The proof now will be divided into two steps:
\ms

\noindent\underline{\textit{Step 1}}: To begin with, we introduce the following auxiliary condition: There exists $\lambda>(p-1)\max_{\ig}C_{f^i}$ such that for any $\ig, (\tx,\vec{y},z)\in\0TR\times\mathbb{R}^{p+d},$ and $(v^1,v^2)\in\mathbb{R}^2$ such that $v^1\geq v^2$ we have
\begin{equation}\label{comps1}
f^i(\tx,[\vec{y}^{-i},v^1],z)-f^i(\tx,[\vec{y}^{-i},v^2],z)\leq -\lambda(v^1-v^2)
\end{equation}
and where $ C_{f^i}$ is the Lipschitz constant of $f^i$ w.r.t. $\vec{y}$.
\ms

So let $i_0$ be an element of $\hat{\Gamma}(t^*,x^*)$ such that
\begin{equation}\lb{inegalite1} u^{i_0}(t^*,x^*)>u^{i_0+1}(t^*,x^*)-\underline{g}_{i_0,i_0+1}(t^*,x^*) \end{equation}
and
\begin{equation}\lb{inegalite2} w^{i_0}(t^*,x^*)<w^{i_0+1}(t^*,x^*)+\overline{g}_{i_0,i_0+1}(t^*,x^*) \end{equation}
which exists by Lemma \ref{lemmeindices}. Next we define the following function: For any $n\geq 1$,
\[ \Phi_n^{i_0}(t,x,y):=(u^{i_0}(\tx)-w^{i_0}(t,y))-\phi_n(t,x,y),\quad (\tx,y)\in [0,T]\times\mathbb{R}^{k+k} \]
where $$\phi_n(t,x,y) :=n|x-y|^{2\gamma+2}+|x-x^*|^{2\gamma+2}+(t-t^*)^2.$$
The function $\Phi_n^{i_0}(t,x,y)$ is usc, then we can find a triple $(t_n,x_n,y_n)\in [0,T]\times \bar B(0,R)^2$ such that
\[ \Phi_n^{i_0}(t_n,x_n,y_n)=\max_{(\tx,y)\in [0,T]\times \bar B(0,R)^2} \Phi_n^{i_0}(t,x,y)\]
($\bar B(0,R)$ is the closure of $B(0,R)$). Then we have
\[ \Phi_n^{i_0}(t^*,x^*,x^*)\leq \Phi_n^{i_0}(t_n,x_n,y_n). \]
From which we deduce that
\begin{align}\label{comp3}
\begin{split}
\Phi_n^{i_0}(t^*,x^*,x^*)& =u^{i_0}(t^*,x^*)-w^{i_0}(t^*,x^*)\\
&\leq \Phi_n^{i_0}(t_n,x_n,y_n)\\
&=u^{i_0}(t_n,x_n)-w^{i_0}(t_n,y_n)-\phi_n(t_n,x_n,y_n)\\&\leq u^{i_0}(t_n,x_n)-w^{i_0}(t_n,y_n)\le C_R
\end{split}
\end{align}
($C_R$ is a constant which may depend on $R$) since the sequences $(t_n)_n$, $(x_n)_n$ and $(y_n)_n$ are bounded and $u^{i_0}$ and $w^{i_0}$ are of polynomial growth. As a result $(x_n-y_n)_{n\geq 0}$ converges to $0$. On the other hand, by boundedness of the sequences, we can find a subsequence, which we still denote by $(t_n,x_n,y_n)_n$, converging to a point denoted $(\hat{t},\hat{x},\hat{x})$. By \eqref{comp3} it satisfies:
\begin{align} \nn u^{i_0}(t^*,x^*)-w^{i_0}(t^*,x^*)&\leq \liminf_n 
(u^{i_0}(t_n,x_n)-w^{i_0}(t_n,y_n)) \\\nn&\le \limsup_n 
(u^{i_0}(t_n,x_n)-w^{i_0}(t_n,y_n)) \\&\nn\le
\limsup_n 
u^{i_0}(t_n,x_n)-\liminf_nw^{i_0}(t_n,y_n) \\&
\leq u^{i_0}(\hat{t},\hat{x})-w^{i_0}(\hat{t},\hat{x})
\end{align}
since $u^{i_0}$ (resp. $w^{i_0}$) is usc (resp. lsc).
As the maximum of $u^{i_0}-w^{i_0}$ on $\0TR$ is reached in $(t^*,x^*)$, then 
$u^{i_0}(\hat{t},\hat{x})-w^{i_0}(\hat{t},\hat{x})=u^{i_0}(t^*,x^*)-w^{i_0}(t^*,x^*)$ and consequently the sequence $(u^{i_0}(t_n,x_n)-w^{i_0}(t_n,y_n))_n$ converges to $u^{i_0}(t^*,x^*)-w^{i_0}(t^*,x^*)$. Next as we have 
\begin{align}\label{comp3x}
\begin{split}
\Phi_n^{i_0}(t^*,x^*,x^*)& =u^{i_0}(t^*,x^*)-w^{i_0}(t^*,x^*)\\
&\leq \Phi_n^{i_0}(t_n,x_n,y_n)\\
&=u^{i_0}(t_n,x_n)-w^{i_0}(t_n,y_n)-\phi_n(t_n,x_n,y_n)
\end{split}
\end{align}
then $(\phi_n(t_n,x_n,y_n))_n$ converges to 0 as $n\rw \infty$ and then $(t_n)_n$, $(x_n)_n$ and $(y_n)$ converge respectively to $t^*$, $x^*$ and $x^*$. Finally 
\begin{align*}
\liminf_nu^{i_0}(t_n,x_n)&= u^{i_0}(t^*,x^*)-w^{i_0}(t^*,x^*)+\liminf_nw^{i_0}(t_n,y_n)\\
&\geq u^{i_0}(t^*,x^*)\geq \limsup_nu^{i_0}(t_n,x_n)
\end{align*}
which implies that the sequence $(u^{i_0}(t_n,x_n))_n$ converges to $u^{i_0}(t^*,x^*)$ and then also the sequence $(w^{i_0}(t_n,y_n))_n$ converges to $w^{i_0}(t^*,x^*).$
\medskip

Next, we recall the definition of $i_0\in\hat{\Gamma}(t^*,x^*)$. By \eqref{inegalite1}-\eqref{inegalite2}, for $n$ large enough we can find a subsequence $(t_n,x_n)_n$ such that
\begin{equation}\label{inegalitevisco1} u^{i_0}(t_n,x_n)>u^{i_0+1}(t_n,x_n)-\underline{g}_{i_0i_0+1}(t_n,x_n) \end{equation}
and
\begin{equation}\label{inegalitevisco2} w^{i_0}(t_n,y_n)<w^{i_0+1}(t_n,y_n)+\overline{g}_{i_0i_0+1}(t_n,y_n). \end{equation}Next we apply Crandall-Ishii-Lions's Lemma (see e.g. \cite{FS06}, pp.216) and then there exist $(p_u^n,q_n^n,M_u^n)\in \bar{J}^+(u^{i_0})(t_n,x_n) \;\mbox{$ and $} \;(p_w^n,q_w^n,M_w^n)\in\bar{J}^-(w^{i_0})(t_n,y_n)$ such that
\begin{equation}\label{comp5}
\left\lbrace
\begin{array}{l}
p_u^n-p_w^n=\partial_t\phi_n(t_n,x_n,y_n)=2(t_n-t^*),\\\
q_u^n=\partial_x\phi_n(t_n,x_n,y_n),\\\
q_w^n=-\partial_y\phi_n(t_n,x_n,y_n) \mbox{ and }\\\
\begin{pmatrix}
 M_u^n & 0 \\\
0 & -M_w^n
\end{pmatrix}
\leq A_n+\dfrac{1}{2n}A_n^2
\end{array}
\right.
\end{equation}
where $A_n=D^2_{xy}\phi_n(t_n,x_n,y_n)$. Next by taking into account that $(u^i)_{\ig}$ and $(w^i)_{\ig}$ are respectively subsolution and supersolution of \eqref{sysedpmilie1} and the inequalities 
\eqref{inegalitevisco1}-\eqref{inegalitevisco2}, we obtain
\begin{equation}\label{comp6}
-p_u^n-b(t_n,x_n)^\top q_u^n-\dfrac{1}{2}Tr[(\sigma\sigma^\top(t_n,x_n))(t_n,x_n)M_u^n]-f^{i_0}(t_n,x_n,(u^l(t_n,x_n))_{l\in\Gamma},\sigma(t_n,x_n)^\top q_u^n)\leq 0
\end{equation}
and
\begin{equation}\label{comp7}
-p_w^n-b(t_n,y_n)^\top q_w^n-\dfrac{1}{2}Tr[(\sigma\sigma^\top(t_n,y_n))(t_n,y_n)M_w^n]-f^{i_0}(t_n,y_n,(w^l(t_n,y_n))_{l\in\Gamma},\sigma(t_n,y_n)^\top q_w^n)\geq 0.
\end{equation}
By taking the difference of \eqref{comp6} and \eqref{comp7}, one deduces that
\begin{align*}
-(p_u^n-p_w^n)-(b(t_n,x_n)^\top q_u^n-b(t_n,y_n)^\top q_w^n)-\dfrac{1}{2}Tr[\lbrace \sigma\sigma^\top(t_n,x_n)M_u^n-\sigma\sigma^\top(t_n,y_n)M_w^n \rbrace]\\
-\lbrace f^{i_0}(t_n,x_n,(u^l(t_n,x_n))_{l\in\Gamma},\sigma(t_n,x_n)^\top q_u^n)-f^{i_0}(t_n,y_n,(w^l(t_n,y_n))_{l\in\Gamma},\sigma(t_n,y_n)^\top q_w^n)\rbrace \leq 0.
\end{align*}
Combining with \eqref{comp5}, there exists some appropriate $\rho_n$ with $\limsup_{n\rwi} \rho_n\leq 0$ such that the last inequality yields the following one:
\[ -\lbrace f^{i_0}(t_n,x_n,(u^l(t_n,x_n))_{l\in\Gamma},\sigma(t_n,x_n)^\top q_u^n)-f^{i_0}(t_n,x_n,(w^l(t_n,y_n))_{l\in\Gamma},\sigma(t_n,x_n)^\top q_u^n)\rbrace \leq \rho_n \]
Next by linearising $f^{i_0}$ and condition \eqref{comps1} we obtain
\begin{equation}\label{comp8}
\lambda(u^{i_0}(t_n,x_n)-w^{i_0}(t_n,y_n))-\sum_{k\in\Gamma^{-i_0}}\Theta_n^k(u^k(t_n,x_n)-w^k(t_n,y_n))\leq \rho_n
\end{equation}
where $\Theta_n^k$ is the increment rate of $f^{i_0}$ w.r.t. $y^k$, which is uniformly bounded w.r.t. $n$ and is non negative by the monotonicity assumption of $f^i$. Therefore \eqref{comp8} becomes
\begin{align*}
\lambda(u^{i_0}(t_n,x_n)-w^{i_0}(t_n,y_n))&\leq \sum_{k\in\Gamma^{-i_0}}\Theta_n^k(u^k(t_n,x_n)-w^k(t_n,y_n))+\rho_n\\
&\leq C_{f^{i_0}}\sum_{k\in\Gamma^{-i_0}}(u^k(t_n,x_n)-w^k(t_n,y_n))^++\rho_n.
\end{align*}
Then by taking $n\to\infty$ the inequality yields
\begin{align*}
\lambda(u^{i_0}(t^*,x^*)-w^{i_0}(t^*,x^*))&\leq \limsup_n C_{f^{i_0}}[\sum_{k\in\Gamma^{-i_0}}(u^k(t_n,x_n)-w^k(t_n,y_n))^+]\\
&\leq C_{f^{i_0}}[\sum_{k\in\Gamma^{-i_0}}(\limsup_n(u^k(t_n,x_n)-w^k(t_n,y_n)))^+ ]\\
&\leq C_{f^{i_0}}[\sum_{k\in\Gamma^{-i_0}}(u^k(t^*,x^*)-w^k(t^*,x^*))^+]
\end{align*}
Next as $i_0 \in\hat{\Gamma}(t^*,x^*)$, we deduce that
\[ \lambda(u^{i_0}(t^*,x^*)-w^{i_0}(t^*,x^*))\leq C_{f^{i_0}}(p-1)(u^{i_0}(t^*,x^*)-w^{i_0}(t^*,x^*)) \]
which is contradictory with the definiton of $\lambda$ given in \eqref{comps1}. As a 
consequence for any $\ig, u^i\leq w^i.$\\

\noindent\textit{\underline{Step 2}: the general case} \\

\nd For any arbitrary $\lambda\in\R$, let us define
\begin{align*}
\hat{u}^i(t,x):=e^{\lambda t}u^i(t,x)\mbox{ and }
\hat{w}^i(\tx):=e^{\lambda t}w^i(\tx).
\end{align*}
Note that $(\hat{u}^i)_{\ig}$  and  $(\hat{w}^i)_{\ig}$ is respectively the subsolution and the supersolution of the following system of PDEs: for any $\ig$ and $\txsp$,
\begin{align*}
&\min\lbrace v^i(\tx)-v^{i+1}(\tx)+e^{\lambda t}\underline{g}_{\iplus1}(\tx);\max[v^i(\tx)-v^{i+1}(\tx)-e^{\lambda t}\overline{g}_{\iplus1}(\tx);\\\\
&-\partial_t v^i(\tx)-\mathcal{L}^Xv^i(\tx)+\lambda v^i(\tx)-e^{\lambda t}f^i(\tx, (e^{-\lambda t}v^l(\tx))_{l\in\Gamma},e^{-\lambda t}\sigma^\top(\tx)D_xv^i(\tx)]\rbrace=0
\end{align*}
and $v^i(T,x)=e^{\lambda T}h_i(x)$. For $\lambda$ large enough, the condition \eqref{comps1} holds, then we go back to the result in Step 1 and we obtain, for any $\ig, \hat{u}^i\leq \hat{w}^i$, which also yields $u^i\leq w^i$. The proof of comparison is now complete. 
\end{proof}
\subsection{Existence and uniqueness of viscosity solution of \eqref{sysedpmilie1}}
Let us recall $(\bar{v}^i)_{\ig}$ and 
$(\bar{v}^{i,m})_{\ig}$ the functions defined in Proposition \ref{pena}. We firstly prove that $(\bar{v}^i)_{\ig}$ is a subsolution of \eqref{sysedpmilie1}, then we show that for a fixed $m_0$, $(\bar{v}^{i,m_0})_{\ig}$ is a supersolution of \eqref{sysedpmilie1}, finally by Perron's method we show that $(\bar{v}^i)_{\ig}$ is the unique solution of \eqref{sysedpmilie1}.
\begin{proposition}\label{barv}
The family $(\bar{v}^i)_{\ig}$ is a viscosity subsolution of \eqref{sysedpmilie1}.
\end{proposition}
\begin{proof}
We first recall that $\forall \ig, \bar{v}^i:=\lim_{m\to\infty}\bar{v}^{i,m},$ is usc function since the sequence $(\bar{v}^{i,m})_{m\geq 0}$ is decreasing and $(\bar{v}^{i,m})_{\ig}$ is continuous. Then thanks to the definition we have $\bar{v}^*=\bar{v}^i$, hence when $t=T$ we have $\bar{v}^i(T,x)=\lim_{m\to\infty}\bar{v}^{i,m}(T,x)=h^i(x).$\\
Next let us recall Definition \ref{visco}. For any $(\tx)\in[0,T)\times \R^k$, $\ig, (\underline{p},\underline{q},\underline{M})\in \bar J^+\bar{v}^i(\tx)$, we shall prove either
\begin{equation}\label{barv1}
\bar{v}^i(\tx)-L^i(\vec{\bar{v}})(\tx)\leq 0
\end{equation}
or
\begin{equation}\label{barv2}
\begin{array}{l}
\max[\bar{v}^i(\tx)-U^i(\vec{\bar{v}})(\tx);\\
\qquad-\underline{p}-b^\top(\tx)\underline{q}-\dfrac{1}{2}Tr(\sigma\sigma^\top)(\tx)\underline{M})-f^i(\tx, (\bar{v}^l(\tx))_{l\in\Gamma},\sigma^\top(\tx).\underline{q})]\leq 0.
\end{array}
\end{equation}
To proceed, we first assume that there exists $\epsilon_0>0$ such that
\[ \bar{v}^i(\tx)\geq \bar{v}^{i+1}(\tx)-\underline{g}_{\iplus1}(\tx)+\epsilon_0 \]
then we need to prove \eqref{barv2}.\\
As for any $\ig, (\bar{v}^{i,m})_{m\geq 0}$ decreasingly converges to $\bar{v}^i$, then there exists $m_0$ such that for any $m\geq m_0$ we have
\[ \bar{v}^{i,m}(\tx)\geq \bar{v}^{i+1,m}(\tx)-\underline{g}_{\iplus1}(\tx)+\dfrac{\epsilon_0}{2} \]
By the continuity of $(\bar{v}^{i,m})_{\ig}$ and $\underline{g}_{\iplus1}$, we can find a neighbourhood $O_m$ of $(\tx)$ such that
\begin{equation}\label{barv3}
\bar{v}^{i,m}(t',x')\geq \bar{v}^{i+1,m}(t',x')-\underline{g}_{\iplus1}(t',x')+\dfrac{\epsilon_0}{4},\,\fr (t',x')\in O_m.
\end{equation}
Next by Lemma 6.1 in \cite{CIL92} there exists a subsequence $(t_k,x_k)_{k\geq 0}$ such that
\[ (t_k,x_k)\to_{k\to \infty} (\tx)\; \text{and}\; \lim_{k\to\infty}\bar{v}^{i,k}(t_k,x_k)=\bar{v}^i(\tx). \]
In addition we can also find a sequence which we still denote by $(p_k,q_k,M_k)\in\bar{J}^+\bar{v}^{i,k}(t_k,x_k)$ such that
\[ \lim_{k\to\infty} (p_k,q_k,M_k)=(\underline{p},\underline{q},\underline{M}) \]
As the sequence $(t_k,x_k)$ can be chosen in the neighbourhood $O_k$, by applying the fact that $(\bar{v}^{i,k})_{\ig}$ is the unique viscosity solution of the following system: For any $\ig$, 
\begin{align}\label{barv4}
\begin{split}
&\min\lbrace \bar{v}^{i,m}(\tx)-L^i((\bar{v}^{l,m})_{l\in\Gamma})(\tx);\\
&-\partial_t\bar{v}^{i,m}(\tx)-b^\top(\tx)D_x\bar{v}^{i,m}(\tx)-f^{i,m}(\tx,(\bar{v}^{l,m}(\tx))_{l\in\Gamma},\sigma^\top(\tx)D_x\bar{v}^{i,m}(\tx))\rbrace=0\\&
\bar{v}^{i,m}(T,x)=h_i(x),
\end{split}
\end{align}we obtain
\begin{align}\label{barv5}
\begin{split}
-p_k-b^\top(t_k,x_k).q_k-\dfrac{1}{2}Tr(\sigma\sigma^\top(t_k,x_k)M_k)-f^{i,k}(\tx,(\bar{v}^{l,k}(t_k,x_k))_{l\in\Gamma},\sigma^\top(t_k,x_k)q_k)\leq 0
\end{split}
\end{align}
where $f^{i,k}(\tx,(v^l(\tx))_{l\in\Gamma},z):=f^i(\tx,(v^l(\tx))_{l\in\Gamma},z)-k(v^i(\tx)-U^i(\vec{v})(\tx))^+$.\\
Moreover as the sequence $(t_k,x_k,p_k,q_k,M_k)_k$ is bounded and $(\bar{v}^{i,m})_{\ig}$ is uniformly of polynomial growth, then we deduce from \eqref{barv5} that
\[ \epsilon_k:=( \bar{v}^{i,k}(t_k,x_k)-\bar{v}^{i+1,k}(t_k,x_k)-\bar{g}_{\iplus1}(t_k,x_k) )^+\to_{k\to\infty} 0 \]
However for any fixed $(\tx)$ and $k_0, (\bar{v}^{i,k}(t,x))_{k\geq k_0}$ is decreasing, then for $k\geq k_0,$
\begin{align*}
\bar{v}^{i,k}(t_k,x_k)&\leq \bar{v}^{i+1,k}(t_k,x_k)+\bar{g}_{\iplus1}(t_k,x_k)+\epsilon_k\\
&\leq \bar{v}^{i+1,k_0}(t_k,x_k)+\bar{g}_{\iplus1}(t_k,x_k)+\epsilon_k
\end{align*}
As $\bar{v}^{i,k_0}$ is continuous, by taking $k\to\infty$ we obtain that
\[ \lim_{k\to\infty}\bar{v}^{i,k}(t_k,x_k)=\bar{v}^i(\tx)\leq \bar{v}^{i+1,k_0}(\tx)+\bar{g}_{\iplus1}(\tx). \]
We then take $k_0\to\infty$ yielding
\begin{equation}\label{barv6}
\bar{v}^i(\tx)\leq \bar{v}^{i+1}(\tx)+\bar{g}_{\iplus1}(\tx).
\end{equation}

In the second place we consider a subsequence $(k_l)$ of $(k)$ such that for any $a\in\Gamma, (\bar{v}^{a,k_l}(t_{k_l},x_{k_l}))_l$ converges, then by taking $l\to\infty$ in \eqref{barv5} we obtain
\[ \lim_{l\to\infty}\lbrace  -p_{k_l}-b(t_{k_l},x_{k_l})q_{k_l}-\dfrac{1}{2}Tr(\sigma\sigma^\top(t_{k_l},x_{k_l})M_{k_l})-f^{i}(t_{k_l},x_{k_l},(\bar{v}^{a,k_l}(t_{k_l},x_{k_l}))_{a\in\Gamma},\sigma^\top(t_{k_l},x_{k_l}).q_{k_l}) \rbrace\leq 0 .\]
Then we deduce that
\begin{align}\label{barv7}
\begin{split}
&-\underline{p}-b^\top(\tx)\underline{q}-\dfrac{1}{2}Tr(\sigma\sigma^\top(\tx)\underline{M})\\
&\leq \lim_{l\to\infty} f^{i}(t_{k_l},x_{k_l},(\bar{v}^{a,k_l}(t_{k_l},x_{k_l}))_{a\in\Gamma},\sigma^\top(t_{k_l},x_{k_l})q_{k_l})\\
&=f^i(\tx,\lim_{l\to\infty}(\bar{v}^{a,k_l}(t_{k_l},x_{k_l}))_{a\in\Gamma},\sigma^\top(\tx)\underline{q})\\&\le 
f^i(\tx,(\bar{v}^{a}(t,x))_{a\in\Gamma},\sigma^\top(\tx)\underline{q}).
\end{split}
\end{align}
The last inequality holds true by the monotonicity assumption (H5) of $f^i$ and the fact that for any $a\in \G$, $\bar v^a$ verifies 
\[  \bar{v}^a(\tx)=\bar{v}^{*,a}(\tx)=\limsup_{(t',x')\to(\tx),m\to\infty} \bar{v}^{a,m}(t',x'),\; (\tx)\in\0TR \]
Thus for any $a\in\Gamma^{-i}$ we have
\[ \bar{v}^{a}(\tx)\geq \lim_{l\in\infty}\bar{v}^{a,k_l}(t_{k_l},x_{k_l})  \]
and
\[
\bar{v}^i(\tx)=\lim_{l\to\infty}\bar{v}^{i,k_l}(t_{k_l},x_{k_l}).\]
Thus \eqref{barv7} becomes
\begin{equation}\label{barv9}
-\underline{p}-b^\top(\tx)\underline{q}-\dfrac{1}{2}Tr(\sigma\sigma^\top(\tx)\underline{M})\leq f^i(\tx,(\bar{v}^a(\tx))_{a\in\Gamma},\sigma^\top(\tx).\underline{q}).
\end{equation}
Hence under \eqref{barv6} and \eqref{barv9}, \eqref{barv2} is satisfied, then $(\bar{v}^i)_{\ig}$ is a viscosity subsolution  of \eqref{sysedpmilie1}.
\end{proof}

\begin{proposition}\label{vsuper}
Let us fix $m_0\in\mathbb{N}$. Then the family $(\bar{v}^{i,m_0})_{\ig}$ is a viscosity supersolution of \eqref{sysedpmilie1}.
\end{proposition}
\begin{proof}
We first recall that the triple $(\bar{Y}^{i,m_0},\bar{Z}^{i,m_0},\bar{K}^{i,m_0,+})_{\ig}$ is the unique solution of the system of RBSDEs associated with $(f^{i,m_0},h^i,\underline{g}_{\iplus1})_{\ig}$ where
\[ f^{i,m_0}(s,X_s^{\tx},\vec{y},z):=f^i(s,X_s^{\tx},\vec{y},z)-m_0(y^i-y^{i+1}-\bar{g}_{\iplus1}(s,X_s^{\tx}))^+.\]
In addition there exist unique deterministic continuous functions with polynomial growth $(\bar{v}^{i,m_0})_{\ig}$ such that for any $\ig, s\in[t,T]$,
\[\bar{Y}_s^{i,m_0}=\bar{v}^{i,m_0}(s,X_s^{\tx})\,\, ((\tx)\in\0TR \mbox{ is fixed}).\]
Now let us define the following processes: $\forall \ig,$ $s\in [t,T]$, 
\begin{align*}
&\tilde{U}_s^{i,m_0}:=Y_s^{i,m} \vee(Y_s^{i+1,m_0}+\overline{g}_{\iplus1}(s,X_s^{\tx}))\\
&\bar{K}_s^{i,m_0,-}:=m_0\int_0^s (Y_s^{i,m_0}-Y_s^{i+1,m_0}-\bar{g}_{\iplus1}(s,X_s^{\tx}))^+ds.
\end{align*}
Then $(\bar{Y}^{i,m_0},\bar{Z}^{i,m_0},\bar{K}^{i,m_0,+},\bar{K}^{i,m_0,-})_{\ig}$ solves the following doubly reflected BSDEs: for any $\ig$, $s\in [t,T]$,
\begin{align*}
\left\lbrace
\begin{array}{l}
\bar{Y}_s^{i,m_0}=h^i(X_T^{\tx})+\int_s^T f^i(r,X_r^{\tx},(\bar{Y}_r^{l,m_0})_{l\in\Gamma},\bar{Z}_r^{i,m_0})dr-\int_s^T\bar{Z}_r^{i,m_0}dB_r\\
\qquad\qquad\qquad\qquad+\bar{K}_T^{i,m_0,+}-\bar{K}_s^{i,m_0,+}-(\bar{K}_T^{i,m_0,-}-\bar{K}_s^{i,m_0,-});\\\\
L_s^{i,m_0}\leq \bar{Y}_s^{i,m_0} \leq \tilde{U}_s^{i,m_0}\;\\\\
\int_t^T (\bar{Y}_s^{i,m_0}-L_s^{i,m_0})d\bar{K}_s^{i,m_0,+}=0 \;\text{and}\;\int_t^T (\bar{Y}_s^{i,m_0}-\tilde{U}_s^{i,m_0})d\bar{K}_s^{i,m_0,-}=0.
\end{array}
\right.
\end{align*}
Accordingly by the results of \cite{CK96} and \cite{shjpl}, $\bar{Y}^{i,m_0}$ is also associated with a zero-sum Dynkin game as follow: 
For any $s\in [t,T]$, 
\begin{align*}
\bar{Y}_s^{i,m_0}&=\esssup_{\sigma\geq s}\essinf_{\tau\geq s}\mathbb{E}[f_s^{\sigma\wedge\tau}f^i(r,X_r^{\tx},(\bar{Y}_r^{l,m_0})_{l\in\Gamma},\bar{Z}_r^{i,m_0})dr\\
&\qq\qq\qq+L_\sigma^{i,m_0}1_{(\sigma<\tau)}+\tilde{U}_\tau^{i,m_0}1_{(\tau\leq \sigma,\t<T)}+h^i(X_T^{\tx})1_{(\tau=\sigma=T)}|\mathcal{F}_s]
\end{align*}
Next following Theorem 3.7 and Theorem 6.2 in \cite{HH05}, $\bar{v}^{i,m_0}$ is the unique solution in viscosity sense of the following PDE with obstacle: 
\begin{align*}
\left\lbrace
\begin{array}{l}
\min\lbrace w(\tx)-L^i((\bar{v}^{l,m_0})_{l\in\Gamma})(\tx);\max[w(\tx)-\tilde{U}((\bar{v}^{l,m_0})_{l\in\Gamma})(\tx);\\
\qquad -\partial_tw(\tx)-b^\top(\tx)D_xw(\tx)-\dfrac{1}{2}Tr[(\sigma\sigma^\top)(\tx)D_{xx}^2w(\tx)]\\
\qquad-f^i(\tx,(\bar{v}^{l,m_0})_{l\in\Gamma},\sigma^\top(\tx)D_xw(\tx))]\rbrace=0;\\\\
w(T,x)=h^i(X_T^{\tx})
\end{array}
\right.
\end{align*}
where $\tilde{U}((\bar{v}^{l,m_0})_{l\in\Gamma})(\tx):=\bar{v}^{i,m_0}(\tx)\vee(\bar{v}^{i+1,m_0}+\bar{g}_{\iplus1})(\tx).$\\
In other words, for any $(\tx)\in\otr$ and for any $(p,q,M)\in\bar{J}^-(\bar{v}^{i,m_0})(\tx)$, it still holds that
\begin{align}
\label{ineglimite2}\bar{v}^{i,m_0}(\tx)\geq L^i((v^{l,m_0})_{l\in\Gamma})(\tx)\end{align}
and
\begin{align}\label{vsuper1}
\begin{split}
&\max[\bar{v}^{i,m_0}(\tx)-\tilde{U}^i((\bar{v}^{l,m_0})_{l\in\Gamma})(\tx);\\
&\qquad -p-b^\top(\tx).q-\dfrac{1}{2}Tr(\sigma\sigma^\top(\tx)M)-f^i(\tx,(\bar{v}^{l,m_0})_{l\in\Gamma},\sigma^\top(\tx)q) ]\geq 0.
\end{split}
\end{align}
Next apply the inequality $a-a\vee b\leq a-b$, then \eqref{vsuper1} yields 
\begin{align*}
&\max[\bar{v}^{i,m_0}(\tx)-(\bar{v}^{i+1,m_0}+\bar{g}_{\iplus1})(\tx);\\
&\qquad -p-b^\top(\tx).q-\dfrac{1}{2}Tr(\sigma\sigma^\top(\tx)M)-f^i(\tx,(\bar{v}^{l,m_0})_{l\in\Gamma},\sigma^\top(\tx)q) ]\geq 0
\end{align*}
Hence, with \eqref{ineglimite2}, this implies that $(\bar{v}^{i,m_0})_{\ig}$ is a viscosity supersolution of \eqref{sysedpmilie1}.
\end{proof}

We are now ready to use Perron's method to provide a solution for \eqref{sysedpmilie1}. So let us consider the following functions denoted by $(\sideset{^{m_0}}{^i}{\mathop{v}})_{\ig}$ and defined as: Let 
\[ \mathcal{U}^{m_0}:=\lbrace \vec{u}:=(u^i)_{i\in\Gamma} , \vec{u} \;\text{is a subsolution of \eqref{sysedpmilie1} and for any}\; \ig, \bar{v}^i\leq u^i\leq \bar{v}^{i,m_0}  \rbrace\]
Note that $\mathcal{U}^{m_0}$ is not empty since $(\bar{v}^i)_{\ig}\in  \mathcal{U}^{m_0}.$ Next for $\ig, (\tx)\in\0TR$ we set
\[\sideset{^{m_0}}{^i}{\mathop{v}}(\tx):=\sup\lbrace{u^i(\tx)},\; (u^i)_{\ig}\in\mathcal{U}^{m_0} \rbrace.\]
We then have:  
\begin{theorem}\label{uniqv}Assume [H2],[H3] and [H5]. Then the functions $(\sideset{^{m_0}}{^i}{\mathop{v}})_{\ig}$ is the unique viscosity solution of \eqref{sysedpmilie1}. Moreover the solution does not depend on $m_0$. Finally for any $\ig$, $\sideset{^{m_0}}{^i}{\mathop{v}}=\bar v^i$. 
\end{theorem}
\begin{proof}
It is obvious that for any ${\ig}$, the function $\sideset{^{m_0}}{^i}{\mathop{v}}$ belongs to class $\Pi_g$ since $(\bar{v}^i)_{\ig}$ and $(\bar{v}^{i,m_0})_{\ig}$ are functions of $\Pi_g$.\\
To proceed, we divide the main proof into three steps. On the other hand, to simplify the notation, we replace $(\sideset{^{m_0}}{^i}{\mathop{v}})_{\ig}$ with $(v^i)_{\ig}$ as there is no possible confusion.\\

\noindent\textit{Step 1: $(v^i)_{\ig}$ is a viscosity subsolution of \eqref{sysedpmilie1}.}\\
For any $\ig, v^i\in\mathcal{U}^{m_0}$ and then it satisfies:
\[\bar{v}^i\leq v^i\leq \bar{v}^{i,m_0}.\]
The inequalities still valid for the upper semicontinuous envelops, i.e., \[ \bar{v}^i\leq v^{i,*}\leq \bar{v}^{i,m_0}\]since $\bar v^i$ is usc and $\bar{v}^{i,m_0}$ is continuous. Therefore we have \[\bar{v}^i(T,x)= v^{i,*}(T,x)=\bar{v}^{i,m_0}(T,x)=h^i(x). \]
It means that $(v^{i,*})_{\ig}$ verify the subsolution property of system \eqref{sysedpmilieu} at time $T$. 
\ms

Next let $(\tilde{v}^k)_{\kg}$ be an arbitrary element of $\mathcal{U}^{m_0}$ and let $\ig$ be fixed. Since $(\tilde{v}^k)_{\kg}$ is a subsolution of \eqref{sysedpmilie1}, then for any $(\tx)\in\otr$ and $(p,q,M)\in\bar{J}^+\tilde{v}^{i,*}(\tx) $ we have
\begin{equation}\label{uniqv1}
\begin{array}{l}
\min\lbrace \tilde{v}^{i,*}(\tx)-L^i((\tilde{v}^{l,*})_{l\in\Gamma})(\tx);\max[ \tilde{v}^{i,*}(\tx)-U^i((\tilde{v}^{l,*})_{l\in\Gamma})(\tx);\\
-p-b^\top(\tx)q-\dfrac{1}{2}Tr(\sigma\sigma^\top(\tx)M)-f^i(\tx,(\tilde{v}^{l,*}(\tx))_{l\in\Gamma},\sigma^\top(\tx)q)]\rbrace\leq 0.
\end{array}
\end{equation}
But for any $\kg$, $ \tilde{v}^k\leq v^k$, then $\tilde{v}^{k,*}\leq v^{k,*}.$ On the other hand, we notice that the operators $(w^l)_{l\in\Gamma}\mapsto \tilde{v}^{i,*}-L^i((w^l)_{l\in\Gamma})$ and $(w^l)_{l\in\Gamma}\mapsto \tilde{v}^{i,*}-U^i((w^l)_{l\in\Gamma})$ are decreasing, then by the monotonicity of $f^i$ ([H5]) and \eqref{uniqv1} we have
\begin{equation}\label{uniqv2}
\begin{array}{l}
\min\lbrace (\tilde{v}^{i,*}-L^i((v^{l,*})_{\l\in\Gamma}))(\tx);\max[( \tilde{v}^{i,*}-U^i((v^{l,*})_{\l\in\Gamma}))(\tx);\\
-p-b^\top(\tx)q-\dfrac{1}{2}Tr(\sigma\sigma^\top(\tx)M)-f^i(\tx,[(v^{l,*}(\tx))_{l\in\Gamma^{-i}},\tilde{v}^{i,*}],\sigma^\top(\tx)q)]\rbrace\leq 0.
\end{array}
\end{equation}
It means that $\tilde{v}^i$ is a subsolution of the following PDE: \begin{align}\label{viscow}
\left\lbrace\begin{array}{l}
\min\lbrace (w-L^i((v^{l,*})_{\l\in\Gamma}))(\tx);\max[( w-U^i((v^{l,*})_{\l\in\Gamma}))(\tx);\\
\quad -p-b^\top(\tx)q-\dfrac{1}{2}Tr(\sigma\sigma^\top(\tx)M)-f^i(\tx,[(v^{l,*}(\tx))_{l\in\Gamma^{-i}},w],\sigma^\top(\tx)q)]\rbrace= 0\\
w(T,x)=h^i(x).
\end{array}
\right.
\end{align}
In addition, the following function is lsc:
\begin{align*}
&(\tx,w,p,q,M)\in[0,T]\times\mathbb{R}^{k+1+1+k}\times\mathbb{S}^k\\
&\qq\qq\longmapsto \min\lbrace w-L^i((v^{l,*})_{l\in\Gamma})(\tx);\max[w-U^i((v^{l,*})_{l\in\Gamma})(\tx);\\
&\quad\qq\qq\qq\qq-p-b^\top(\tx)q-f^i(\tx,[(v^{l,*}(\tx))_{l\in\Gamma^{-i}},w],\sigma^\top(\tx).q) ]\rbrace.
\end{align*}
As $v^i$ is the supremum of $\tilde{v}^i$, thanks to Lemma 4.2 in \cite{CIL92}, $v^i$ is a viscosity subsolution of \eqref{viscow}. But $i$ is arbitrary, then $(v^i)_{\ig}$ is a viscosity subsolution of system \eqref{sysedpmilie1}.\\

\noindent\textit{\underline{Step 2}: $(v^i)_{\ig}$ is a viscosity supersolution of \eqref{sysedpmilie1}.}\\

\nd We first focus on the terminal condition. For any $\ig,\; v_*^i(T,x)=h^i(x)$ from the inequality $\underline{v}^i=\underline v_*^i\leq \bar{v}^i_*\leq v^i_*\leq \bar{v}_*^{i,m_0}= \bar{v}^{i,m_0}$ since 
$\underline{v}^i$ is lsc and $\bar{v}^{i,m_0}$ is continuous. \\
\ms

Next by contradiction we assume that $(v^i)_{\ig}$ is not a supersolution of \eqref{sysedpmilie1}, i.e. there exists at least one $\ig$ and for some $(t_0,x_0)\in (0,T)\times \R^k$ and $(p,q,M)\in  J^-(v_*^i)(\tx)$ such that we have:
\begin{align}\label{uniqv3}
\begin{split}
&\min\lbrace v_*^i(t_0,x_0)-L^i((v_*^l)_{l\in\Gamma})(t_0,x_0);\max[ v_*^i(t_0,x_0)-U^i((v_*^l)_{l\in\Gamma})(t_0,x_0);\\
&\; -p-b^\top(t_0,x_0)q-\dfrac{1}{2}Tr(\sigma\sigma^\top(t_0,x_0)M)-f^i(t_0,x_0,(v_*^l(t_0,x_0))_{l\in\Gamma},\sigma^\top(t_0,x_0)q) ]\rbrace <0.
\end{split}
\end{align}
Next for any positive constants $\delta, \gamma$ and $r$ let us define:
\begin{align}\label{uniqv4}
\begin{split}
&u_{\delta,\gamma}(\tx):=v_*^i(t_0,x_0)+\delta+<q,x-x_0>+p(t-t_0)+\frac{1}{2}<(M-2\gamma)(x-x_0),(x-x_0)>\\
&\mbox{ and }B_r:=\lbrace (\tx)\in\0TR \;\text{such that}\; |t-t_0|+|x-x_0|<r\rbrace.
\end{split}
\end{align}
By choosing $\delta$ and $\gamma$
small enough, we deduce from \eqref{uniqv3} that
\begin{align}\label{uniqv5}
\begin{split}
&\min\lbrace v_*^i(t_0,x_0)-L^i((v_*^l)_{l\in\Gamma})(t_0,x_0)+\delta;\max[ v_*^i(t_0,x_0)-U^i((v_*^l)_{l\in\Gamma})(t_0,x_0)+\delta;\\
&\quad -p-b^\top(t_0,x_0)q-\dfrac{1}{2}Tr(\sigma\sigma^\top(t_0,x_0)(M-2\gamma))\\
&\quad -f^i(t_0,x_0,[(v_*^l(t_0,x_0))_{l\in\Gamma^{-i}},v_*^i(t_0,x_0)+\delta],\sigma^\top(t_0,x_0)q) ]\rbrace <0.
\end{split}
\end{align}
Next let us define the following function:
\begin{align*}
\Theta(\tx):=&\min\lbrace u_{\delta,\gamma}(\tx)-L^i((v_*^l)_{l\in\Gamma})(\tx);\max[u_{\delta,\gamma}(\tx)-U^i((v_*^l)_{l\in\Gamma})(\tx);\\
& \qquad-p -b^\top(\tx).q-\frac{1}{2}Tr(\sigma\sigma^\top(\tx))(M-2\gamma)\\
&\qquad -f^i(\tx,[(v_*^l(\tx))_{l\in\Gamma^{-i}},u_{\delta,\gamma}(\tx)],\sigma^\top(\tx)q)]\rbrace
\end{align*}
According to \eqref{uniqv5} we have $\Theta(t_0,x_0)<0.$ On the other hand, $\Theta$ is usc since the functions $v_*^i$, $\ig$,  are lsc, $u_{\delta,\gamma}$ is continuous and $f^i$ is continuous and verifies the monotonicity property. Therefore for any $\epsilon>0$, there is some $\eta>0$ such that for any $(\tx)\in B_\eta$ we have \[ \Theta(\tx)\leq \Theta(t_0,x_0)+\epsilon  \]
Next as $\Theta(t_0,x_0)<0,$ we can choose $\epsilon$ small enough to obtain $\Theta(\tx)\leq 0$ for any $(\tx)\in B_\eta$. Thus for any $(\tx)\in B_\eta$, $u_{\delta,\gamma}$ is nothing but a viscosity subsolution of the following PDE (on $B_\eta$):
\begin{align}\label{uniqv6}
\begin{split}
&\min\lbrace w(\tx)-L^i((v_*^l)_{l\in\Gamma})(\tx);\max[w(\tx)-U^i((v_*^l)_{l\in\Gamma})(\tx);\\
& \qquad-\partial_tw(\tx) -b^\top(\tx)D_xw(\tx)-\frac{1}{2}Tr(\sigma\sigma^\top(\tx)D_{xx}^2w(\tx))\\
&\qquad -f^i(\tx,[(v_*^l(\tx))_{l\in\Gamma^{-i}},w(\tx)],\sigma^\top(\tx)D_xw(\tx))]\rbrace=0.
\end{split}
\end{align}
As for any $\ig, v_*^i\leq v^{i,*}$, then $u_{\delta,\gamma}$ is also a viscosity subsolution of \eqref{uniqv6} by replacing $(v_*^i)_{\ig}$ with $(v^{i,*})_{\ig}$, i.e.
\begin{align*}
&\min\lbrace w(\tx)-L^i((v^{l,*})_{l\in\Gamma})(\tx);\max[w(\tx)-U^i((v^{l,*})_{l\in\Gamma})(\tx);\\
& \qquad-\partial_tw(\tx) -b^\top(\tx)D_xw(\tx)-\frac{1}{2}Tr(\sigma\sigma^\top(\tx)D_{xx}^2w(\tx))\\
&\qquad -f^i(\tx,[(v^{l,*}(\tx))_{l\in\Gamma^{-i}},w(\tx)],\sigma^\top(\tx)D_xw(\tx))]\rbrace=0.
\end{align*}
On the other hand since $(p,q,M)\in J^-(v_*^i(t_0,x_0))$, by the definition of the subjet (\cite{CIL92}) we have: $\forall \ig,$
\begin{align*}
v^i(\tx)&\geq v_*^i(\tx)\\
&\geq v_*^i(t_0,x_0)+p(t-t_0)+<q,x-x_0>+\frac{1}{2}<M(x-x_0),(x-x_0)>\\
&\qquad+o(|t-t_0|)+o(|x-x_0|^2).
\end{align*}
Next let us set $\delta=\frac{r^2}{8}\gamma$ and let us go back to the definition of $u_{\delta,\gamma}$ yielding
\begin{align*}
v^i(t,x)>u_{\delta,\gamma}(\tx)&=v_*^i(t_0,x_0)+\frac{r^2}{8}\gamma+<q,x-x_0>+p(t-t_0)+\frac{1}{2}<M(x-x_0),(x-x_0)>\\
&\quad -<\gamma(x-x_0),(x-x_0)>
\end{align*}
when $ \frac{r}{2}\le |x-x_0|\leq r$ and $r$ small enough. Next let us take $r\leq \eta$ and let us define the function $\tilde{u}^i$ by:
\begin{equation*}
\tilde{u}^i(t,x)=
\left\lbrace
\begin{array}{l}
\max(v^i(\tx),u_{\delta,\gamma}(\tx)),\quad \text{if}\; (\tx)\in B_r;\\
v^i(\tx)\mbox{ otherwise. }
\end{array}
\right.
\end{equation*}
Then according to \eqref{uniqv6} and Lemma 1.2 in \cite{CIL92}, $\tilde{u}^i$ is also a subsolution of the following PDE:
\begin{equation*}
\left\lbrace
\begin{array}{l}
\min\lbrace w(\tx)-L^i((v^{l,*})_{l\in\Gamma})(\tx);\max[w(\tx)-U^i((v^{l,*})_{l\in\Gamma})(\tx);\\\\
 \qquad-p -b^\top(\tx)q-\frac{1}{2}Tr(\sigma\sigma^\top(\tx)M)\\\\
\qquad -f^i(\tx,[(v^{l,*}(\tx))_{l\in\Gamma^{-i}},w(\tx)],\sigma^\top(\tx)q)]\rbrace=0\\\\
w(T,x)=h^i(x).
\end{array}
\right.
\end{equation*}
Once more by the monotonicity of $f^i$ and the fact that $\tilde{u}^i\geq v^i$, $[(v^l)_{l\in\Gamma^{-i}},\tilde{u}^i]$ is also a subsolution of \eqref{sysedpmilie1} which belongs to $\Pi_g$. Then by comparison we obtain that $[(v^l)_{l\in\Gamma^{-i}},\tilde{u}^i]$ belongs to $\mathcal{U}^{m_0}$. Next by the definition of $v_*^i$, we can find a sequence $(t_n,x_n,v^i(t_n,x_n))_{n\geq 1}$ which converges to $(t_0,x_0,v_*^i(t_0,x_0))$, then we have
\begin{align*}
\lim_{n\to\infty}(\tilde{u}^i-v^i)(t_n,x_n)&=\lim_{n\to\infty}(u_{\delta,\gamma}-v_*^i)(t_n,x_n)\\
&=v_*^i(t_0,x_0)+\delta-v_*^i(t_0,x_0)>0
\end{align*}
This result implies that we can find some points $(t_n,x_n)$ such that $\tilde{u}^i(t_n,x_n)>v^i(t_n,x_n)$, which is contradictory against the fact that $[(v^l)_{l\in\Gamma^{-i}},\tilde{u}^i]$ belongs to $\mathcal{U}^{m_0}$ and the definition of $(v^i)_{\ig}$.

\noindent\textit{Step 3: Continuity and uniqueness of $(v^i)_{\ig}$.}\\
Following the definition of usc envelop $(v^{i,*})_{\ig}$ (resp. lsc envelop $(v_*^i)_{\ig}$), $(v^{i,*})_{\ig}$ (resp.$(v_*^i)_{\ig}$) is a usc subsolution (resp. lsc supersolution) of \eqref{sysedpmilie1}, then by Proposition \ref{comp} we obtain $\forall \ig,$
\[v^{i,*}\leq v_*^i\]
Meanwhile it holds true that $v_*^i\leq v^i\leq v^{i,*}$ then $v_*^i=v^{i,*}$, which implies the continuity of $v^i$.

Next we assume that there exists another solution $(\hat{v}^i)_{\ig}$ of \eqref{sysedpmilie1} which belongs to class $\Pi_g.$ As $(v^i)_{\ig}$ and $(\hat{v}^i)_{\ig}$ are both subsolutions and supersolutions, by the comparison result we obtain both $v^i\leq \hat{v}^i$ and $v^i\geq \hat{v}^i$ with al $\ig$, as a result the solution is unique. The uniqueness of solution leads us directly to the fact that the solution $(v^i)_{\ig}$ does not depend on $m_0.$ Finally for any $\ig$ and $m_0$ we have 
$$
\bar v^i\leq v^i\le v^{i,m_0}.
$$Just send $m_0$ to $+\infty$ to obtain that for any $\ig$, $\bar v^i=v^i$. 
\end{proof}
%************************************************************************%
%**                           References                               **%
%************************************************************************%

\end{document}